\documentclass[reqno]{amsart}
\usepackage{geometry}
\usepackage{enumitem}

\usepackage{amsmath,amssymb,amsthm,amsfonts}
\usepackage{physics}
\allowdisplaybreaks[4]
\usepackage{mathrsfs}
\usepackage{appendix}
\usepackage{graphicx}
\usepackage{setspace}
\usepackage{mathtools}
\usepackage{latexsym}
\usepackage{mleftright}
\usepackage{soul}
\usepackage{cite}
\usepackage{color}
\definecolor{purple}{rgb}{0.7,0.2,0.9}

\usepackage{hyperref}

\newtheorem{lemma}{Lemma}[section]
\newtheorem{theorem}{Theorem}[section]
\newtheorem{definition}{Definition}[section]
\newtheorem{proposition}{Proposition}[section]
\newtheorem{remark}{Remark}[section]
\newtheorem{corollary}{Corollary}[section]
\numberwithin{equation}{section}
\newcommand\Om{\Omega}

\DeclareMathOperator{\Div}{div}
\DeclareMathOperator{\Exp}{exp}

\newcommand{\beq}{\begin{equation}}
\newcommand{\eeq}{\end{equation}}
\newcommand{\ben}{\begin{eqnarray}}
\newcommand{\een}{\end{eqnarray}}
\newcommand{\beno}{\begin{eqnarray*}}
\newcommand{\eeno}{\end{eqnarray*}}

\newcommand{\wc}{\rightharpoonup}
\newcommand{\wsc}{\rightharpoonup}
\newcommand{\w}{\quad\text{weakly in }}
\newcommand{\ws}{\quad\text{weakly-* in }}
\newcommand{\s}{\quad\text{strongly in }}
\begin{document}
\title[ ]{Well-posedness and Blow-up Criterion for \\ Strong Solutions of the compressible Navier-Stokes/Allen-Cahn System with Vacuum}
\thanks{$^*$Corresponding author.}
\author[Yinghua Li]{Yinghua Li}
\address[Y. Li]{School of Mathematical Sciences, South China Normal University,
Guangzhou, 510631, China}
\email{yinghua@scnu.edu.cn}
\author[Haoran Zheng]{Haoran Zheng}
\address[H. Zheng]{School of Mathematical Sciences, South China Normal University,
Guangzhou, 510631, China}
\email{hrzheng@scnu.edu.cn}
\author[Junquan Zhou]{Junquan Zhou$^*$}
\address[J. Zhou]{School of Mathematical Sciences, South China Normal University,
Guangzhou, 510631, China}
\email{zhoujunquan@m.scnu.edu.cn}
\date{\today}

\begin{abstract}
This paper is devoted to the study of strong solutions for the compressible Navier-Stokes/Allen-Cahn system in bounded domain $\Omega\subset\mathbb R^3$, allowing for the presence of initial vacuum. A characteristic of this system is the strong coupling between density and the Allen-Cahn equation, which leads to strong degeneracy in vacuum regions. Under a compatibility condition on the initial phase-field variable, we establish the local existence and uniqueness of strong solutions for $0\le\rho_0\in W^{1,q}$ with $q\in(3,6)$, $u_0\in H_0^1$ and $\chi_0\in H^2$. Owing to time-weighted estimates, no compatibility condition is required for the velocity, but these estimates introduce a singularity in proving uniqueness. We then establish a criterion for the possible breakdown of such a local strong solution at finite time in terms of blow-up of the quantities $\|\nabla u\|_{L_t^{1} L_x^{\infty}}$,
	$\|u\|_{L_t^{2} L_x^{\infty}}$ and $\|\nabla \chi\|_{L_t^{2} L_x^{\infty}}$.
\end{abstract}

\maketitle



\vspace{-2mm}

\section{Introduction}
The dynamic behavior of a binary mixture of two immiscible fluids subject to phase separation is usually described by diffuse interface models. These models consist of the Navier-Stokes equations governing the evolution of the mixture and a phase-field system that introduces diffuse effects, such as the Cahn-Hilliard, Allen-Cahn, or other types of phase-field equations.
This paper is concerned with the compressible Navier-Stokes/Allen-Cahn (NSAC) system proposed by Blesgen \cite{BLesgen99}
\begin{align*}
\begin{cases}
\partial_t\rho+{\rm div}(\rho u)=0,
\\
\partial_t(\rho u)+{\rm div}(\rho u\otimes u)={\rm div}{\mathbb T},
\\
\partial_t(\rho\chi)+{\rm div}(\rho\chi u)=-\mu,
\\
\rho\mu=-\Delta\chi+\rho\dfrac{\partial \tilde F(\rho, \chi)}{\partial\chi},
\end{cases}
\end{align*}
where $(x, t)\in\Omega\times(0, T)$, $\Omega\subset\mathbb R^3$ is a bounded domain with smooth boundary $\partial\Omega$ and $T>0$ is a given time.
$\rho=\rho(x,t)$ denotes the total density, $u=u(x,t)$ is the mean velocity,
phase-field variable $\chi=\chi(x,t)$ represents the concentration difference of the two fluids,
$\mu=\mu(x,t)$ is the chemical potential, the Cauchy stress-tensor is given by
\begin{align*}
{\mathbb T}=&{\mathbb S}-\left(\nabla\chi\otimes\nabla\chi-\frac{|\nabla\chi|^2}{2}{\mathbb I}\right)-p(\rho, \chi){\mathbb I},
\\
{\mathbb S}=&\nu\left(\nabla u+\nabla u^T-\frac23{\rm div}u{\mathbb I}\right)+\eta\,{\rm div}u\,{\mathbb I},
\end{align*}
where $\nu>0$, $\eta\ge0$ are viscosity coefficients,
and $p(\rho,\chi)=\rho^2\dfrac{\partial \tilde F(\rho, \chi)}{\partial\rho}$ denotes the pressure.
In this paper we take the specific free energy $\tilde F$ as follows
$$
\tilde F(\rho,\chi)=\frac{A\rho^{\gamma-1}}{\gamma-1}+\frac{(\chi^2-1)^2}{4}.
$$
For simplicity, we use the Lam\'{e} operator instead of ${\rm div}\,{\mathbb S}$, i.e.
$$
\mathcal{L}u=\nu\Delta u + (\nu+\lambda)\nabla (\Div u),
$$
where $\nu>0$ and $\lambda=\eta-\frac23\nu$ satisfy $2\nu+3\lambda \geq 0$.
Then the compressible NSAC system can be rewritten as follows
\begin{equation}\label{NSAC}
	\begin{cases}
		\rho_t + \Div (\rho u)=0,\\
		\rho u_t + \rho (u\cdot\nabla) u + \nabla (P(\rho))
		= \mathcal{L}u - \Div( \nabla\chi\otimes\nabla\chi
		- \frac{|\nabla \chi|^2}{2} \mathbb{I}), \\
		\rho \chi_t + \rho ( u \cdot \nabla ) \chi = - \mu, \\
		\rho \mu = - \Delta\chi + \rho f(\chi),
	\end{cases}
\end{equation}
where $P(\rho)=A\rho^\gamma$ with $A>0$ and $\gamma>1$, the free energy density $f(s)=F'(s)$, and $F(s)=\frac{1}{4}(s^2-1)^2$ for $s\in\mathbb R$
is the Landau type potential.

\vskip1mm
There have been many theoretical results concerning the 1D compressible NSAC system.
When the initial density is away from vacuum, Ding-Li-Luo \cite{DLL13} proved the existence and uniqueness of global classical solutions, the existence of weak solutions, and the uniqueness of strong solutions for the initial boundary value problem.
Later, Zhang \cite{Z21} improved the regularity to $H^i (i=2,4)$.
Then Ding-Li-Tang \cite{DLT19} considered the free boundary problem and obtained the existence and uniqueness of strong solutions.
In the case where the viscosity coefficient depends on the density,
Ding-Li-Wang \cite{DLW24} obtained the $H^{2m}$-solution $(m\in\mathbb N)$ of the free boundary problem.
Moreover, the existence of time-periodic solutions have been obtained by Song-Zhang-Wang \cite{SZW20}.
For the Cauchy problem, Chen-Huang-Shi \cite{CHS24} investigated the global well-posedness of the solutions and their sharp interface limit,
while Luo-Yin-Zhu \cite{LYZ18, LYZ20} proved the stability of the rarefaction wave and the composite wave.
For the non-isentropic NSAC system without vacuum, we refer the readers to \cite{YDL22, CHHS21, CHHS25} and the references therein.
When vacuum is allowed, Chen-Guo \cite{CG17} established the existence and uniqueness of strong solutions for the initial boundary value problem under the compatibility condition
$$
\nu u_{0xx}-(A\rho_0^\gamma)_x-\frac12(\chi_{0x}^2)_x=\rho_0 f(x),  \quad f\in L^2,
$$
and the existence of unique classical solution under the compatibility conditions
\begin{align*}
\begin{cases}
\nu u_{0xx}-(A\rho_0^\gamma)_x-\frac12(\chi_{0x}^2)_x=\rho_0 g(x), ~ & g\in H^1,
\\
\mu_0=\rho_0h(x),  & h_{xx}\in L^2
\end{cases}
\end{align*}
for $x\in(0,1)$. Chen-Zhu \cite{CZ21} assumed that the viscosity coefficient depends both on the density and the phase variable, and they established the local existence of strong solutions for the initial boundary value problem along with the corresponding blow-up criteria, imposing the compatibility condition
$$
[\nu(\rho_0,\chi_0)u_{0x}]_x-(A\rho_0^\gamma)_x-\frac12(\chi_{0x}^2)_x=\rho_0f(x),  \quad f\in L^2
$$
for $x\in(0,1)$. When the viscosity coefficient depends only on the density, local strong solutions can be improved to classical solutions under more complicated compatibility conditions.
Meanwhile, Su \cite{S21} established the existence of unique strong solutions using time-weighted estimates without any compatibility condition,
and these strong solutions become classical if additional compatibility conditions are imposed.
It should be pointed out that all results concerning initial vacuum assume the phase-field variable satisfies a Dirichlet boundary condition, i.e. $\chi\big|_{x=0,1}=0$.

\vskip1mm
For the 3D compressible NSAC system, Feireisl et al. \cite{FPRS10} first established the global existence of weak solutions to the initial boundary boundary problem under the assumption $\gamma>6$ and without initial vacuum. Later, Chen-Wen-Zhu \cite{CWZ19} relaxed the range to $\gamma>2$, allowing for initial vacuum, and obtained the existence of weak solutions for system with Landau potential.
Li-Xie \cite{LX25} proved that if the Mach number and initial data are small and away from vacuum,
global strong spatial-periodic solutions exist.
Under small initial perturbations away from vacuum, Zhao \cite{Zhao22} proved the global existence of classical solutions to the 3D Cauchy problem.
For the strong solutions to the initial boundary value problem of the NSAC system, only Kotschote \cite{Kotschote12} derived the local well-posedness result for the non-isentropic case.
Whether the strong solutions to the initial boundary value problem of the multi-dimensional compressible NSAC system exist is still an open problem.

\vskip1mm
Noticing that if the phase variable is constant, then the NSAC system reduces to the compressible Navier-Stokes equations, for which the mathematical theory is considerably more developed.
For the compressible Navier-Stokes equations with vacuum,  the first result was given by Salvi and Stra\v{s}kraba \cite{SS93},
where the compatibility condition on the initial data was introduced to guarantee the uniqueness of local strong solutions,
and this analytical framework was subsequently refined and extended in \cite{CK03,CCK04,CK06}.
Later, Huang-Li-Xin \cite{HLX12} established a global-in-time existence result under a small initial energy assumption.
In 1D, Jiu-Li-Ye \cite{JLY14} proved the global existence of classical solutions for arbitrarily large initial data with vacuum.
Recently, the compatibility condition above was removed in \cite{Li17,LZ23}, where the local existence and uniqueness of strong solutions were established.
While extensive literature exists on this topic, we do not enumerate all contributions here.
Comparing with the compressible Navier-Stokes equations, results on the multi-dimensional compressible NSAC system are relatively rare.
In this paper, we establish the local existence and uniqueness of strong solutions allowing for initial vacuum, with a compatibility condition
imposed on the phase-field variable but none required for the velocity.

\vskip1mm
When deriving the existence of local strong solutions, it is natural to consider global well-posedness.
However, as is known, the existence of global strong solutions remains an open problem even for the Navier-Stokes equations,
so an alternative approach is to consider either smallness restrictions or blow-up criteria, and
in this paper, we adopt the latter.
As we mentioned before, there is no local existence result for strong solutions to the compressible multi-dimensional NSAC system.
But some results have been established for the nonhomogeneous incompressible case, including both the NSAC and Navier-Stokes/Cahn-Hilliard (NSCH) system.
The local well-posedness without vacuum was established in \cite{LH18}.
Subsequently, Li-Ding-Huang \cite{LDH16} derived the blow-up criterion for the 3D compressible NSAC system: for the maximum existence time $T^*$, it holds that
$$
\lim_{T\to T^*}\left(\|Du\|_{L^1(0,T; L^\infty)}+\|u\|_{L^2(0,T; L^\infty)}
+\|\nabla\chi\|_{L^2(0,T; L^\infty)}\right)=\infty.
$$
Then Zhang \cite{Z16} and Fan-Li \cite{FL19} improved the criterion by relaxing the restriction $\|\nabla\chi\|_{L^2(0,T; L^\infty)}$ to
$\|\nabla\chi\|_{L^2(0,T; BMO)}$ and removing it, respectively.
Recently, Fang-Nei-Guo \cite{FNG25} gave a blow-up criterion for local strong solutions to the initial boundary value problem of the NSCH system in the case of initial density away from zero.
Moreover, there have been global existence and uniqueness results for both nonhomogeneous NSAC and nonhomogeneous NSCH systems, such as \cite{LY25, GT20, JWX25}.
For NSCH systems with a non-Newtonian term and initial vacuum, Fang-Duan-Guo \cite{FDG25} proved the global existence of weak solutions for the
initial-boundary value problem and the global existence of spatial-periodic strong solutions.
In this paper, we establish a blow-up criterion for our local strong solutions analogous to that in \cite{LDH16}.

\vskip1mm
This paper aims to establish the mathematical theory of compressible NSAC systems with vacuum in 3D.
Throughout this paper, we consider the following initial value conditions:
\begin{align}
	\label{I}
	(\rho,\,  \rho u , \, \chi)\Big|_{t=0}=(\rho_0,\,\rho_0 u_0, \, \chi_0) \quad \text{in }\Om,
\end{align}
and boundary value conditions
\begin{align}
	\label{B1}
	(u, \partial_{\boldsymbol{n}}\chi)\Big|_{\partial\Om} = (0, 0)  \quad  \text{on }\partial\Om,
\end{align}
where $\boldsymbol{n}$ is the unit outward normal vector of $\partial\Om$. We emphasize that the initial density $\rho_0$ is allowed to vanish on an open subset of $\Omega$; in other words, the presence of an initial vacuum is permitted. Throughout this paper, we always assume that $\int_\Om \rho_0 \,dx = 1$,
thus $\eqref{NSAC}_1$ implies that
\begin{equation*}
	\int_{\Om} \rho dx = \int_{\Om} \rho_0 dx = 1.
\end{equation*}

Before stating our main result, we first explain the notations used throughout this paper.
For $1 \leq p \leq \infty$ and integer $k \geq 0$, the standard Sobolev spaces are denoted by
$$
L^p=L^p(\Omega),~~ W^{k, p}=W^{k, p}(\Omega), ~~H^k=H^{k, 2}(\Omega),~~ H_0^1=\left\{u \in H^1; \left.u\right|_{\partial \Omega}=0\right\}.
$$
Now we define precisely what we mean by strong solutions to the system \eqref{NSAC}--$\eqref{B1}$.
\begin{definition}[Strong solutions]\label{def1}
	For $T>0$, $(\rho, u, \chi, \mu)$ is called a strong solution to the system
	$\eqref{NSAC}$--$\eqref{B1}$ in $\Om \times(0,T)$ if for some $q\in(3,6)$, it has the regularities
	\begin{equation*}
		\begin{cases}
			0 \le \rho \in L^{\infty}(0, T; W^{1, q}) \cap C([0, T]; C(\overline{\Om})),
			\quad \rho_t \in L^\infty(0, T; L^2), \\
			u \in L^\infty(0, T; H_0^1) \cap L^2(0, T; H^2) \cap L^1(0, T; W^{2,q}),
			\quad \sqrt{\rho}u \in C([0, T]; L^2), \\
			\sqrt{t} \nabla^2 u \in L^{\infty}(0, T; L^2)
			\cap  L^{2}(0, T; L^q), \\
			\sqrt{\rho} u_t\in L^2(0, T; L^2), \quad \sqrt{t} u_t \in L^2(0, T; H_0^1), \\
			\chi \in C([0, T]; H^1) \cap L^{\infty}(0, T; H^2) \cap L^2(0, T; H^{3}),
			\quad \sqrt{t} \nabla^3 \chi \in L^{\infty}(0, T; L^2), \\
			\chi_t \in L^2(0, T; H^1), \quad \rho \chi_t \in L^\infty(0, T; L^2),
			\quad \sqrt{t} \nabla\chi_t \in L^\infty(0, T; L^2), \\
			\mu \in L^{\infty}(0, T; L^{2}).
		\end{cases}
	\end{equation*}
	Besides, $(\rho, u, \chi, \mu)$ satisfies \eqref{NSAC} almost everywhere
	in $\Om \times(0,T)$ and fulfills the initial condition \eqref{I}.
	The boundary condition $\partial_{\boldsymbol{n}} \chi = 0$ holds almost
	everywhere on $\partial{\Om}\times(0, T)$.
\end{definition}

Our first result is the local existence and uniqueness of strong solutions.
\begin{theorem}\label{th1}
	Let $\Om$ be a smooth bounded domain in ${\mathbb R}^3$  and the initial data $(\rho_0, u_0, \chi_0)$ satisfies  $0 \le \rho_0 \in W^{1,q}$, with $q \in (3, 6)$,  $u_0 \in H_0^1 $ and $\chi_0 \in H^2$. If, in addition, the following compatibility conditions hold:
\begin{equation}\label{com1}
	\left\{
		\begin{aligned}
			\Delta \chi_0 &= \rho_0 {h},
			&&\text{for some }{h} \in L^2,  \\
		\partial_{\boldsymbol{n}} \chi_0
			&= 0,
			&& \text{on } \partial\Om,
		\end{aligned}
	\right.
\end{equation}	
	then, there exists a positive time $T_0>0$, depending only on $A$, $\gamma$, $\nu$,
	$\lambda$, $q$, $\Om$ and $\Phi_0$, such that the system \eqref{NSAC}, subject to \eqref{I}--\eqref{B1},
	admits a unique strong solution  $(\rho, u, \chi, \mu)$ in the sense of Definition~\ref{def1}, where
	\begin{align*}
		\Phi_0=\|\rho_0\|_{W^{1,q}} + \|\nabla u_0\|_{L^{2}} + \| \chi_0 \|_{H^2} + \| h \|_{L^2}.
	\end{align*}
\end{theorem}
\begin{remark}
The main difficulty of this problem arises from the strong coupling between density and the Allen-Cahn equation,
which leads to strong degeneracy within the vacuum region.
A compatibility condition is imposed on the initial phase-field variable $\chi_0$ (see \eqref{com1}).
This is necessary for two reasons. Mathematically, the equation for $\chi$ involves $\rho^2\chi_t$,
the stronger degeneracy in vacuum prevents the time-weighted estimates used for the velocity from controlling the initial singularity of $\chi_t$.
The compatibility condition ensures that the initial chemical potential $\mu_0$ is well defined in $L^2$.
Physically, in the vacuum region where $\rho_0=0$, there is no matter and hence no driving force for phase separation.
This forces $\Delta\chi_0=0$ and thus $\mu_0=0$, which is exactly encoded in the compatibility condition.
\end{remark}
\begin{remark}
The compatibility condition \eqref{com1} is also necessary.
Let $(\rho, u, \chi, \mu)$ be a strong solution as in Definition~\ref{def1}. From the regularity,
$\mu \wc g$ in $L^2$ as $t\to 0^+$ for some $g\in L^2$. Using $\eqref{NSAC}_4$ and passing to the limit, we obtain
\begin{equation*}
	\Delta \chi_0 = \rho_0 (f(\chi_0) - g) \quad \text{a.e. in } \Om.
\end{equation*}
Hence $\Delta\chi_0=\rho_0 h$ with $h=f(\chi_0) - g \in L^2$.
Moreover, from the regularity we can deduce $\chi\in C([0,T]; H^s)$ for some $3/2<s<2$. 
Then by the continuity of the trace operator, the boundary condition $\partial_{\boldsymbol{n}}\chi=0$ on $\partial\Omega$ for $t>0$ implies $\partial_{\boldsymbol{n}}\chi_0=0$. Thus \eqref{com1} is necessary.
\end{remark}
\begin{remark}
Under the assumptions of Theorem~\ref{th1}, if the compatibility conditions
	\begin{equation*}
	\begin{cases}
	\mathcal{L} u_0 - \nabla P(\rho_0) - \Delta \chi_0 \nabla \chi_0 = \sqrt\rho_0 {g}, & {g} \in L^2, \\
	\mu_0 = \rho_0 {h}, & {h} \in H^1
	\end{cases}
	\end{equation*}
are imposed and the initial data are more regular, namely $u_0 \in H_0^1 \cap H^2$ and $\chi_0 \in H^3$, then a local strong solution $(\rho, u, \chi, \mu)$ can be obtained in the sense of Definition \ref{def1}, with all time-weighted norms in Definition~\ref{def1} replaced by their
	corresponding non-weighted counterparts.
The corresponding a priori estimates are analogous to those in Section 2, and the proof, with only minor modifications, follows arguments similar to those in Theorem 1.2 of \cite{HWW12}.
\end{remark}

Our second result is a criterion for the possible breakdown of such a local strong solution at finite time.
To the best of our knowledge, this is the first blow-up criterion for the compressible NSAC system in 3D bounded domain with vacuum.
\begin{theorem}\label{th2}
	Let $(\rho, u, \chi, \mu)$ be the strong solution of \eqref{NSAC} with the initial boundary conditions \eqref{I} and \eqref{B1}. Assume that the initial data $(\rho_0, u_0, \chi_0)$ have the same regularity as in Theorem \ref{th1} and satisfy \eqref{com1}. If $0 < T^* < \infty $ is the maximum time of existence, then
	\begin{equation}\label{Blowup_Criteria}
\lim\limits_{T\to T^*} \left(\|\nabla u\|_{L^1(0,T; L^\infty)}+\| u\|_{L^2(0,T; L^\infty)}+\|\nabla \chi\|_{L^2(0,T; L^\infty)}\right) = \infty.
	\end{equation}
\end{theorem}


\vskip1mm
The main contributions of this paper are as follows.

\vskip1mm
\noindent
$\bullet$\quad {\it Existence of strong solutions with vacuum for the 3D NSAC system.} To the best of our knowledge, this is the first result establishing the local existence of strong solutions to the initial-boundary value problem of the 3D compressible NSAC system allowing initial vacuum. The proof relies on time-weighted estimates that remove the compatibility condition on the velocity, while a mild compatibility condition on the phase-field variable $\chi_0$ is retained due to the stronger degeneracy $\rho^2\chi_t$.
Moreover, in contrast to existing 1D results that impose Dirichlet conditions, we handle the physically relevant Neumann condition $\partial_{\boldsymbol{n}}\chi\big|_{\partial\Omega}=0$.

\vskip1mm
\noindent
$\bullet$\quad {\it Uniqueness without Lagrangian coordinates.} Unlike the heat-conductive compressible Navier-Stokes equations where the initial data only prescribe $\rho_0\theta_0$ and strong thermal nonlinearities force a Lagrangian approach \cite{LZ23}, our NSAC system prescribes $\chi_0$ directly with sufficient regularity and involves no thermal effect. Consequently, we are able to prove uniqueness entirely in Eulerian coordinates using singular time-weighted estimates. This simplifies the argument and highlights the structural difference introduced by the Allen-Cahn coupling.

\vskip1mm
	The remainder of this paper is organized as follows.
	In Section~\ref{section priori}, we derive the a priori estimates, which play a fundamental role in the subsequent analysis.
	Section~\ref{Galerkin scheme} is devoted to the construction of approximate solutions via a Galerkin scheme.
	In Section~\ref{Proof of Theorem 1.1}, we present the proof of Theorem~\ref{th1}.
	Finally, Section~\ref{Proof of Theorem 1.2} is devoted to the proof of Theorem~\ref{th2}.
    Moreover, one can find the proof of Proposition~\ref{appro_sol} in Appendix.

\vskip6mm
\section{A priori estimates}\label{section priori}
Our main purpose of this section is to derive some a priori estimates for the strong
or smooth solutions {$(\rho, u, \chi, \mu)$} to the system \eqref{NSAC}, associated with the
initial conditions \eqref{I} and the boundary conditions \eqref{B1}, provided that the
initial density function has a positive lower bound, $\rho_0\geq\delta>0$.
All these a priori estimates we will obtain are independent of $\delta>0$.

Throughout the paper, we denote by $C$ the generic positive constants that depending
only on $A$, $\gamma$, $\nu$, $\lambda$, $q$, ${\Om}$ and $\Phi_0$, but independent of $\delta>0$.
With out loss of generality, we always assume $0 < T \le 1$.
For each $t \in (0, T)$ and $q \in (3,6)$, denote
\begin{equation*}
	\begin{aligned}
		\varPhi(t) &:= \sup_{0 \le s \le t} \left(\|\rho\|_{W^{1, q}}
        + \|\nabla u\|^2_{L^2}+ \| \sqrt{s} \sqrt \rho  u_t \|^2_{L^2}
		+ \|\chi\|^2_{H^2} + \|\rho \chi_t\|^2_{L^2}
		+ \| \sqrt{s} \nabla\chi_t \|^2_{L^2} \right)
        \\
		&  + \int_0^t {\left( \|\nabla^2 u\|^2_{L^2}+ \| \sqrt{\rho} u_t \|^2_{L^2}
        + \|\sqrt{s} \nabla u_t \|^2_{L^2}
        + \| \nabla^3 \chi \|^2_{L^2}+\|\nabla\chi_t\|^2_{L^2}
		 \right)} ds + 1.
	\end{aligned}
\end{equation*}

The following lemma in \cite[Lemma 3.2]{F04} provides a weighted Poincar\'{e}-type
inequality that will be frequently used throughout the paper.
\begin{lemma}\label{lem1}
	If $\varphi \in H^1(\Om)$, $\Om \subset\subset \mathbb{R}^3$, and
	$\int_{\Om} \rho^{\gamma} \,dx \le E_0 $  for some positive number $E_0$, then
	\begin{equation*} 
		\| \varphi\|_{L^2}^2  \le C(E_0) \left( \| \nabla \varphi \|_{L^2}^2
		+ \left(\int_{\Om} \rho |\varphi| dx\right)^2 \right),
	\end{equation*}
	and
	\begin{equation*} 
		\| \varphi\|_{L^2}^2  \le C(\Om,E_0)\left(\| \nabla \varphi \|_{L^2}^2
		+ \|\rho^{\alpha} \varphi\|_{L^2}^2  \right), \quad\alpha=\frac{1}{2}, 1.
	\end{equation*}
\end{lemma}


We begin with the standard energy estimates.
\begin{lemma}[Energy Equality]\label{lem2}
	Under the conditions of Theorem \ref{th1}, it holds for any $t \in (0,T)$ that
	\begin{equation}\label{E0}
		\begin{split}
			&\frac{1}{2}\|\sqrt{\rho} u \|_{L^2}^2 + \frac{1}{2}\|\nabla \chi\|_{L^2}^2
			+ \frac{A}{\gamma-1} \int_{\Om} \rho^{\gamma} \,dx + \int_{\Om} \rho F(\chi) \,dx \\
			&\quad + \int_0^t ( \|\mu\|_{L^2}^2 + \nu\|\nabla u\|_{L^2}^2
			+ ( \lambda + \nu ) \|\Div u\|_{L^2}^2 ) \,ds = E_0,
		\end{split}
	\end{equation}
	where
	$$
	E_0 = \frac{1}{2} \| \sqrt{\rho_0} u_0 \|_{L^2}^2
	+ \frac{A}{\gamma-1} \int_{\Om} \rho_0^{\gamma} \,dx
	+ \frac{1}{2} \| \nabla\chi_0 \|_{L^2}^2 + \int_{\Om} \rho_0 F(\chi_0) \,dx.
	$$
\end{lemma}

\begin{proof}
	Multiplying  $\eqref{NSAC}_2$ by $u$ and $\eqref{NSAC}_3$ by $\mu$,
integrating over $\Omega\times(0,t)$, after integration by parts and utilizing the boundary conditions $\eqref{B1}$,
it is easy to arrive at \eqref{E0}.
\end{proof}

The following result is a direct consequence of Lemma \ref{lem1} and Lemma \ref{lem2}.
\begin{corollary}\label{cor1}
	Let {$(\rho, u, \chi, \mu)$} be a smooth solution to the system \eqref{NSAC},
	associated with the initial boundary conditions \eqref{I}-\eqref{B1}. It holds that
	\begin{align*}
		\sup_{0 \le t  \le T} \left( \|\chi\|_{H^1}
		+ \|f(\chi)\|_{L^2} + \|f'(\chi)\|_{L^3} \right) \le C .
	\end{align*}
\end{corollary}

{Before establishing the a priori estimates, we present two preparatory lemmas. }
\begin{lemma}\label{lem3}
	Let {$(\rho, u, \chi, \mu)$} be a smooth solution to the system \eqref{NSAC},
	associated with the initial boundary conditions \eqref{I}-\eqref{B1}. It holds that
	\begin{align*}
		\int_{0}^{T} \big( \|\nabla u\|_{L^\infty} + \|\nabla^2 u\|_{L^q} \big) \, dt
		\le C T^{\frac{6-q}{4q}} \varPhi^{\gamma + 1}(T).
	\end{align*}
\end{lemma}

\begin{proof}
	By the $W^{2,q}$-estimate of Lam\'{e} equation (cf.~\cite{ADN64, CCK04}), one has
	\begin{align*}
		\|\nabla^{2}u\|_{L^q}& \le C ( \|\rho u_{t}\|_{L^q}
		+ \| \rho ( u \cdot \nabla ) u \|_{L^q}
		+ \| \nabla (P(\rho))\|_{L^q} + \| \Delta \chi \nabla \chi \|_{L^q} ).
	\end{align*}
	It follows from the H\"{o}lder, Sobolev, Poincar\'e, Gagliardo-Nirenberg inequalities that
	\begin{align*}
		&\|\rho u_{t}\|_{L^q} \le C \|\rho u_{t}\|_{L^2}^{\frac{6-q}{2q}}
		\|\rho u_{t}\|_{L^6}^{\frac{3q-6}{2q}}
		\le C \|\rho\|_{L^\infty}^{\frac{5q-6}{4q}} \| \sqrt\rho u_{t}\|_{L^2}^{\frac{6-q}{2q}}
		\| \nabla u_{t}\|_{L^2}^{\frac{3q-6}{2q}}, \\
		&\| \rho ( u \cdot \nabla ) u \|_{L^q}
		\le \|\rho\|_{L^\infty} \|u\|_{L^\infty} {\|\nabla u\|_{L^q}}
		\le C \|\rho\|_{L^\infty} \|\nabla u\|_{L^2}^{\frac{1}{2}}
		\|\nabla^2 u\|_{L^2}^{\frac{3}{2}}, \\
		&\| \nabla (P(\rho))\|_{L^q} \le C \|\rho\|_{L^\infty}^{\gamma-1} \|\nabla\rho \|_{L^q}, \\
		&\| \Delta\chi \nabla\chi \|_{L^q} \le
		\|\nabla\chi\|_{L^\infty} \|\nabla^2 \chi\|_{L^q}
		\le C \|\nabla^2 \chi\|_{L^2}^{\frac{1}{2}}  \| \nabla^3 \chi \|_{L^2}^{\frac{3}{2}}.
	\end{align*}
	Integrating the above estimates over $(0,T)$, one has
	\begin{align*}
		\int_0^T \|\rho u_{t}\|_{L^q} \,dt &\le
		\int_0^T \|\rho\|_{L^\infty}^{\frac{5q-6}{4q}} \| \sqrt\rho u_{t}\|_{L^2}^{\frac{6-q}{2q}}
		\|\sqrt{t} \nabla u_{t}\|_{L^2}^{\frac{3q-6}{2q}} t^{-\frac{3q-6}{4q}} \,dt \\
		& \le C {\varPhi^{\frac{5q-6}{4q}}(T) \left( \int_0^T \|\sqrt\rho u_{t}\|^2_{L^2} \,dt
		\right)^{\frac{6-q}{4q}} \left( \int_0^T \|\sqrt{t} \nabla u_{t}\|^2_{L^2} \,dt
		\right)^{\frac{3q-6}{4q}} T^{\frac{6-q}{4q}}} \\
		& \le C T^{\frac{6-q}{4q}} \varPhi^{\frac{7q-6}{4q}}(T),  \\
		\int_0^T \| \rho ( u \cdot \nabla ) u \|_{L^q} \,dt
		&\le \int_0^T \|\rho\|_{L^\infty} \|\nabla u\|_{L^2}^{\frac{1}{2}}
		\|\nabla^2 u\|_{L^2}^{\frac{3}{2}} \,dt  \\
		&\le C \varPhi^{\frac{5}{4}}(T)  \left( \int_0^T \|\nabla^2 u \|_{L^2}^2 \,dt
		\right)^{\frac{3}{4}} T^{\frac{1}{4}}  \le CT^{\frac{1}{4}}\varPhi^2(T), \\
		\int_0^T \| \nabla (P(\rho))\|_{L^q} \,dt
		&\le \int_0^T \|\rho\|_{L^\infty}^{\gamma-1} \|\nabla\rho \|_{L^q} \,dt  \le CT\varPhi^{\gamma}(T),\\
		\int_0^T \| \Delta\chi \nabla\chi \|_{L^q} \,dt
		&\le C \int_0^T \|\nabla^2 \chi\|_{L^2}^{\frac{1}{2}}
		\| \nabla^3 \chi \|_{L^2}^{\frac{3}{2}} \,dt \\
		&\le C \varPhi^{\frac{1}{4}}(T) \left( \int_0^T \|\nabla^3 \chi\|_{L^2}^2 \,dt \right)^{\frac{3}{4}} T^{\frac{1}{4}} \le CT^{\frac{1}{4}}\varPhi^2(T).
	\end{align*}
	Combining the above estimates, we obtain
	\begin{align*}
		\int_0^T \|\nabla^2 u\|_{L^q}^2 \,dt
		\le C  \left(T^{\frac{6-q}{4q}} \varPhi^{\frac{7q-6}{4q}}(T)
		+ T^{\frac{1}{4}}\varPhi^{2}(T) + T \varPhi^{\gamma}(T) + T^{\frac{1}{4}}\varPhi^2(T)\right)
		\le C T^{\frac{6-q}{4q}} \varPhi^{1 + \gamma},
	\end{align*}
	where we have used $\frac{6-q}{4q} \le \frac{1}{4}$ for each $q \in (3, 6)$, $T \le 1$, $\gamma > 1$
	and $\varPhi(T) \ge 1$. It follows from the Sobolev and Poincar\'e inequalities that
	\begin{align*}
		\int_{0}^{T} \|\nabla u\|_{L^\infty} \,dt \le C \int_0^T \|\nabla^2 u\|_{L^q} \, dt
		\le C T^{\frac{6-q}{4q}} \varPhi^{\gamma + 1}(T).
	\end{align*}
	This completes the proof.
\end{proof}

\begin{lemma}\label{lem4}
	Let {$(\rho, u, \chi, \mu)$} be a smooth solution to the system \eqref{NSAC},
	associated with the initial boundary conditions \eqref{I}-\eqref{B1}. There exists a constant $ \varepsilon_0 \in ( 0, 1 ) $, depending only on $A$, $\gamma$, $\nu$,
	$\lambda$, $q$, $\Om$ and $\Phi_0$, such that, if
	\begin{equation}\label{Assump1}
		T^{\frac{6-q}{4q}}\varPhi^{\gamma + 1}(T) \le \varepsilon_0,
	\end{equation}
	then it holds that
	\begin{align*}
		\sup_{0 \le t  \le T} \left( \|\rho\|_{L^\infty} + \|\rho\|_{W^{1,q}} \right) \le C.
	\end{align*}
\end{lemma}
\begin{proof}
	Multiplying $\eqref{NSAC}_1$ by $ q\rho^{q-1}$ and integrating over $\Om$,
	it follows from integration by parts that
	\begin{align*}
			\frac{d}{dt}\|\rho\|_{L^q}^q & =  (1 - q ) \int_{\Om} \rho^{q} \Div u \,dx  \\
			&\le C\|\nabla u\|_{L^\infty}\|\rho\|_{L^q}^q .
	\end{align*}
	By the Gr\"onwall inequality, we obtain
	\begin{align*}
		\sup_{0 \le t \le T}\|\rho\|_{L^q}
		\le  \| \rho_0\|_{L^q} \Exp
		\left\{C\int_{0}^{T}\|\nabla u\|_{L^\infty} \, dt \right\},
	\end{align*}
	and
	\begin{align}\label{sup_rho_infty}
		\sup_{0 \le t  \le T} \|\rho\|_{L^\infty}
		\le \|\rho_0\|_{L^\infty} \Exp
		\left\{C \int_{0}^{T} \|\nabla u\|_{L^\infty} \,dt \right\}.
	\end{align}
	Combining \eqref{sup_rho_infty} with Lemma~\ref{lem3} to get
	\begin{align*}
		\sup_{0 \le t  \le T} \|\rho\|_{L^\infty}
		\le \|\rho_0\|_{L^\infty}
		\Exp \left\{ C T^{\frac{6-q}{4q}}\varPhi^{\gamma + 1}(T) \right\}.
	\end{align*}
	Choosing $\varepsilon_0$ sufficiently small, then it follows from \eqref{Assump1} that
	\begin{align*}
		\sup_{0 \le t  \le T} \|\rho\|_{L^\infty}
		\le 2 \|\rho_0\|_{L^\infty}.
	\end{align*}

	Differentiating $\eqref{NSAC}_1$ with respect to $x$, multiplying the resultant with $ q|\nabla\rho|^{q-2}\nabla\rho$ and integrating over $\Om$, it follows from integration by parts that
	\begin{align*}
		\frac{d}{dt} \|\nabla\rho\|_{L^q}^q
		&=  -q \int_{\Om} ( \nabla\rho \cdot \nabla ) u \cdot
		\nabla\rho | \nabla\rho |^{q - 2}  \,dx
		+ (1 - q ) \int_{\Om}  \Div u |\nabla\rho|^q \,dx \\
		&\quad - q \int_{\Om} \rho \nabla \Div u \cdot \nabla\rho |\nabla\rho|^{q-2} \,dx \\
		& \le C\|\nabla u\|_{L^\infty}\|\nabla \rho\|_{L^q}^q
		+ C\|\rho\|_{L^\infty} \|\nabla \rho\|_{L^q}^{q-1} \|\nabla^2 u \|_{L^q},
	\end{align*}
	which indicates that
	\begin{align}\label{nabla_rho_q} 
		\frac{d}{dt} \|\nabla\rho\|_{L^q}
		\le C\|\nabla u\|_{L^\infty} \|\nabla\rho\|_{L^q}
		+ C\|\rho\|_{L^\infty} \|\nabla^2 u \|_{L^q}.
	\end{align}
	Then, one obtains
	\begin{align*}
		\frac{d}{dt} \|\rho\|_{W^{1, q}}
		\le C\|\nabla u\|_{L^\infty} \|\rho\|_{W^{1, q}}
		+ C \|\rho\|_{L^\infty} \|\nabla^2 u \|_{L^q}
		\le C\|\nabla u\|_{L^\infty} \|\rho\|_{W^{1, q}}
		+ C \|\nabla^2 u \|_{L^q}.
	\end{align*}
	By Lemma~\ref{lem3} and the Gr\"onwall inequality, one has
	\begin{align*}
		\|\rho\|_{W^{1,q}} &\le \left( \| \rho_0\|_{W^{1,q}}
		+ \int_0^T  \|\nabla^2 u \|_{L^q} \,dt \right)
		\Exp \left\{C\int_{0}^{T}\|\nabla u\|_{L^\infty} \, dt \right\}\\
		&\le C\left(1 +   T^{\frac{6-q}{4q}}\varPhi^{\gamma + 1}(T) \right)
		\Exp \left\{ C T^{\frac{6-q}{4q}}\varPhi^{\gamma + 1}(T) \right\} \le C.
	\end{align*}
	This completes the proof.
\end{proof}

\begin{remark}
	Under the assumption \eqref{Assump1}, it follows from $\varPhi(T) \ge 1 $ that
	\begin{equation}\label{A2}
		T \varPhi^{\alpha}(T) \le \big( T^{\frac{6-q}{4q}} \varPhi^{\gamma+1}(T) \big)^{\frac{4q}{6-q}}
		(\varPhi (T) \big)^{\alpha - \frac{4q ( \gamma + 1 )}{6-q}} \le C,
	\end{equation}
	as long as $\alpha \le 4\gamma +4$.
\end{remark}

Now we turn to do some a priori estimates on phase variable $\chi$.
\begin{lemma}\label{lem:chi1}
	Under the assumptions of Lemma~\ref{lem4}, and if \eqref{Assump1} holds, then the following estimate holds
	\begin{equation*}
		\sup_{0 \le t \le T} \| \rho  \chi_t \|_{L^2}^2
		+ \int_{0}^{T} \| \nabla  \chi_t \|_{L^2}^2 \,dt \le  C .
	\end{equation*}
\end{lemma}

\begin{proof}
Rewrite the equations $\eqref{NSAC}_{3,4}$ as follows
\begin{equation}
	\label{NSAC_34}
	\rho^2  \chi_t + \rho^2 u \cdot \nabla \chi = \Delta \chi - \rho f(\chi).
\end{equation}
	Differentiating \eqref{NSAC_34} with respect to $t$ gives
	\begin{equation}
		\label{dt_NSAC_34}
		2 \rho   \rho_t \chi_t + \rho^2 \chi_{tt} + 2 \rho  \rho_t u \cdot \nabla \chi + \rho^2   u_t \cdot \nabla \chi + \rho^2 u \cdot \nabla  \chi_t = \Delta  \chi_t - \rho_t f(\chi) - \rho f'(\chi) \chi_t.
	\end{equation}
	Multiplying \eqref{dt_NSAC_34} by $ \chi_t$, and integrating over $\Om$, using $\eqref{NSAC}_1$, applying integration by parts, one has
\begin{align*}
	\frac{1}{2} \frac{d}{dt} \|\rho \chi_t\|_{L^2}^{2} + \|\nabla\chi_t\|_{L^2}^{2}
	=& \frac{1}{2} \int_{\Om} \rho^2 \Div u \chi_t^{2} \,dx
	- 2 \int_{\Om} \rho^{2} ( u \cdot \nabla \chi_t) \chi_t\,dx
	+ 2 \int_{\Om} \rho \rho_t (u \cdot \nabla \chi ) \chi_t \,dx\\
	& + \int_{\Om}\rho ( u_t \cdot \nabla \chi ) \rho \chi_t \,dx
	- \int_{\Om} \rho ( u \cdot \nabla\chi_t ) f(\chi) \,dx \\
	& - \int_{\Om} \rho f'(\chi) ( u \cdot \nabla\chi ) \chi_t \,dx
	- \int_{\Om} \rho f'(x) \chi_t^{2} \,dx \\
	=& \sum_{i=1}^{7}{I}_{i}.
\end{align*}
	Combining Lemma~\ref{lem1}, Lemma~\ref{lem4}, and Corollary~\ref{cor1}, we deduce that
\begin{align*}
	I_1 + I_2 &\le C \|\rho\|_{L^\infty} ( \|\nabla u\|_{L^3} \|\chi_t\|_{L^6}
	+ \|u\|_{L^\infty} \|\nabla\chi_t\|_{L^2} ) \|\rho \chi_t\|_{L^2} \\
	& \le C \|\nabla u\|_{L^2}^{\frac{1}{2}} \|\nabla^2 u\|_{L^2}^{\frac{1}{2}}
	( \| \rho \chi_t\|_{L^2} + \|\nabla\chi_t\|_{L^2}) \|\rho \chi_t\|_{L^2} \\
	& \le \frac{1}{4} \| \nabla \chi_t \|_{L^2}^2
	+ C (1 + \|\nabla u\|_{L^2} \|\nabla^2 u\|_{L^2} ) \|\rho \chi_t\|_{L^2}^2, \\
	I_3 + I_6 &\le C ( \|\rho_t\|_{L^3} + \|f'(\chi)\|_{L^3} )
	\|u\|_{L^\infty} \|\nabla\chi\|_{L^6} \| \rho\chi_t\|_{L^2}  \\
	&\le C ( \|\rho\|_{L^\infty} \|\nabla u\|_{L^3}
	+ \|\nabla\rho\|_{L^3} \|u\|_{L^\infty} + 1 ) \|u\|_{L^\infty}
	\|\nabla^2 \chi\|_{L^2} \|\rho \chi_t\|_{L^2} \\
	& \le C (1 + \|\nabla u\|_{L^2} \|\nabla^2 u\|_{L^2} )
	\|\nabla^2 \chi\|_{L^2} \| \rho\chi_t \|_{L^2},  \\
	I_4 &\le \|\rho\|_{L^\infty}^\frac{1}{2} \| \sqrt\rho u_t \|_{L^2}
	\|\nabla\chi\|_{L^\infty} \| \rho \chi_t\|_{L^2}  \\
	&\le C \| \sqrt\rho u_t\|_{L^2} \|\nabla^3\chi\|_{L^2}^{\frac{1}{2}}
	\|\nabla^2\chi\|_{L^2}^{\frac{1}{2}} \|\rho \chi_t\|_{L^2},   \\
	I_5 + I_7 &\le \|\rho\|_{L^\infty} \|u\|_{L^\infty}
	\|\nabla\chi_t\|_{L^2} \|f(\chi)\|_{L^2}
	+ \| \rho \chi_t\|_{L^2} \|f'(\chi)\|_{L^3} \|\chi_t\|_{L^6} \\
	& \le C \|\nabla u\|_{L^2}^{\frac{1}{2}} \|\nabla^2 u\|_{L^2}^{\frac{1}{2}}
	\|\nabla\chi_t\|_{L^2}
	+ C \|\rho \chi_t\|_{L^2}
	(\|\rho \chi_t\|_{L^2} + \|\nabla\chi_t\|_{L^2}) \\
	& \le \frac{1}{4} \|\nabla\chi_t\|_{L^2}^2
	+ C \|\nabla u\|_{L^2} \|\nabla^2 u\|_{L^2} + C \| \rho \chi_t\|_{L^2}^2.
\end{align*}	
	Then we obtain
	\begin{equation*}
		\begin{split}
			\frac{d}{dt} \|\rho \chi_t\|_{L^2}^{2} + \|\nabla \chi_t \|_{L^2}^{2}
			& \le C (1 + \|\nabla u\|_{L^2} \|\nabla^2 u\|_{L^2} )
			(1+ \|\nabla^2 \chi\|_{L^2}^2 + \| \rho\chi_t \|_{L^2}^2 )\\
			&\quad + C \| \sqrt\rho u_t\|_{L^2} \|\nabla^3\chi\|_{L^2}^{\frac{1}{2}}
			\|\nabla^2\chi\|_{L^2}^{\frac{1}{2}} \|\rho \chi_t\|_{L^2}   .
		\end{split}
	\end{equation*}
	Integrating it over $(0, T )$,  using \eqref{com1} and \eqref{A2}, together with the fact that $\varPhi(T) \geq 1$, we deduce from the H\"older inequality that
	\begin{align*}
		&\sup_{ 0 \le t \le T} \|\rho \chi_t\|_{L^2}^{2}
		+ \int_0^T \| \nabla\chi_t \|_{L^2}^{2} \,dt\\
		\le& \lim_{\tau \rightarrow 0}\|\rho \chi_t\|_{L^2}^{2} (\tau) + T + T \varPhi(T)+ T^{\frac{1}{2}} \varPhi^{\frac{3}{2}}(T)
		\left( \int_0^T \| \nabla^2 u \|_{L^2}^2 \,dt \right)^{\frac{1}{2}}\\
		&
		+ T^{\frac{1}{4}} \varPhi^{\frac{3}{4}}(T)
		\left( \int_0^T \| \sqrt{\rho} u_t \|_{L^2}^2 \,dt \right)^{\frac{1}{2}}
		\left( \int_0^T \| \nabla^3 \chi \|_{L^2}^2 \,dt \right)^{\frac{1}{4}}\\
		\le& C + C T \varPhi^2(T) +  T^{\frac{1}{4}} \varPhi^{\frac{3}{2}}(T) \le C,
	\end{align*}
	where we have used
	\begin{align*}
		\lim_{\tau \rightarrow 0}\|\rho \chi_t\|_{L^2}^{2} (\tau)
		&\le C ( \| \rho_0 ( u_0 \cdot \nabla\chi_0 ) \|_{L^2}^2
		+ \left\| \frac{\Delta\chi_0}{\rho_0} \right\|_{L^2}^2
		+ \| f(\chi_0) \|_{L^2}^2 ) \\
		& \le C ( \|\rho_0\|_{L^\infty}^2 \|u_0\|_{L^6}^2 \| \nabla \chi_0\|_{L^3}^2
		+\| h \|_{L^2}^2 + 1 ) \le C.
	\end{align*}
\end{proof}

\begin{lemma}\label{lem:chi2}
	Under the assumptions of Lemma~\ref{lem4}, and if \eqref{Assump1} holds, then the following estimate holds
	\begin{equation*}
		\int_0^T \|\nabla^3\chi\|_{L^2}^2 \,dt \le  C .
	\end{equation*}
\end{lemma}

\begin{proof}
	Applying the gradient operator to \eqref{NSAC_34}, we obtain
	\begin{align*}
		\|\nabla \Delta \chi\|_{L^2} &\le \big( \| 2 \rho \nabla\rho \chi_t \|_{L^2}
		+ \| \rho^2 \nabla\chi_t \|_{L^2}
		+ \| 2 \rho \nabla\rho ( u \cdot \nabla\chi ) \|_{L^2}
		+ \|\rho^2 \nabla u \nabla\chi\|_{L^2} \\
		&\quad + \| \rho^2 \nabla^2\chi u \|_{L^2}
		+ \|\nabla\rho f(\chi)\|_{L^2}
		+ \|\rho f'(\chi) \nabla\chi\|_{L^2} \big).
	\end{align*}
	By Lemma \ref{lem1}, Lemma~\ref{lem4} and Lemma~\ref{lem:chi1}, a direct computation shows
	\begin{align*}
		\| 2 \rho \nabla\rho \chi_t \|_{L^2}
		+ \| \rho^2 \nabla\chi_t \|_{L^2}
		&\le 2 \|\rho\|_{L^\infty} \|\nabla\rho\|_{L^3} \|\chi_{t}\|_{L^6}
		+ \|\rho\|_{L^\infty}^2 \|\nabla\chi_{t}\|_{L^2} \\
		&\le C (1 + \|\nabla\chi_{t}\|_{L^2} ), \\
		\|2 \rho \nabla\rho ( u \cdot \nabla ) \chi \|_{L^2}
		+ \|\rho^2 \nabla u \nabla\chi\|_{L^2}
		&\le ( 2 \|\rho\|_{L^\infty} \|\nabla\rho\|_{L^3} \|u \|_{L^6}
		+ \|\rho\|_{L^\infty}^2 \|\nabla u\|_{L^2} ) \|\nabla\chi\|_{L^\infty} \\
		&\le C \|\nabla u\|_{L^2} \|\nabla^2\chi\|_{L^2}^{\frac{1}{2}}
		\| \nabla ^3 \chi\|_{L^2} ^{\frac{1}{2}} \\
		&\le C \|\nabla u\|_{L^2} \|\nabla^2\chi\|_{L^2}^{\frac{1}{2}}
		\| \nabla \Delta \chi\|_{L^2} ^{\frac{1}{2}} \\
		&\le \frac{1}{2} \|\nabla \Delta \chi\|_{L^2}
		+ C \|\nabla u\|_{L^2}^2 \|\nabla^2\chi\|_{L^2},\\
		\|\nabla\rho f(\chi)\|_{L^2}
		+ \|\rho f'(\chi) \nabla\chi\|_{L^2}
		&\le C (\|\nabla \rho\|_{L^3} \|f(\chi) \|_{L^6}
		+ \|\rho\|_{L^\infty} \|f'(\chi) \|_{L^3}\|\nabla \chi\|_{L^6}) \\
		&\le C(1 + \|\nabla^2\chi\|_{L^2}),
	\end{align*}
	and
	\begin{align*}
		\| \rho^2 \nabla^2\chi u \|_{L^2} &\le \|\rho\|_{L^\infty}^2 \|\nabla^2\chi\|_{L^2} \|u\|_{L^\infty}\\
		&\le C \|\nabla^2\chi\|_{L^2} \|\nabla u\|_{L^2}^{\frac12} \|\nabla^2 u\|_{L^2}^{\frac12}.
	\end{align*}
	Then, using the $H^3$-estimates of Neumann-Laplacian and \eqref{B1}, it follows that
	\begin{equation}\label{nabla_cube_chi}
		\begin{aligned}
			\|\nabla^3\chi\|_{L^2}^2
			 &\le C \|\nabla \Delta \chi\|_{L^2}^2 \\
			&\le C (1  + \|\nabla\chi_{t}\|_{L^2}^2
			+ \|\nabla u\|_{L^2}^4 \|\nabla^2 \chi\|_{L^2}^2 + \|\nabla^2\chi\|_{L^2}^2
			+ \|\nabla^2\chi\|_{L^2}^2 \|\nabla u\|_{L^2} \|\nabla^2 u\|_{L^2}).
		\end{aligned}
	\end{equation}
	Therefore, one deduces by \eqref{A2}, \eqref{nabla_cube_chi},
	Lemma~\ref{lem:chi1} and the H\"older inequality that
	\begin{align*}
		&\int_0^T \|\nabla^3\chi\|_{L^2}^2 \,dt\\
		\le&{ C\int_0^T  \left(1 + \|\nabla\chi_{t}\|_{L^2}^2
		+ \|\nabla u\|_{L^2}^4 \|\nabla^2 \chi\|_{L^2}^2 + \|\nabla^2\chi\|_{L^2}^2
		+ \|\nabla^2\chi\|_{L^2}^2 \|\nabla u\|_{L^2} \|\nabla^2 u\|_{L^2} \right) \,dt }\\
		\le& C + C T \varPhi^3(T) + T \varPhi(T) + T^{\frac12} \varPhi^2(T) \le C.
	\end{align*}
	This completes the proof.
\end{proof}

The following a priori estimates are derived by treating the convection term
$\| \rho ( u \cdot \nabla ) \chi \|_{L^2}^2$ as an energy, without requiring high regularity of $u_t$.
\begin{lemma}\label{lem:u1}
	Under the assumptions of Lemma~\ref{lem4} and assuming that \eqref{Assump1} holds, we have
	\begin{equation*}
		\sup_{0 \le t \le T} \left( \|\nabla u\|_{L^2}^2
		+ \|\mu\|_{L^2}^2 + \| \rho ( u \cdot \nabla ) \chi \|_{L^2}^2 \right)
		+ \int_0^T ( \| \sqrt{\rho} u_t \|_{L^2}^2 + \| \nabla^2u \|_{L^2}^2 ) \,dt \le  C .
	\end{equation*}
\end{lemma}
\begin{proof}
	By Lemma~\ref{lem4} and $\eqref{NSAC}_2$,  one deduces by the standard $H^2$-estimate for the Lam\'e system that
	\begin{align*}
		\|\nabla^2 u\|_{L^2}^2 &\le C ( \| \rho u_t \|_{L^2}^2 + \| \rho ( u \cdot \nabla ) u \|_{L^2}^2
		+ \|\nabla P(\rho)\|_{L^2}^2 + \| \Delta\chi \nabla\chi \|_{L^2}^2 ) \\
		&\le C ( \|\rho\|_{L^\infty} \| \sqrt{\rho} u_t \| _{L^2}^2
		+ \|\rho\|_{L^\infty}^2 \|u\|_{L^6}^2 \|\nabla u\|_{L^3}^2
		+ \|\rho\|_{L^\infty}^{2\gamma-2} \|\nabla\rho\|_{L^2}^2
		+ \|\nabla\chi\|_{L^\infty}^2 \|\nabla^2\chi\|_{L^2}^2 ) \\
		&\le C (1+ \|\sqrt{\rho} u_t\| _{L^2}^2 + \|\nabla u\|_{L^2}^3 \|\nabla^2 u\|_{L^2}
		+ \|\nabla^2\chi\|_{L^2}^3 \| \nabla^3 \chi \|_{L^2} )\\
		&\le \frac{1}{2} \|\nabla^2 u\|_{L^2}^2 + C (1+ \| \sqrt{\rho} u_t \| _{L^2}^2
		+ \|\nabla u\|_{L^2}^6 + \|\nabla^2\chi\|_{L^2}^3 \| \nabla^3 \chi \|_{L^2} ),
	\end{align*}
	thus we have
	\begin{equation}\label{lem:u1:1}
		\|\nabla^2 u \|_{L^2}^2  \le C (1 +\| \sqrt{\rho} u_t \| _{L^2}^2 + \|\nabla u\|_{L^2}^6
		+ \|\nabla^2\chi\|_{L^2}^3 \|\nabla^3 \chi\|_{L^2}).
	\end{equation}
	From $\eqref{NSAC}_1$ and \eqref{NSAC_34}, after integration by parts, one obtains
	\begin{align*}
		&\int_{\Om} \nabla P(\rho) \cdot  u_{t} \,dx
		= -\frac{d}{dt} \int_\Om P(\rho) \Div u  \,dx
		+ \int_\Om P'(\rho) \rho_t \Div u \,dx, \\
		&\int_\Om  P'(\rho) \rho_t \Div u \,dx
		= -\int_\Om ( \nabla P(\rho) \cdot u ) \Div u \,dx
		- \gamma \int_\Om P(\rho) | \Div u |^2 \,dx ,\\
		&- \int_\Om \Div ( \nabla\chi \otimes \nabla\chi
		- \frac{ | \nabla\chi |^2}{2} \mathbb{I} ) \cdot u_t \,dx
		=  - \frac{d}{dt} \int_\Om \Delta \chi \nabla\chi \cdot u \,dx
		- \int_\Om(\nabla \chi_t \cdot \nabla ) u \cdot \nabla \chi \,dx\\
		&\quad\quad\quad\quad\quad\quad\quad\quad\quad\quad\quad\quad\quad\quad\quad\quad\quad\quad- \int_\Om(\nabla \chi \cdot \nabla ) u \cdot \nabla \chi_t \,dx
		+ \int_\Om \nabla \chi \cdot \nabla \chi_t \Div u \,dx,
	\end{align*}
	and
	\begin{align*}
		- \frac{d}{dt} \int_\Omega  \Delta \chi \nabla\chi \cdot u \,dx
		&= - \frac{d}{dt} \int_\Omega (  \rho^2\chi_t
		+ \rho^2 ( u \cdot \nabla\chi ) + \rho F'(\chi)  )
		(  u\cdot \nabla\chi ) \,dx \\
		&= - \frac{1}{2} \frac{d}{dt} \left( \| \rho ( u \cdot \nabla\chi ) \|_{L^2}^2
		+ \|\mu\|_{L^2}^2 \right) \\
		&\quad+ \frac{d}{dt} \left( \frac{1}{2} \|\rho \chi_t\|_{L^2}^2
		- \int_\Om \rho F'(\chi) (u \cdot \nabla\chi )\,dx \right),
	\end{align*}
	where we have used the identity
	\begin{align*}
		\|\mu\|_{L^2}^2 = \|\rho \chi_t\|_{L^2}^2
		+ 2 \int_\Om \rho^2 \chi_t ( u \cdot \nabla\chi ) \,dx
		+ \| \rho ( u \cdot \nabla\chi ) \|_{L^2}^2.
	\end{align*}
	Multiplying $\eqref{NSAC}_2$ by $u_{t}$, one deduces by the above identities that
	\begin{equation}
		\begin{split}\label{lem:u1:2}
			&\frac{1}{2} \frac{d}{dt} \left( \nu \|\nabla u\|_{L^2}^2
			+ (\lambda + \nu ) \|\Div u\|_{L^2}^2 + \|\mu\|_{L^2}^2
			+ \| \rho ( u \cdot \nabla\chi ) \|_{L^2}^2 \right)  + \|\sqrt{\rho} u_t\|_{L^2}^2\\
			=& - \int_{\Om} \rho ( u \cdot \nabla ) u \cdot u_t \,dx
			+ \gamma \int_{\Om} P(\rho) | \Div u|^2 \,dx
			+ \int_{\Om} \nabla P(\rho) \cdot  u \Div u \,dx \\
			&- \int_\Om ( \nabla\chi_t \cdot \nabla ) u \cdot \nabla\chi \,dx
			- \int_\Om ( \nabla\chi \cdot \nabla ) u \cdot \nabla\chi_t \,dx
			+ \int_\Om \nabla\chi \cdot \nabla\chi_t \Div u \,dx\\
			& +  \frac{d}{dt} \left( \int_{\Om} P(\rho) \Div  u \,dx
			+ \frac{1}{2} \|\rho \chi_t\|_{L^2}^2
			- \int_\Om \rho F'(\chi) (u \cdot \nabla\chi )\,dx \right)\\
			=& \sum_{i=1}^6 J_i + G'(t),
		\end{split}
	\end{equation}
	where
	$$G(t) :=  \int_{\Om} P(\rho) \Div  u \,dx
	+ \frac{1}{2} \|\rho \chi_t\|_{L^2}^2
	- \int_\Om \rho F'(\chi) (u \cdot \nabla\chi )\,dx. $$
	For each $\eta > 0 $, it follows from Lemma~\ref{lem4} that,
	\begin{align*}
		J_1 &\le \|\rho\|_{L^\infty}^{\frac{1}{2}}
		\| \sqrt{\rho}u_t \|_{L^2} \|u\|_{L^6} \|\nabla u\|_{L^3}
		\le C \| \sqrt{\rho} u_t \|_{L^2} \|\nabla u\|_{L^2}^{\frac{3}{2}}
		\|\nabla^2 u\|_{L^2}^{\frac{1}{2}}  \\
		&\le \frac{1}{4} \| \sqrt{\rho} u_t \|_{L^2}^2
		+ \eta \|\nabla^2 u\|_{L^2}^2 + C_\eta \|\nabla u\|_{L^2}^6, \\
		J_2 + J_3 &\le C \|\rho\|_{L^\infty}^{\gamma} \|\nabla u\|_{L^2}^2
		+ C \|\rho\|_{L^\infty}^{\gamma-1} \|\nabla\rho\|_{L^3}
		\|u\|_{L^6} \|\nabla u\|_{L^2}
		\le C \|\nabla u\|_{L^2}^2, \\
		\sum_{i=4}^6 {J_i} &\le C \|\nabla\chi_t\|_{L^2}
		\|\nabla\chi\|_{L^\infty} \|\nabla u\|_{L^2}
		\le C \|\nabla\chi_t\|_{L^2} \|\nabla^2\chi\|_{L^2}^{\frac{1}{2}}
		\| \nabla^3 \chi\|_{L^2}^{\frac{1}{2}}\|\nabla u\|_{L^2}.
	\end{align*}
	Substituting the above estimates and \eqref{lem:u1:1} into \eqref{lem:u1:2}, we obtain
	\begin{align*}
		&\frac{1}{2} \frac{d}{dt} \left( \nu \|\nabla u\|_{L^2}^2
		+ (\lambda + \nu ) \|\Div u\|_{L^2}^2 + \|\mu\|_{L^2}^2
		+ \| \rho ( u \cdot \nabla\chi ) \|_{L^2}^2 \right)  + \frac{3}{4} \|\sqrt{\rho} u_t\|_{L^2}^2
		+ \eta \|\nabla^2 u\|_{L^2} \\
	 \le& 2 \eta \|\nabla^2 u\|_{L^2}^2 + C_\eta \|\nabla u\|_{L^2}^6
		+ C ( \|\nabla\chi_t\|_{L^2} \|\nabla^2\chi\|_{L^2}^{\frac{1}{2}}
		\|\nabla^3\chi\|_{L^2}^{\frac{1}{2}} \|\nabla u\|_{L^2} + \|\nabla u\|_{L^2}^2 ) + G'(t)\\
		\le& C \eta \| \sqrt{\rho} u_t \| _{L^2}^2 + \tilde{C}_{\eta} \|\nabla u\|_{L^2}^6
		+ C (\|\nabla^2\chi\|_{L^2}^3 \|\nabla^3 \chi\|_{L^2}  + \|\nabla\chi_t\|_{L^2} \|\nabla^2\chi\|_{L^2}^{\frac{1}{2}}
		\|\nabla^3\chi\|_{L^2}^{\frac{1}{2}} \|\nabla u\|_{L^2} \\
		&+ \|\nabla u\|_{L^2}^2 + 1 ) + G'(t).
	\end{align*}
	Choosing $\eta$ sufficiently small, one deduces by the Young inequality that
	\begin{equation}\label{lem:u1:3}
		\begin{split}
			&\frac{d}{dt} \left( \nu \|\nabla u\|_{L^2}^2
			+ (\lambda + \nu ) \|\Div u\|_{L^2}^2 + \|\mu\|_{L^2}^2
			+ \| \rho ( u \cdot \nabla\chi ) \|_{L^2}^2 \right)  + \|\sqrt{\rho} u_t\|_{L^2}^2
			+ 2\eta \|\nabla^2 u\|_{L^2}\\
			\le&C ( \|\nabla u\|_{L^2}^6 + \|\nabla^2\chi\|_{L^2}^6
			+ \|\nabla^3 \chi\|_{L^2}^2 + \|\nabla \chi_t\|_{L^2}^2
			+ \|\nabla u\|_{L^2}^8 + \|\nabla^2\chi\|_{L^2}^4 + 1 ) + 2 G'(t).
		\end{split}
	\end{equation}
	It follows from \eqref{com1}, Lemma~\ref{lem4} and Lemma~\ref{lem:chi1}   that
	\begin{align*}
		 G(t) &\le \|P(\rho)\|_{L^2} \|\nabla u\|_{L^2}
		+ \frac{1}{2} \|\rho \chi_t\|_{L^2}
		+ \| \rho ( u \cdot \nabla ) \chi \|_{L^2} \|f(\chi)\|_{L^2}\\
		&\le \frac{\nu}{4} \|\nabla u\|_{L^2}^2
		+ \frac{1}{4} \| \rho ( u \cdot \nabla ) \chi \|_{L^2}^2 + C, \\
		 |G(0)| &\le \|P(\rho_0)\|_{L^2} \|\nabla u_0\|_{L^2}
		+ \frac{1}{2} \|\rho \chi_t\|_{L^2}(0)
		+ \| \rho_0 ( u_0 \cdot \nabla ) \chi_0 \|_{L^2} \| f(\chi_0) \|_{L^2}\\
		&\le C + \|\rho_0\|_{L^3} \|u_0\|_{L^6} \|\nabla \chi_0\|_{L^6}\le C.
	\end{align*}
	Plugging the above estimates into \eqref{lem:u1:3} and integrating over $(0,T)$, it follows from
	\eqref{com1}, \eqref{A2}, Lemma~\ref{lem:chi1} and Lemma~\ref{lem:chi2} that
	\begin{align*}
		&\sup_{0 \le t \le T} \left( \|\nabla u\|_{L^2}^2
		+\|\mu\|_{L^2}^2 + \| \rho ( u \cdot \nabla ) \chi \|_{L^2}^2 \right)
		+ \int_0^T ( \| \sqrt{\rho} u_t \|_{L^2}^2 +  \| \nabla^2u \|_{L^2}^2 ) \,dt\\
		\le& C \int_0^T ( \|\nabla u\|_{L^2}^6  + \|\nabla^2\chi\|_{L^2}^6
		+ \| \nabla^3\chi\|_{L^2}^2 + \|\nabla\chi_t\|_{L^2}^2 + \|\nabla u\|_{L^2}^8
		+ \|\nabla^2\chi\|_{L^2}^4) \,dt + C \\
		\le& C T \varPhi^{4}(T) + C T \varPhi^3(T) + C T \varPhi^2(T) + C \\
		\le& C.
	\end{align*}
	This completes the proof.
\end{proof}

\begin{corollary}\label{cor2}
	Under the assumptions of Lemma~\ref{lem4} and assuming that \eqref{Assump1} holds, we have
	\begin{equation*}
		\sup_{0 \le t \le T} \left( \|\rho_t\|_{L^2} + \|\nabla^2\chi\|_{L^2}
		+ \|\chi\|_{L^\infty} \right) \le  C .
	\end{equation*}
\end{corollary}

\begin{proof}
	It follows from $\eqref{NSAC}_1$, $\eqref{NSAC}_4$, Lemma~\ref{cor1}, Lemma~\ref{lem4},
	the Sobolev and Poincar\'e inequalities that
	\begin{align*}
		&\|\rho_t\|_{L^2}  \le \|\nabla\rho\|_{L^3}\|u\|_{L^6}
		+ \|\rho\|_{L^\infty} \|\nabla u\|_{L^2} \le C \|\nabla u\|_{L^2} ,\\
		&\|\nabla^2\chi\|_{L^2} \le \|\rho\|_{\infty}
		(\|\mu\|_{L^2} + \|f(\chi)\|_{L^2} ) \le C (1 + \|\mu\|_{L^2} ), \\
		&\|\chi\|_{L^\infty} \le C \|\chi\|_{H^2} \le C (1 + \|\nabla^2\chi\|_{L^2} ),
	\end{align*}
	By using Lemma~\ref{lem:u1}, the proof is complete.
\end{proof}

In what follows, we do some time-weighted a priori estimates.
\begin{lemma}\label{lem:u2}
	Under the assumptions of Lemma~\ref{lem4} and assuming that \eqref{Assump1} holds, there holds
	\begin{equation*}
		\sup_{0 \le t  \le T} \| \sqrt{t}\sqrt{\rho}u_t \|_{L^2}^2
		+ \int_0^T \|\sqrt{t} \nabla u_t \|_{L^2}^2  \,dt  \le  C .
	\end{equation*}
\end{lemma}
\begin{proof}
	Differentiating $\eqref{NSAC}_2$ with respect to $t$, we have
	\begin{equation}
		\begin{split}\label{lem:u2:1}
			&\rho_t  u_t + \rho u_{tt} + \rho_t ( u \cdot \nabla ) u + \rho ( u_t \cdot \nabla ) u
			+ \rho(  u \cdot \nabla ) u_t + \nabla \Big(P'(\rho) \rho_t \Big) \\
			&\quad = \mathcal{L} u_t - \Div ( \nabla\chi_t \otimes \nabla\chi
			+ \nabla\chi \otimes \nabla\chi_t
			- \nabla\chi \cdot \nabla\chi_t \mathbb{I} ).
		\end{split}
	\end{equation}
	It follows from integration by parts and $\eqref{NSAC}_1$ that
	\begin{align*}
		&- \int_{\Om} \rho_t | u_t |^2 \,dx
		= \int_{\Om} \Div( \rho u ) | u_t |^2 \,dx = - \int_{\Om} \rho u \cdot \nabla|u_t|^2 \,dx, \\
		&- \int_{\Om} u_{t} \cdot\nabla \Big(P'(\rho) \rho_t \Big) \,dx
		= - A \gamma \int_{\Om} \rho^{\gamma-1} \rho_t \Div u_{t} \,dx.
	\end{align*}
	Testing \eqref{lem:u2:1} with $u_{t}$, and using the above identities, one deduces that
	\begin{align*}
		&\frac{1}{2} \frac{d}{dt} \| \sqrt\rho u_{t} \|_{L^2}^{2}
		+ \nu \|\nabla u_{t} \|_{L^2}^{2} + ( \lambda + \nu ) \| \Div u_{t} \|_{L^2}^{2} \\
		=& - \int_{\Om} \rho u \cdot \nabla|u_t|^2 \,dx
		- \int_{\Om} \rho_t ( u\cdot\nabla ) u \cdot u_t \,dx
		- \int_{\Om} \rho ( u_t \cdot \nabla ) u \cdot u_t \,dx \\
		& + A \gamma \int_{\Om} \rho^{\gamma-1} \rho_t \Div u_t \,dx
		+ \int_{\Om} ( \nabla\chi_t \cdot \nabla) u_t \cdot \nabla\chi \,dx
		+ \int_{\Om} ( \nabla\chi \cdot \nabla) u_t \cdot \nabla\chi_t \,dx \\
		& - \int_{\Om} \nabla\chi \cdot \nabla\chi_t \Div u_t \,dx.
	\end{align*}
	Multiplying the above equation by $t$, then we have
	\begin{equation}
		\begin{split}
			\label{lem:u2:2}
			&\frac{1}{2} \frac{d}{dt} \| \sqrt{t} \sqrt\rho u_{t} \|_{L^2}^{2}
			+ \nu \|\sqrt{t} \nabla u_{t}\|_{L^2}^{2}
			+ ( \lambda + \nu ) \| \sqrt{t} \Div u_{t} \|_{L^2}^{2} \\
			=& \frac{1}{2}\|\sqrt\rho u_{t} \|_{L^2}^{2} - t \int_{\Om} \rho u \cdot \nabla|u_t|^2 \,dx
			- t \int_{\Om} \rho_t ( u \cdot \nabla ) u \cdot u_t \,dx \\
			&- t \int_{\Om} \rho ( u_t \cdot \nabla ) u \cdot u_t \,dx
			+ A \gamma t \int_{\Om} \rho^{\gamma-1} \rho_t \Div u_t \,dx
			+ t \int_{\Om} ( \nabla\chi_t \cdot \nabla) u_t \cdot \nabla\chi \,dx \\
			& + t \int_{\Om} ( \nabla\chi \cdot \nabla) u_t \cdot \nabla\chi_t \,dx
			- t \int_{\Om} \nabla\chi \cdot \nabla\chi_t \Div u_t\,dx\\
			 =&\sum_{i=1}^{8}{K}_i.
		\end{split}
	\end{equation}
	Using the H\"{o}lder, Sobolev, Young, Poincar\'e and Gagliardo-Nirenberg inequalities, it follows from Lemma~\ref{lem4}, Lemma~\ref{lem:u1} and Corollary~\ref{cor2} that
	\begin{align*}
		K_2+ K_4 &\le C \|\rho\|_{L^\infty}^{\frac{1}{2}} \| \sqrt{t} \sqrt\rho u_t\|_{L^2}
		(  \|u\|_{L^\infty} \|\sqrt{t} \nabla u_t\|_{L^2}
		+ \|\nabla u\|_{L^3}\| \sqrt{t}  u_t\|_{L^6} ) \\
		&\le C \|\nabla u\|_{L^2}^{\frac{1}{2}} \|\nabla^2 u\|_{L^2}^{\frac{1}{2}}
		\| \sqrt{t} \sqrt\rho u_t \|_{L^2}\|\sqrt{t} \nabla u_t\|_{L^2} \\
		&\le \frac{\nu}{6} \| \sqrt{t} \nabla u_t\|_{L^2}^2
		+ C \|\nabla^2 u\|_{L^2} \|\sqrt{t} \sqrt\rho u_t \|_{L^2}^2, \\
		K_3 + K_5 &\le C t \| \rho_t\|_{L^2}( \|u\|_{L^6} \|\nabla u\|_{L^6}
		\| u_t \|_{L^6} + \|\rho\|_{L^\infty}^{\gamma-1} \| \nabla u_t \|_{L^2} ) \\
		&\le C\sqrt{t} ( \|\nabla u\|_{L^2} \|\nabla^2 u\|_{L^2} + 1 )
		\| \sqrt{t}\nabla u_t\|_{L^2} \\
		& \le \frac{\nu}{6} \|\sqrt{t} \nabla u_t\|_{L^2}^2
		+ C(1 + \|\sqrt{t} \nabla^2 u\|_{L^2}^2 ),\\
		K_6 + K_7 + K_8 &\le C \|\sqrt{t} \nabla\chi_t\|_{L^2}
		\|\nabla\chi\|_{L^\infty} \| \sqrt{t} \nabla u_t\|_{L^2} \\
		&\le C \|\sqrt{t} \nabla\chi_t\|_{L^2} \|\nabla^2\chi\|_{L^2}^{\frac{1}{2}}
		\|\nabla^3\chi\|_{L^2}^{\frac{1}{2}} \| \sqrt{t} \nabla u_t \|_{L^2} \\
		&\le \frac{\nu}{6} \| \sqrt{t} \nabla u_t\|_{L^2}^2
		+ C \| \nabla^3 \chi\|_{L^2} \|\sqrt{t} \nabla\chi_t\|_{L^2}^2.
	\end{align*}
	Substituting the above estimates into \eqref{lem:u2:2}, we obtain
	\begin{align*}
		&\frac{d}{dt}\|\sqrt{t}\sqrt\rho u_{t} \|_{L^2}^{2} + \nu \|\sqrt{t}\nabla u_{t} \|_{L^2}^{2} + (\lambda+\nu) \|\sqrt{t} \Div  u_{t} \|_{L^2}^{2} \\
		 \le& \| \sqrt\rho u_{t} \|_{L^2}^{2} + C (1 + \| \sqrt{t} \nabla^2 u \|_{L^2}^2 )
		+ C \|\nabla^2 u\|_{L^2} \| \sqrt{t} \sqrt\rho u_t \|_{L^2}^2
		+ C \| \nabla^3 \chi\|_{L^2} \|\sqrt{t} \nabla\chi_t\|_{L^2}^2.
	\end{align*}
	Integrating it over $(0,T)$, one deduces by Lemma~\ref{lem:chi2}, Lemma~\ref{lem:u1}
	and \eqref{A2} that
	\begin{align*}
		&\sup_{0 \le t  \le T} \| \sqrt{t} \sqrt{\rho}u_t \|_{L^2}^2
		+ \nu \int_0^T \|\sqrt{t} \nabla u_t \|_{L^2}^2  \,dt\\
		 \le& C \int_0^T (1 + \| \sqrt\rho u_{t} \|_{L^2}^{2}
		+ \|\nabla^2 u\|_{L^2}^2 ) \,dt + T^{\frac{1}{2}} \varPhi(T)
		\left( \big( \int_0^T \|\nabla^2 u\|_{L^2}^2 \,dt \big)^{\frac{1}{2}}
		+ \big( \int_0^T\| \nabla^3 \chi\|_{L^2} \,dt \big)^{\frac{1}{2}} \right) \\
		 \le& C (1 + T^{\frac{1}{2}} \varPhi(T)) \le C.
	\end{align*}
	The proof is complete.
\end{proof}

\begin{lemma}\label{lem:chi3}
Under the assumptions of Lemma~\ref{lem4} and assuming that \eqref{Assump1} holds, there holds
	\begin{equation*}
		\sup_{0 \le t  \le T} \|\sqrt{t}\nabla\chi_t\|_{L^2}^2
		+ \int_0^T \|\sqrt{t} \rho \chi_{tt} \|_{L^2}^2 \,dt \le  C .
	\end{equation*}
\end{lemma}
\begin{proof}
	Differentiating \eqref{NSAC_34} with respect to $t$, one obtains
	\begin{equation}
		\label{lem:chi3:1}
		2 \rho \rho_t \chi_t + \rho^2 \chi_{tt} + 2 \rho \rho_t u \cdot \nabla \chi
		+ \rho^2 u_t \cdot \nabla \chi + \rho^2 u \cdot \nabla  \chi_t
		= \Delta \chi_t - \rho_t f(\chi) - \rho f'(\chi) \chi_t.
	\end{equation}
	Note that
	\begin{align*}
		- \int_{\Om} \rho_{t} f(\chi) \chi_{tt} dx \,
		&= - \frac{d}{dt} \int_{\Om} \rho_t f(\chi)\chi_t \,dx
		+ \int_{\Om} \rho_{tt} f(\chi) \chi_t \,dx
		+ \int_{\Om} \rho_t ( f'(\chi) \chi_t ) \chi_t \,dx \\
		&= - \frac{d}{dt}\int_{\Om} \rho_t f(\chi)\chi_t \,dx
		- \int_{\Om} \Div (\rho_t u + \rho u_t) f(\chi) \chi_t \,dx
		+ \int_{\Om} \rho_t  f'(\chi)  \chi_t^2 \,dx \\
		&= - \frac{d}{dt}\int_{\Om} \rho_t f(\chi)\chi_t \,dx
		+ \int_{\Om} (\rho_t u + \rho u_t) ( f'(\chi) \nabla\chi \chi_t
		+ f(\chi) \nabla\chi_t ) \,dx \\
		&\quad+ \int_{\Om} \rho_t f'(\chi) \chi_t^2 \,dx.
	\end{align*}
	Multiplying \eqref{lem:chi3:1} by $ \chi_{tt}$ and utilizing the above
	identities, it follows from integration by parts that
	\begin{equation}
		\begin{aligned}\label{lem:chi3:2}
			&\frac{d}{dt} \left( \frac{1}{2}\|\nabla\chi_t\|_{L^2}^2
			+ \int_{\Om} \rho_t f(\chi) \chi_t \,dx \right)
			+ \| \rho \chi_{tt} \|_{L^2}^2 \\
			 =& -2 \int_{\Om}  \rho \rho_t \chi_t \chi_{tt} \,dx
			- 2 \int_{\Om} \rho \rho_t ( u \cdot \nabla\chi ) \chi_{tt} \,dx
			- \int_{\Om}  \rho^2 ( u_t \cdot \nabla \chi ) \chi_{tt} \,dx \\
			&- \int_{\Om}  \rho^2 ( u \cdot \nabla\chi_t ) \chi_{tt} \,dx
			+ \int_{\Om} \rho_t f'(\chi) ( u \cdot \nabla\chi ) \chi_t \,dx
			+ \int_{\Om} \rho_t  f(\chi) (u \cdot \nabla\chi_t ) \,dx \\
			& + \int_{\Om} \rho f'(\chi) ( u_t \cdot \nabla\chi ) \chi_t \,dx
			+ \int_{\Om} \rho f(\chi) ( u_t \cdot \nabla\chi_t) \,dx
			+ \int_{\Om} \rho_t f'(\chi) \chi_t^2  \,dx \\
			& - \int_{\Om} \rho f'(\chi) \chi_t \chi_{tt} \,dx.
		\end{aligned}
	\end{equation}
	Multiplying the above identity with $t$ to obtain
	\begin{equation}\label{lem:chi3:3}
		\begin{aligned}
			&\frac{1}{2} \frac{d}{dt} \|\sqrt{t} \nabla\chi_t\|_{L^2}^2 +
			\| \sqrt{t} \rho\chi_{tt}\|_{L^2}^2 \\
			=& \frac{d}{dt} [ t H(t) ] - H(t) + \frac{1}{2}\|\nabla\chi_t\|_{L^2}
			- 2 t \int_{\Om} \rho \rho_t \chi_t \chi_{tt} \,dx
			- 2 t \int_{\Om} \rho \rho_t ( u \cdot \nabla\chi ) \chi_{tt} \,dx \\
			& - t \int_{\Om}  \rho^2 ( u_t \cdot \nabla \chi ) \chi_{tt} \,dx
			- t \int_{\Om}  \rho^2 ( u \cdot \nabla\chi_t ) \chi_{tt} \,dx
			+ t \int_{\Om} \rho_t f'(\chi) ( u \cdot \nabla\chi ) \chi_t \,dx \\
			& + t \int_{\Om} \rho_t  f(\chi) (u \cdot \nabla\chi_t ) \,dx
			+ t \int_{\Om} \rho f'(\chi) ( u_t \cdot \nabla\chi ) \chi_t \,dx
			+ t \int_{\Om} \rho f(\chi) ( u_t \cdot \nabla\chi_t) \,dx \\
			& + t \int_{\Om} \rho_t f'(\chi) \chi_t^2  \,dx
			- t \int_{\Om} \rho f'(\chi) \chi_t \chi_{tt} \,dx \\
			=& \frac{d}{dt} [ t H(t) ] - H(t)
			+ \frac{1}{2}\|\nabla\chi_t\|_{L^2} + \sum_{i=1}^{10}{L_i},
		\end{aligned}
	\end{equation}
	where $H(t) := - \int_{\Om} \rho_t f(\chi)\chi_t \,dx$. It follows from Lemma \ref{lem1}, Lemma~\ref{lem4}--\ref{lem:u2}, the H\"{o}lder, Poincar\'e, Young and Gagliardo-Nirenberg inequalities that
	\begin{align*}
		| H(t) | &\le \|\rho_t\|_{L^2} \|f(\chi)\|_{L^\infty} \|\chi_t\|_{L^2}\\
		&\le C (\|\rho\chi_{t}\|_{L^2} + \|\nabla\chi_{t}\|_{L^2}) \\
		&\le C (1 + \|\nabla\chi_{t}\|_{L^2} ), \\
		L_1 + L_2 &\le \sqrt{t} \|\sqrt{t} \rho \chi_{tt}\|_{L^2} \|\rho_t\|_{L^3}
		( \| \chi_{t} \|_{L^6} + \|u\|_{L^\infty} \|\nabla\chi\|_{L^6} ) \\
		&\le \sqrt{t} \|\sqrt{t} \rho \chi_{tt}\|_{L^2} \|\nabla^2 u\|_{L^2}^{\frac{1}{2}}
		(\|\rho\chi_{t}\|_{L^2} + \|\nabla\chi_{t}\|_{L^2}
		+ \|u\|_{L^\infty} \|\nabla^2\chi\|_{L^2} )   \\
		&\le \sqrt{t} \|\sqrt{t} \rho \chi_{tt}\|_{L^2}  \|\nabla^2 u\|_{L^2}^{\frac{1}{2}}
		(1 + \|\nabla\chi_t\|_{L^2} + \|\nabla^2 u\|_{L^2}^{\frac{1}{2}} )   \\
		&\le \frac{1}{8} \|\sqrt{t} \rho \chi_{tt}\|_{L^2}^2
		+ C \|\nabla^2 u\|_{L^2} (1 + \|\sqrt{t}\nabla\chi_t\|_{L^2}^2)
		+ C \|\nabla^2 u\|_{L^2}^2, \\
		L_3 &\le \|\rho\|_{L^\infty}^{\frac{1}{2}} \|\sqrt{t} \sqrt\rho u_{t}\|_{L^2}
		\|\nabla\chi\|_{L^\infty} \|\sqrt{t} \rho \chi_{tt}\|_{L^2} \\
		&\le C \|\nabla^2\chi\|_{L^2}^{\frac{1}{2}} \| \nabla^3 \chi \|_{L^2}^{\frac{1}{2}}
		\|\sqrt{t} \rho \chi_{tt}\|_{L^2} \\
		& \le \frac{1}{8} \|\sqrt{t} \rho \chi_{tt}\|_{L^2}^2
		+ C \|\nabla^3\chi\|_{L^2},   \\
		L_4 &\le \|\sqrt{t} \rho\chi_{tt}\|_{L^2} \|\rho\|_{L^\infty}
		\|u\|_{L^\infty} \|\sqrt{t} \nabla\chi_t\|_{L^2} \\
		&\le C \| \rho \chi_{tt} \|_{L^2} \|\nabla u\|_{L^2}^{\frac{1}{2}}
		\|\nabla^2 u\|_{L^2}^{\frac{1}{2}} \|\sqrt{t} \nabla\chi_t\|_{L^2} \\
		&\le \frac{1}{8} \|\sqrt{t} \rho\chi_{tt}\|_{L^2}^2
		+ C \|\nabla^2 u\|_{L^2} \|\sqrt{t} \nabla\chi_t\|_{L^2}^2, \\
		L_5 + L_6 &\le C \sqrt{t} \|\rho_t\|_{L^2} \|u\|_{L^\infty}
		( \|f'(\chi)\|_{L^6} \|\nabla\chi\|_{L^6} \|\sqrt{t} \chi_t\|_{L^6}
		+ \|f(\chi)\|_{L^\infty} \|\sqrt{t} \nabla\chi_t\|_{L^2})  \\
		&\le C \|\nabla u\|_{L^2}^{\frac{1}{2}} \|\nabla^2 u\|_{L^2}^{\frac{1}{2}}
		(\|\rho\chi_{t}\|_{L^2} + \|\sqrt{t} \nabla\chi_t\|_{L^2} )     \\
		&\le C(1 + \|\nabla^2 u\|_{L^2}   + \|\sqrt{t} \nabla\chi_t\|_{L^2}^2 ),  \\
		L_7 +L_8&\le C \|\rho\|_{L^\infty}^{\frac{1}{2}} \|\sqrt{t} \sqrt\rho u_{t}\|_{L^2}
		( \|f'(\chi)\|_{L^6} \|\nabla\chi\|_{L^6} \|\sqrt{t} \chi_t\|_{L^6}
		+ \|f(\chi)\|_{L^\infty} \|\sqrt{t} \nabla\chi_t\|_{L^2})  \\
		&\le C (\|\rho\chi_{t}\|_{L^2}+  \|\sqrt{t} \nabla\chi_t\|_{L^2} )     \\
		&\le C (1  + \|\sqrt{t} \nabla\chi_t\|_{L^2}^2 ),
	\end{align*}
	and
\begin{align*}
		L_9 + L_{10}&\le C (\|f'(\chi)\|_{L^6} \|\rho_t\|_{L^2}  \|\sqrt{t} \chi_t\|_{L^6}^2
		+ \|f'(\chi)\|_{L^3} \|\sqrt{t}\rho\chi_{tt}\|_{L^2} \|\sqrt{t}\chi_{t}\|_{L^6}) \\
		&\le \frac{1}{8} \|\sqrt{t} \rho\chi_{tt}\|_{L^2}^2
		+ C (\|\rho\chi_t\|_{L^2}^2 + \|\sqrt{t} \nabla\chi_t\|_{L^2}^2) \\
		& \le \frac{1}{8} \|\sqrt{t} \rho\chi_{tt}\|_{L^2}^2
		+ C (1 + \|\sqrt{t} \nabla\chi_t\|_{L^2}^2).
	\end{align*}
	Plugging the above estimates into \eqref{lem:chi3:3} and using the Young inequality, one obtains
	\begin{equation}
		\begin{aligned} \label{lem:chi3:4}
			& \frac{d}{dt} \|\sqrt{t} \nabla\chi_t\|_{L^2}^2 + \|\sqrt{t} \rho \chi_{tt}\|_{L^2}^2 \\
			 \le& 2 \frac{d}{dt} [ t H(t) ] + C (1 + \|\nabla\chi_t\|_{L^2}^2
			+ \|\nabla^2 u\|_{L^2}^2 + \| \nabla^3 \chi \|_{L^2}^2 )
			+ C \|\nabla^2 u\|_{L^2} \|\sqrt{t}\nabla\chi_t\|_{L^2}^2.
		\end{aligned}
	\end{equation}
	By Lemma~\ref{lem1}, Lemma~\ref{lem:chi1}, Corollary \ref{cor2}, and Sobolev
	inequality, we have
		\begin{equation}
		\begin{aligned}\label{lem:chi3:5}
			2 t H(t) &\le 2 \sqrt{t} \|\rho_t\|_{L^2} \|f(\chi)\|_{L^3} \|\sqrt{t} \chi_t\|_{L^6}\\
			&\le \frac{1}{2} \|\sqrt{t} \nabla \chi_t\|_{L^2}^2 + C.
		\end{aligned}
	\end{equation}
	Integrating \eqref{lem:chi3:4} over $(0, T)$ and using the H\"older inequality, one deduces by \eqref{A2}, \eqref{lem:chi3:5}, and Lemma~\ref{lem:u1} that
	\begin{align*}
		&\sup_{0 \le t  \le T} \|\sqrt{t} \nabla\chi_t\|_{L^2}^2
		+ \int_0^T \|\sqrt{t} \rho \chi_{tt}\|_{L^2}^2 \,dt\\
		 \le& 2 \sup_{0 \le t  \le T} \left( t H(t) \right) + C(T+\varPhi(T))
		+ T^{\frac{1}{2}} \varPhi(T) \left( \int_0^T \|\nabla^2 u\|_{L^2}^2 \,dt \right)^{\frac{1}{2}} \\
		 \le& \frac{1}{2} \sup_{0 \le t  \le T} \|\sqrt{t} \nabla \chi_t\|_{L^2}^2 + C.
	\end{align*}
	Therefore, one has
	$$\sup_{0 \le t  \le T} \|\sqrt{t}\nabla\chi_t\|_{L^2}^2
	+ \int_0^T \|\sqrt{t} \rho \chi_{tt} \|_{L^2}^2 \,dt \le  C .$$
	The proof is complete.
\end{proof}

Finally, we summarize the prior estimates obtained previously and derive some additional estimates essential for proving uniqueness.
\begin{proposition}
	Under the assumptions of Lemma~\ref{lem4} and assuming that \eqref{Assump1} holds, there holds
	\begin{equation*}
		\begin{aligned}
			&\varPhi(T) + \sup_{0 \le t  \le T} \left( \|\rho_t\|_{L^2}^2
			+ \| \sqrt{t} \nabla^2 u \|_{L^2}^2
			+ \| \sqrt{t} \nabla^3\chi \|_{L^2}^2 + \|\mu\|_{L^2}^2 \right) \\
			&\quad+ \int_0^T \left( \|\nabla^2 u\|_{L^q} + \|\sqrt{t} \nabla^2 u\|_{L^q}^2
			+ \|\chi_t\|_{H^1}^2
			+ \| \sqrt{t} \rho \chi_{tt} \|_{L^2}^2\right) \,dt \le C.
		\end{aligned}
	\end{equation*}
\end{proposition}

\begin{proof}
	It follows from  Lemma~\ref{lem3}--\ref{lem:chi3} and Corollary~\ref{cor2} that
	\begin{equation*}
		\varPhi(T) + \sup_{0 \le t  \le T} (\|\rho_t\|_{L^2}^2
		+ \|\mu\|_{L^2}^2) + \int_0^T \left( \|\nabla^2 u\|_{L^q}
		+ \|\sqrt{t} \rho \chi_{tt}\|_{L^2}^2  \right)\,dt \le C.
	\end{equation*}
	Recalling \eqref{nabla_cube_chi} and \eqref{lem:u1:1}, one deduces by the Young inequality that
	\begin{align*}
		\|\nabla^3\chi\|_{L^2}^2 \leq C(1 + \| \nabla\chi_{t}\|_{L^2}^2  + \|  \sqrt{\rho} u_t
		\| _{L^2}^2),
	\end{align*}
	which further indicates that
	\begin{align*}
		\sup_{0 \le t  \le T} \| \sqrt{t} \nabla^3\chi \|_{L^2}^2
		&\le C \sup_{0 \le t  \le T} \left(1
		+ \| \sqrt{t} \nabla\chi_{t}\|_{L^2}^2
		+ \| \sqrt{t} \sqrt{\rho} u_t
		\| _{L^2}^2 \right)\\
		&\le C\left(1 + \varPhi(T) \right) \le C.
	\end{align*}
	Moreover, one has
	\begin{align*}
		\sup_{0 \le t  \le T} \| \sqrt{t} \nabla^2 u \|_{L^2}^2
		&\le C \sup_{0 \le t  \le T} \left( \| \sqrt{t} \sqrt{\rho} u_t
		\| _{L^2}^2 + \|\nabla u\|_{L^2}^6 + \|\nabla^2\chi\|_{L^2}^3
		\| \sqrt{t} \nabla^3 \chi\|_{L^2} + 1 \right) \\
		&\le  C \varPhi(T) + C \varPhi^3(T) + C \varPhi^{\frac{3}{2}}
		\sup_{0 \le t  \le T} \| \sqrt{t} \nabla^3\chi \|_{L^2}^2 \le C.
	\end{align*}
	One also deduces by Lemma~\ref{lem1} and $\varPhi(T) \le C$ that
	$$\int_0^T \|\chi_t\|_{H^1}^2 \,dt \le C. $$
	Besides, it follows from the $W^{2, q}-$estimates of $\eqref{NSAC}_2$,
	the H\"older, Poincar\'e, Sobolev inequalities that
		\begin{equation} \label{nabla2uq}
			\begin{aligned}		
				\|\nabla^{2}u\|_{L^q}&\le C ( \|\rho u_{t}\|_{L^q}
				+ \|\rho ( u \cdot \nabla ) u\|_{L^q}
				+ \|\nabla (P(\rho))\|_{L^q} + \|\Delta \chi \nabla \chi\|_{L^q}) \\
				&\le C (\|\rho\|_{L^\infty} \|u_t\|_{L^6}
				+ \|\rho\|_{L^\infty} \|u\|_{L^\infty} \|\nabla u\|_{L^6} \\
				&\quad+ \|\nabla \rho\|_{L^q} \|\rho\|_{L^\infty}^{\gamma-1}
				+ \|\nabla \chi\|_{L^\infty} \|\nabla^2\chi\|_{L^q}) \\
				&\le C (1 + \|\nabla u_t\|_{L^2} + \|\nabla^2 u\|_{L^2}^{\frac{3}{2}}
				+ \|\nabla^3 \chi\|_{L^2}^{\frac{3}{2}}).
			\end{aligned}
		\end{equation}
	Consequently, one has
	\begin{equation*}
		\begin{aligned}
			\int_0^T \|\sqrt{t} \nabla^2 u\|_{L^q}^2 \,dt
			&\le C \int_0^T (1 + \|\sqrt{t} \nabla u_t\|_{L^2}^2
			+ \|\sqrt{t} \nabla^2 u\|_{L^2} \|\nabla^2 u\|_{L^2}^2 + \|\sqrt{t}\nabla^3 \chi\|_{L^2} \|\nabla^3 \chi\|_{L^2}^2) \,dt \\
			&\le C (1 + \varPhi(T)) \le C.
		\end{aligned}
	\end{equation*}
	The conclusion follows directly from the estimates above.
\end{proof}

\vskip6mm
\section{Galerkin scheme}\label{Galerkin scheme}
Let $\{w_i\}_{i=1}^\infty$ be the eigenfunctions associated with the eigenvalues
$\{\lambda_i\}_{i=1}^\infty$ of the operator $\mathcal{L}$, solving
\begin{align*}
	\begin{cases}
		- \mathcal{L} w_i = \lambda_i w_i \quad &\text{in } \Omega,\\
		w_i = 0 \quad &\text{on } \partial \Omega.
	\end{cases}
\end{align*}
Here $0<\lambda_1\le\lambda_2\le\cdots$ and $\lambda_i\to\infty$ as $i\to\infty$.
The family $\{w_i\}$ can be chosen as an orthonormal basis of $L^2(\Omega)$ and
an orthogonal basis of $H_0^1(\Omega)$.
For any $N\in\mathbb{N}_+$, we define
$$
X_N := \mathrm{span}\{w_1,w_2,\ldots,w_N\},
$$
and denote by $\|\cdot\|_{X_N}$ the induced norm on $X_N$. Let
$$
\rho_{0N} = \rho_0 * \eta_{\frac1N} + \frac1N,
\qquad
h_N = h * \eta_{\frac1N} - C_N,
$$
where $\eta_{\frac1N}$ is a standard mollifier
and $C_N = \frac{\overline{\rho_{0N} h * \eta_{\frac1N}}}{\overline{\rho_{0N}}}$. Here $\overline{f}:= \frac1{|\Omega|}\int_\Om f\,dx$
denotes the spatial average over $\Om$ of an integrable function $f$.
Then we have
\begin{align}
	\rho_{0N} \to \rho_0 \quad &\text{in } W^{1,q}, \label{rho0n}\\
	h_N \to h \quad &\text{in } L^2, \notag \\
	\overline{\rho_{0N} h_N} = 0. \notag
\end{align}

By classical elliptic theory (see, e.g., \cite{BH74,LM68}), there exists a
$\chi_{0N}\in H^2$ satisfying
\begin{align*}
	\begin{cases}
		\Delta \chi_{0N} = \rho_{0N} h_N - \rho_0 h \quad &\text{in } \Omega,\\
		\partial_{\boldsymbol n}\chi_{0N}
		= 0 \quad &\text{on } \partial \Omega,
	\end{cases}
\end{align*}
and
\begin{align}\label{chi0n}
	\chi_{0N} \to \chi_0 \quad \text{in } H^2.
\end{align}
Let $u_{0N}=P_Nu_0$ be the $L^2$-projection of $u_0$ onto $X_N$, where $P_N$ is defined by
\[
P_N g := \sum_{i=1}^N \langle g,w_i\rangle w_i,
\qquad g\in L^2.
\]
It is straightforward to verify that
\begin{align}\label{u0n}
	u_{0N} \to u_0 \quad \text{in } H^1.
\end{align}
We are now ready to formulate the Galerkin approximate system:
\begin{equation}\label{appro_pro}
	\begin{cases}
		\partial_t\rho_N + \Div(\rho_N u_N) = 0,\\
		\rho_N\partial_t u_N
		+ P_N\!\left[\rho_N(u_N\!\cdot\!\nabla)u_N
		+ \nabla P(\rho_N)\right]
		= \mathcal{L}u_N
		- P_N\!\left[\Div\!\left(\nabla\chi_N\otimes\nabla\chi_N
		-\frac{|\nabla\chi_N|^2}{2}\mathbb{I}\right)\right],\\
		\rho_N\partial_t\chi_N + \rho_N u_N\cdot\nabla\chi_N = -\mu_N,\\
		\rho_N\mu_N = -\Delta\chi_N + \rho_N F'(\chi_N),\\
		(\rho_N,u_N,\chi_N)|_{t=0} = (\rho_{0N},u_{0N},\chi_{0N}),\\
		(u_N,\partial_{\boldsymbol n}\chi_N)|_{\partial\Omega} = (0,0).
	\end{cases}
\end{equation}

\subsection{Existence of approximate solutions}\label{Sec_3.1}
By Schaefer's fixed point theorem (cf.~\cite[Section 9.2.2, Theorem 4]{Evans10}), the well-posedness of system \eqref{appro_pro} can be established in a standard manner.

\begin{proposition}\label{appro_sol}
	For any $T>0$ and $N\in \mathbb{N}_+$, there exists some $(\rho_N, u_N, \chi_N, \mu_N)$ solving \eqref{appro_pro}, such that:
	\begin{itemize}
		\item[(a)] $\rho_N \in C^1(\overline{\Om} \times[0,T])$ and satisfies
		\begin{equation*}
			\partial_t \rho_N + \Div(\rho_N u_N) = 0 \quad \text{\rm in } \Om\times(0,T); \\
		\end{equation*}
		\item[(b)] $u_N \in C([0, T]; X_N)$ and, for every $\varphi_N \in X_N$ and almost every $t \in (0, T)$, satisfies
		\begin{equation}\label{Galerkin_Equation}
			\begin{split}
				&\langle \rho_N \partial_t u_N + \rho_N (u_N\cdot\nabla) u_N + \nabla (P(\rho_N)), \varphi_N \rangle + \nu \langle  \nabla u_N, \nabla \varphi_N \rangle + (\lambda + \nu) \langle \Div u_N, \Div \varphi_N \rangle \\
				=&\langle\Div(\nabla \chi_N \otimes \nabla \chi_N - \frac{|\nabla \chi_N|^2}  {2}\mathbb{I}), \varphi_N \rangle;
			\end{split}
		\end{equation}
		\item[(c)] $\chi_N \in C([0, T]; H^1) \cap L^{\infty} (0, T; H^2) \cap L^2(0,T;H^3)$, with
		$\partial_t \chi_N \in L^\infty ( 0, T; H^1)$, and satisfies
		\begin{equation*}
			\rho_N^2 \partial_t \chi_N + \rho_N^2 u_N \cdot \nabla \chi_N - \Delta \chi_N
			+ \rho_N F'(\chi_N) = 0 \quad \text{\rm a.e. in } \Om \times (0, T);
		\end{equation*}
		\item[(d)] the initial and boundary conditions hold
		$$
		(\rho_N(0), u_N(0), \chi_N(0))= (\rho_{0N}, u_{0N}, \chi_{0N}) \quad \text{a.e. in } \Om,
		$$
		and
		$$
		\partial_{\boldsymbol{n}} \chi_N = 0 \quad \text{a.e. on } \partial \Om \times (0, T);
		$$
		\item[(e)] the following energy equality holds for almost every $t \in (0, T)$
		\begin{equation*}
			\begin{aligned}
				E_N(t) + \int_0^t ( \| \mu_N \|_{L^2}^2
				+ \nu \| \nabla u_N \|_{L^2}^2
				+ ( \lambda + \nu ) \| \Div u_N \|_{L^2}^2 ) \, ds = E_N(0),
			\end{aligned}
		\end{equation*}
		where
		\begin{equation}\label{EN}
			E_N(t) = \frac{1}{2} \| \sqrt{\rho_N} u_N \|_{L^2}^2
			+ \frac{A}{\gamma-1} \int_{\Om} \rho_N^{\gamma} \,dx
			+ \frac{1}{2} \| \nabla\chi_N \|_{L^2}^2
			+ \int_{\Om} \rho_N F(\chi_N)\,dx.
		\end{equation}
	\end{itemize}
\end{proposition}

For the sake of completeness, the proof of Proposition \ref{appro_sol} is given in the Appendix.

\subsection{Uniform bounds for the approximate solutions}\label{Sec_3.2}
Let $(\rho_N, u_N, \chi_N, \mu_N)$ be the strong solutions in Proposition~\ref{appro_sol}. Set
\begin{equation*}
	\begin{aligned}
		\varPhi_N(t) &= \sup_{0 \le s \le t} \left(\|\rho_N\|_{W^{1, q}}
		+ \|\big(\nabla u_N, \sqrt{s} \sqrt{\rho_N}  \partial_t u_N, \nabla^2 \chi_N,
		\rho_N \partial_t \chi_N, \sqrt{s} \nabla \partial_t \chi_N\big) \|^2_{L^2}\right) \\
		&\quad+ \int_0^t \left(\|\big(\nabla^2 u_N, \sqrt{\rho_N} \partial_t u_N,
		\sqrt{s} \nabla \partial_t u_N, \nabla^3 \chi_N,
		\nabla \partial_t \chi_N \big)\|^2_{L^2}\right) \,ds + 1.
	\end{aligned}
\end{equation*}

\begin{proposition}\label{Uniform_bounds}
	Let $\varepsilon_0$ be the constant given in Lemma \ref{lem4}. Then, there exists a time $T_0>0$
	and a constant $C>0$, depending only on $A$, $\gamma$, $\nu$,
	$\lambda$, $q$, $\Om$ and $\Phi_0$, such that
	\begin{equation*}
		T_0^{\frac{6-q}{4q}} \varPhi_N^{\gamma + 1}(T_0) \le \varepsilon_0.
	\end{equation*}
	Moreover, we have
	\begin{equation*}
		\begin{aligned}
			&\varPhi_N(T_0) +  \sup_{0  \le t  \le T_0} \left(\|\partial_t \rho_N\|_{L^2}^2
			+ \|\sqrt{t} \nabla^2 u_N\|_{L^2}^2 + \|\sqrt{t} \nabla^3 \chi_N\|_{L^2}^2
			+ \|\mu_N\|_{L^2}^2\right) \\
			\quad& + \int_0^{T_0} \left(\|\nabla^2 u_N\|_{L^q}
			+ \|\sqrt{t} \nabla^2 u_N\|_{L^q}^2 + \|\partial_t \chi_N\|_{H^1}^2
			+ \|\sqrt{t} \rho_N \partial_{tt}^2 \chi_N\|^2_{L^2}\right) \,dt \le C
		\end{aligned}
	\end{equation*}
	for all $N \in \mathbb{N}$.
\end{proposition}

\begin{proof} Taking $T = 1$ in Proposition~\ref{appro_sol} and denoting
	\begin{equation*}
		\Psi_N(t) = t^{\frac{6-q}{4q}} \varPhi_N^{\gamma + 1 }(t).
	\end{equation*}
	Since $\varepsilon_0 \in (0,1)$ and $\Psi_N$ is continuous on $[0,1]$,
	there exists some $T_N \in (0,1)$ such that
	$\Psi_N(T_N) = \varepsilon_0$. In a manner similar to that in Section \ref{section priori}, one can show that $\varPhi_N(T_N) \le C$ uniformly in $N$. By the definition of $\Psi_N$, we have
	\begin{align*}
		T_N &= \left( \varPhi_N(T_N)^{ -\gamma - 1 } \varepsilon_0 \right)^{\frac{4q}{6-q}} \\
		&\ge \left( C^{ -\gamma - 1 } \varepsilon_0 \right)^{\frac{4q}{6-q}},
	\end{align*}
	which implies that ${T_N}$ admits a uniform positive lower bound, independent of $N$.
	
	Now set $T_0 = \inf\limits_{N \in \mathbb{N}}{T_N}$, then $\varPhi_N(T_0) \le C$. Moreover, as in Section \ref{section priori}, we can show that
	\begin{align*}
		\int_0^{T_0} \| \nabla^2 u_N \|_{L^q} \,dt +  \sup_{0  \le t  \le T_0} \left( \| \partial_t \rho_N \|_{L^2}^2
		+ \| \sqrt{t} \partial_t \nabla \chi_N \|_{L^2}^2 + \| \sqrt{t} \nabla^3 \chi_N \|_{L^2}^2
		+ \| \sqrt{t} \nabla^2 u_N \|_{L^2}^2 \right) \le C.
	\end{align*}
	However, the upper bound of $\int_0^{T_0} \|\nabla^2 u\|_{L^q} \,dt$ is estimated
	slightly different to that in Section \ref{section priori}. To be specific, taking
	$\varphi_N = P_N( | \mathcal{L} u_N|^{q - 2} \mathcal{L} u_N )$ in \eqref{Galerkin_Equation} and using
	integration by parts, we have
	\begin{align*}
		&\| \nabla^2 u_N \|_{L^q}^q \le C \| \mathcal{L}u_N \|_{L^q}^q
		= C \langle \mathcal{L} u_N, \varphi_N \rangle \\
		&\quad= C \langle \rho_N \partial_t u_N + \rho_N (u_N\cdot\nabla) u_N + \nabla (P(\rho_N))
		+ \Delta \chi_N \nabla \chi_N , \varphi_N \rangle\\
		&\quad\le C (\|\rho_N \partial_t u_N \|_{L^q} + \| \rho_N (u_N\cdot\nabla) u_N \|_{L^q}
		+ \| \nabla (P(\rho_N)) \|_{L^q} + \| \Delta \chi_N \nabla \chi_N \|_{L^q}
		) \| \varphi_N \|_{L^{\frac{q}{q - 1}}}\\
		&\quad\le C (\|\rho_N \partial_t u_N \|_{L^q} + \| \rho_N (u_N\cdot\nabla) u_N \|_{L^q}
		+ \| \nabla (P(\rho_N)) \|_{L^q} + \| \Delta \chi_N \nabla \chi_N \|_{L^q}
		) \| \nabla^2 u_N \|_{L^q}^{q - 1}.
	\end{align*}
	That is,
	\begin{align*}
		\| \nabla^2 u_N \|_{L^q} &\le
		C (\|\rho_N \partial_t u_N \|_{L^q} + \| \rho_N (u_N\cdot\nabla) u_N \|_{L^q}
		+ \| \nabla (P(\rho_N)) \|_{L^q} + \| \Delta \chi_N \nabla \chi_N \|_{L^q} ).
	\end{align*}
	The rest are analogous to that in Section \ref{section priori}.
\end{proof}

\vskip6mm
\section{Proof of Theorem 1.1}\label{Proof of Theorem 1.1}

\subsection{Convergence of the approximate solutions}\label{Sec_4.1}
In this section, we prove the convergence of the approximate solutions in Proposition~\ref{appro_sol} and derive the regularity of the limit, which will be used in the proof of Theorem~\ref{th1}.
\begin{theorem}\label{Existence}
	Assume that all the conditions of Theorem~\ref{th1} are satisfied. Then, there exists a positive time $T_0>0$, depending only on $A$, $\gamma$, $\nu$,
	$\lambda$, $q$, $\Om$ and $\Phi_0$, such that the system \eqref{NSAC}, subject to \eqref{I}-\eqref{B1}, in $\Om \times (0, T_0)$, admits a solution $(\rho, u, \chi, \mu)$, satisfying all the properties listed in Definition~\ref{def1},
	except that the property $\sqrt{\rho}u \in C([0, T_0];L^2)$ is replaced by
	\begin{align*}
		\sqrt{\rho}u \in C((0, T_0];L^2),\quad \rho u \in C([0, T_0];L^2),
	\end{align*}
	and
	\begin{equation}\label{normcon}
		\lim_{t \rightarrow 0^+} \| \sqrt{\rho} u \|_{L^2}(t) = \| \sqrt{\rho_0} u_0 \|_{L^2}.
	\end{equation}
\end{theorem}
\begin{proof}\textbf{Step 1. Limit passage. }
	Let $(\rho_{0N}, u_{0N}, \chi_{0N})$ be chosen as in Section~3, then we have
	\begin{align*}
		&\rho_{0N} \rightarrow \rho_0 \qquad \mathrm{in}\ W^{1,q}, \\
		&u_{0N} \rightarrow u_0 \qquad \mathrm{in}\ H^1, \\
		&\chi_{0N} \rightarrow \chi_0 \qquad \mathrm{in}\  H^2.
	\end{align*}
	It follows from Proposition~\ref{appro_sol} and Proposition~\ref{Uniform_bounds} that, there exists a time $T_0>0$ and a constant $C>0$, depending only on $A$, $\gamma$, $\nu$,
	$\lambda$, $q$, $\Om$ and $\Phi_0$, such that
		\begin{align*}
			&\sup\limits_{0 \le t \le T_0} \left(\|\partial_t \rho_N\|_{L^2}^2
			+ \|\rho_N\|_{W^{1,q}} \right) \le C, \\
			&\sup\limits_{0 \le t \le T_0} \left(\|\nabla u_N\|_{L^2}^2 + \|\sqrt{t} \nabla^2 u_N\|_{L^2}^2 \right)
			+ \int_0^{T_0} \left(\| \sqrt{t} \nabla \partial_t  u_N \|_{L^2}^2
			+ \| \nabla^2 u_N \|_{L^2}^2 \right) \,dt \le C,\\
			&\sup\limits_{0 \le t \le T_0} \left( \| \chi_N \|_{H^2}^2
			+ \| \sqrt{t} \nabla^3 \chi_N \|_{L^2}^2 \right)
			+ \int_0^{T_0} \| \partial_t \chi_N \|_{H^1}^2 \,dt\le C,\\
			&\sup\limits_{0 \le t \le T_0} \left( \| \rho_N \partial_t \chi_N\|_{L^2}^2
			+ \|\mu_N\|_{L^2}^2 + \| \sqrt{t} \partial_t \nabla\chi_N\|_{L^2}^2
			+ \| \sqrt{t} \sqrt{\rho_N} \partial_t u_N \|_{L^2}^2 \right) \le C, \\
			&\int_0^{T_0} \left( \|\sqrt{\rho_N} \partial_t u_N\|_{L^2}^2
			+ \| \nabla^2 u_N \|_{L^q}
			+ \|\sqrt{t} \nabla^2 u\|_{L^q}^2 + \|\nabla^3 \chi_N \|_{L^2}^2
			+ \| \sqrt{t} \rho_N \partial_{tt}^2 \chi_N \|_{L^2}^2 \right) \,dt \le C,
		\end{align*}
	for any $N \in \mathbb{N}$. Using the standard diagonal argument, we can extract a subsequence of
	$(\rho_N, u_N, \chi_N, \mu_N)$, still denoted by $(\rho_N, u_N, \chi_N, \mu_N)$, such that
	\begin{equation}\label{wc1}
		\begin{aligned}
			\begin{aligned}
				&\rho_N \wsc \rho &&\ws L^{\infty}(0, T_0; W^{1,q}), \\
				&\partial_t \rho_N \wsc \rho_t &&\ws L^{\infty}(0, T_0; L^2), \\
				&u_N \wsc u &&\ws L^{\infty}(0, T_0; H_0^1),\\
				&u_N \wc u &&\w L^2(0, T_0; H^2),\\
				&\partial_t u_N \wc u_t &&\w L^2(\delta, T_0; H^1_0),\\
			\end{aligned}
			\qquad
			\begin{aligned}
				&\chi_N \wsc \chi &&\ws L^{\infty}(0, T_0; H^2),\\
				&\chi_N \wc \chi &&\w L^2(0, T_0; H^3),\\
				&\partial_t \chi_N \wc \chi_t &&\w L^2(0, T_0; H^1),\\
				&\partial_t \chi_N \wc \chi_t &&\ws L^\infty(\delta, T_0; H^1),\\
				&\mu_N  \wsc \mu  && \ws L^\infty(0, T_0; L^2),
			\end{aligned}
		\end{aligned}
	\end{equation}
	for any $\delta \in (0, T_0)$. Here $(\rho, u, \chi, \mu)$ satisfies
	\begin{equation}\label{reg}
		\begin{cases}
			0 \le \rho \in L^{\infty}(0, T_0; W^{1, q}),
			\quad \rho_t \in L^\infty(0, T_0; L^2), \\
			u \in L^\infty(0, T_0; H_0^1) \cap L^2(0, T_0; H^2) \cap L^1(0, T_0; W^{2,q}),
			\quad \sqrt{t} u_t \in L^2(0, T_0; H_0^1), \\
			\sqrt{t} \nabla^2 u \in L^\infty(0, T_0;L^2) \cap L^2(0, T_0; L^q),  \\
			\chi \in L^{\infty}(0, T_0; H^2) \cap L^2(0, T_0; H^{3}),
			\quad \sqrt{t} \nabla^3 \chi \in L^{\infty}(0, T_0; L^2), \\
			\chi_t \in L^2(0, T_0; H^1), \quad \sqrt{t} \nabla\chi_t \in L^\infty(0, T_0; L^2), \\
			\mu \in L^\infty(0, T_0; L^2).
		\end{cases}
	\end{equation}
	Therefore, by the Aubin-Lions lemma, together with the embeddings $H^3 \hookrightarrow \hookrightarrow H^2 \hookrightarrow\hookrightarrow H^1\hookrightarrow \hookrightarrow L^2$, and $ W^{1, q} \hookrightarrow \hookrightarrow C(\overline{\Om}) \hookrightarrow L^2$, for $q \in (3, 6)$,
	we infer that
	\begin{align}
		\rho_N \rightarrow \rho &\s C([0, T_0]; C(\overline{\Om})), \label{sc1}\\
		u_N \rightarrow u &\s C([\delta, T_0]; L^2) \cap
		L^2(\delta, T_0; H_0^1), \label{sc2} \\
		\chi_N \rightarrow \chi &\s C([0, T_0]; H^1) \cap L^2(0, T_0; H^2) \label{sc3}
	\end{align}
	for any $\delta \in (0, T_0)$. By the convergence of the nonlinear terms \eqref{wc1} and \eqref{sc1}--\eqref{sc3}, we have the following convergence of the nonlinear terms
	\begin{equation}\label{cnl}
		\begin{aligned}
			(\rho_N u_N, \sqrt{\rho_N} u_N) &\rightarrow (\rho u,\sqrt{\rho} u)
			&&\s C([\delta, T_0]; L^2), \\
			(\rho_N\partial_t u_N ,\sqrt{\rho_N} \partial_t u_N)
			&\wc (\rho u_t ,\sqrt{\rho} u_t) &&\w L^2(\delta, T_0; L^2), \\
			\rho_N ( u_N \cdot \nabla ) u_N &\rightarrow \rho (u \cdot \nabla) u
			&& \s L^1(\Om \times (\delta, T_0)), \\
			P(\rho_N) &\rightarrow P(\rho) &&\s C([0, T_0]; C(\overline{\Om})), \\
			\Delta\chi_N \nabla\chi_N &\rightarrow \Delta\chi \nabla\chi &&\s
			L^1(\Om \times (0, T_0)). \\
			(\rho_N^2 \partial_t\chi_N, \rho_N \partial_t\chi_N) &\wc
			(\rho^2 \partial_t\chi, \rho \partial_t\chi) &&\w L^2(0, T_0; L^2), \\
			\rho_N^2 ( u_N \cdot \nabla ) \chi_N &\wc \rho^2 ( u \cdot \nabla ) \chi
			&&\w L^1(\Om \times (\delta, T_0)), \\
			\rho_N F'(\chi_N) &\rightarrow \rho F'(\chi) &&\s C([0, T_0]; L^2)
		\end{aligned}
	\end{equation}
	for any $\delta \in (0, T_0)$. From the weakly lower semi-continuity of the norms and \eqref{cnl}, one concludes that
	\begin{align*}
		\int_{\delta}^{T_0} \| \sqrt \rho u_t \|_{L^2}^2 \,dt
		\le \liminf_{N \rightarrow \infty} \int_{\delta}^{T_0} \| \sqrt{\rho_N}
		\partial_t u_N \|_{L^2}^2 \,dt \le C,
	\end{align*}
	for any $\delta \in (0,T_0)$.
	Thus, we have
	\begin{align*}
		\sqrt{\rho} u_t \in L^2(0, T; L^2).
	\end{align*}
	
	\textbf{Step 2. The existence. }
	It follows from \eqref{wc1}, \eqref{sc2}, \eqref{sc3}, \eqref{cnl}, and Proposition~\ref{appro_sol} that $(\rho, u, \chi, \mu)$ satisfies \eqref{NSAC} in the sense of distributions and also satisfies \eqref{B1} almost everywhere on $\partial \Omega \times (0, T_0)$. Furthermore, in view of the regularity property \eqref{reg} and an integration by parts argument, we conclude that $(\rho, u, \chi, \mu)$ satisfies \eqref{NSAC} almost everywhere in $\Omega \times (0, T_0)$.
	Thus, we have shown that $(\rho, u, \chi, \mu)$ satisfies all properties in Definition~\ref{def1}, except $\sqrt{\rho}u \in C([0, T_0];L^2)$ and the initial conditions \eqref{I}.

	It follows from \eqref{rho0n}, \eqref{chi0n}, \eqref{sc1} and \eqref{sc3} that
	\begin{equation}
		\begin{cases}\label{CRC}
			\rho \rightarrow \rho_0 \quad \text{in } C(\overline{\Om})
			\quad\text{as } t \rightarrow 0^+, \\
			\chi \rightarrow \chi_0 \quad \text{in } H^1
			\quad\text{as } t \rightarrow 0^+. \\
		\end{cases}
	\end{equation}
	We also deduce by \eqref{cnl} that
	\begin{align*}
		&(\sqrt{\rho}u, \rho u) \in C((0, T];L^2).
	\end{align*}
	By letting $N \rightarrow \infty$ in \eqref{EN} and \eqref{sc1}--\eqref{sc3}, one has
	\begin{align*}
		&\frac{1}{2}\|\sqrt{\rho} u \|_{L^2}^2(t) + \frac{1}{2} \|\nabla \chi\|_{L^2}^2 (t)
		+ \frac{A}{\gamma-1} \int_{\Om} \rho^{\gamma}(x, t) \,dx
		+ \int_{\Om} \rho(x, t) F(\chi(x, t)) \,dx \\
		&\quad + \int_0^t ( \|\mu\|_{L^2}^2 + \nu \|\nabla u\|_{L^2}^2
		+ ( \lambda + \nu ) \|\Div u\|_{L^2}^2 ) \, ds \\
		&= \frac{1}{2} \|\sqrt{\rho_0} u_0 \|_{L^2}^2 + \frac{1}{2} \|\nabla \chi_0\|_{L^2}^2
		+ \frac{A}{\gamma-1} \int_{\Om} \rho_0^{\gamma} \,dx
		+ \int_{\Om} \rho_0 F(\chi_0) \,dx
	\end{align*}
	for any $t \in (0, T_0)$.
	It follows from \eqref{CRC} that
	\begin{align*}
		&\lim_{t \rightarrow 0^+} \|\sqrt{\rho} u \|_{L^2}^2(t) = \|\sqrt{\rho_0} u_0 \|_{L^2}^2.
	\end{align*}
	
	It remains to show that
	\begin{equation}\label{CRU}
		\lim_{t \rightarrow 0^+} \| \rho u - \rho_0 u_0 \|_{L^2}(t) = 0.
	\end{equation}
	One deduces by the Gagliardo-Nirenberg and H\"older inequalities that $\int_{0}^{{T_0}} \| \partial_{t} (\rho_N u_N) \|_{L^2}^{2} \,dt \le C$.
	It follows from the H\"older inequality that
	\begin{align*}
		&\| \rho u - \rho_0 u_0 \|_{L^2}(t) \\
		\le& \| \rho u - \rho_N u_N \|_{L^2}(t)
		+ \| \rho_N u_N -\rho_{0N}u_{0N} \|_{L^2}(t)
		+ \| \rho_{0N} u_{0N} - \rho_{0N} u_0 \|_{L^2}
		+ \| \rho_{0N} u_0 - \rho_{0} u_0 \|_{L^2} \\
		\le& \| \rho u - \rho_N u_N \|_{L^2}(t)
		+ \int_0^{t}\|\partial_{t}(\rho_N u_N) \|_{L^2} \,ds
		+ \| \rho_{0N} \|_{L^\infty} \| u_{0N} - u_0 \|_{L^2}
		+\frac{\|\rho_0\|_{L^\infty}}{N} \|u_0\|_{L^2} \\
		\le& \| \rho u - \rho_N u_N \|_{L^2}(t)
		+ C \sqrt{t}
		+ C \| u_{0N} - u_0 \|_{L^2}
		+\frac{C}{N} \|u_0\|_{L^2}.
	\end{align*}
	Recalling \eqref{u0n}, one gets
	$\| \rho u - \rho_0 u_0 \|_{L^2}(t) \le C \sqrt{t}$.
	This completes the proof.
\end{proof}

\subsection{Existence}\label{Sec_4.2}
To prove Theorem \ref{th1}, it suffices to show that the solution $(\rho, u, \chi, \mu)$
obtained in Theorem \ref{Existence} satisfies $\sqrt{\rho}u \rightarrow \sqrt{\rho_0} u_0$ in $L^2$
as $t\rightarrow 0$. Thanks to \eqref{normcon}, it is sufficient to show
$\sqrt{\rho}u \wc \sqrt{\rho_0} u_0$ in $L^2$ as $t\rightarrow 0$.
Let $\{t_k\}$ be any decreasing sequence such that $t_k \downarrow 0$,
It remains to show that there exists a subsequence $\{t_{k_j}\}$ of $\{t_k\}$ such that
\begin{equation}\label{CRU2}
	 \sqrt{\rho} u (t_{k_j}) \wc \sqrt{\rho_0} u_0 \quad \text{in } L^2
	 \quad \text{as } j\rightarrow \infty.
\end{equation}
Owing to the uniform bound of $\|\sqrt{\rho} u\|_{L^2}(t_k)$, one can extract a subsequence $\{t_{k_j}\}$ of $\{t_k\}$ such that
$$\sqrt{\rho}u (t_{k_j}) \wc \psi  \quad \text{in } L^2
\quad \text{as } j\rightarrow \infty,$$
for some $\psi\in L^2$.
Consequently,
$$
\sqrt{\rho_0}\sqrt{\rho}u (t_{k_j}) \wc \sqrt{\rho_0}\psi
\quad \text{in } L^2 \quad \text{as } j\rightarrow \infty.
$$
On the other hand, by using \eqref{CRC}, the H\"{o}lder inequality and the uniform bounds of $u$ in $L^6$ and $\sqrt{\rho}$ in $L^\infty$, we have $\|\sqrt{\rho_0}\sqrt{\rho}u-\rho u\|_{L^2}\le \|\sqrt{\rho_0}-\sqrt{\rho}\|_{L^3}\|\sqrt{\rho}\|_{L^\infty}\|u\|_{L^6}\rightarrow 0$, as $t\rightarrow 0^+$.
Hence, from \eqref{CRU} we can deduce that
$$
\sqrt{\rho_0} \sqrt{\rho} u (t) \rightarrow \rho_0 u_0 \quad \text{in } L^2
\quad \text{as } t\rightarrow 0^+.
$$
By the uniqueness of weak limits, it follows that $\sqrt{\rho_0} \psi = \rho_0 u_0$ a.e. in $\Om$.
We decompose the domain into the non-vacuum and the vacuum regions:
$$\Om = \Om_+ \cup \Om_0, \quad \Om_+ = \{x\in\Om:\rho_0(x) > 0\},
\quad\Om_0 = \{x\in\Om:\rho_0(x) = 0\}. $$ Hence, one deduces that
$\psi = \sqrt{\rho_0} u_0$ a.e. in  $\Om_+$.
Let $\phi\in L^2$ be arbitrary, then we have
\begin{equation}\label{CRU3}
	\int_{\Om_+} (\sqrt{\rho} u - \sqrt{\rho_0} u_0)\phi \,dx (t_{k_j})
	=\int_{\Om_+} (\sqrt{\rho} u - \psi)\phi \,dx (t_{k_j})
	\rightarrow 0  \quad \text{as } j\rightarrow \infty
\end{equation}
and
\begin{equation}
	\begin{aligned}\label{CRU4}
	\int_{\Om_0} (\sqrt{\rho} u - \sqrt{\rho_0} u_0)\phi \,dx (t_{k_j})
	&=\int_{\Om_0} \sqrt{\rho} u \phi \,dx (t_{k_j}) \\
	&= \int_{\Om_0} (\sqrt{\rho} - \sqrt{\rho_0}) u \phi \,dx (t_{k_j})
	\rightarrow 0 \quad \text{as } j\rightarrow \infty.
	\end{aligned}
\end{equation} 	
Combining \eqref{CRU3} and \eqref{CRU4} yields \eqref{CRU2}. Since the sequence $\{t_k\}$ is arbitrary,
it follows that
$$\sqrt{\rho}u(t) \wc \sqrt{\rho_0} u_0 \quad \text{in } L^2 \quad \text{as } t\rightarrow 0^+. $$
This completes the proof. \qed

\subsection{Uniqueness}\label{Sec_4.3}
This section is devoted to proving the uniqueness of the strong solutions obtained in the previous section.

Let $(\rho_i,u_i,\chi_i)$ $(i=1,2)$ be two strong solutions to system \eqref{NSAC} in $\Om \times(0,T_0)$, subject to the initial condition \eqref{I} and the boundary condition \eqref{B1}. Denote $\rho=\rho_1-\rho_2$, $u=u_1-u_2$, $\chi=\chi_1-\chi_2$. It is straightforward to verify that
$(\rho, u, \chi)$ satisfies that
\begin{equation}
	\begin{split}\label{NSAC_diff}
		\begin{cases}
			\rho_{t} + u_1 \cdot \nabla \rho + \rho \Div u_1
			+ u \cdot \nabla \rho_2 + \rho_2 \Div u=0, \\
			\rho_1 u_{t}-\mathcal{L} u = -\rho \partial_t u_2
			- \rho_1 ( u_1 \cdot \nabla ) u - \rho_1 ( u \cdot \nabla ) u_2 - \rho ( u_2 \cdot \nabla) u_2
			\\\quad- \nabla\bigl( P(\rho_1) - P(\rho_2) \bigr)
			- \Div\bigl( \nabla \chi_1 \otimes \nabla \chi + \nabla \chi \otimes \nabla \chi_2 \bigr)
			- \frac{1}{2} \nabla \bigl( \nabla \chi \cdot ( \nabla \chi_1 + \nabla \chi_2) \bigr),
			\\
			\rho_1^2 \chi_{t} - \Delta \chi = - \rho ( \rho_1 + \rho_2 ) \partial_t \chi_2
			- \rho_1^2 u_1 \cdot \nabla \chi - \rho_1^2 u \cdot \nabla \chi_2
			- \rho (\rho_1 + \rho_2) u_2 \cdot \nabla \chi_2 \\
			\quad - \rho_1 \chi ( \chi_1^2 + \chi_1 \chi_2 + \chi_2^2 - 1) - \rho ( \chi_2^3 - \chi_2)
		\end{cases}
	\end{split}
\end{equation}
a.e. in $\Om\times(0, T_0)$ and $( u, \partial_{\boldsymbol{n}} \chi ) \big|_{\partial \Om} = 0$.
It follows from the H\"older and Poincar\'e inequalities that
\begin{align*}
	\|\sqrt{\rho_1} u\|_{L^2}^2(\tau) &= \|\rho_1 |u|^2\|_{L^1}(\tau)
	\le \|\rho_1 u\|_{L^2}(\tau) \|u\|_{L^2}(\tau) \\
	&\le C \left(\|\rho_1 u_1 - \rho_2 u_2\|_{L^2}(\tau)
	+ \|(\rho_1 - \rho_2) u_2 \|_{L^2}(\tau)\right) \|\nabla u\|_{L^2}(\tau) \\
	&\le C \left(\|\rho_1 u_1 - \rho_2 u_2\|_{L^2}(\tau)
	+ \|\rho_1 - \rho_2\|_{L^3}(\tau) \| u_2 \|_{L^6}(\tau)\right) \\
	&\le C \left(\|\rho_1 u_1 - \rho_2 u_2\|_{L^2}(\tau)
	+ \|\rho_1 - \rho_2\|_{L^3}(\tau)\right)
\end{align*}
for every $\tau>0$. Hence, one deduces by the regularities of $(\rho_i, u_i, \chi_i)$ that
\begin{equation}\label{IU}
	\lim_{\tau \rightarrow 0} \|(\rho,  \sqrt{\rho_1}u, \chi )\|_{L^2}(\tau) = 0,
\end{equation}

The proof of uniqueness proceeds in three steps: energy estimates for $(\rho, u, \chi)$, growth estimates for $(\rho, \chi)$, and a singular-weighted Gr\"onwall argument.

\textbf{Step 1. Energy inequalities. }
Multiplying $\eqref{NSAC_diff}_1$ by $\rho$ and integrating over $\Om$, it follows from
integration by parts, the H\"older, Sobolev and Poincar\'e inequalities that
\begin{equation}\label{rho_diff}
	\begin{split}
		\frac{d}{dt} \|\rho\|_{L^2}^{2}
		&= - \int_{\Om} \rho^{2} \Div u_{1} \,dx
		- 2 \int_{\Om} \rho  u \cdot \nabla \rho_{2} \,dx
		- 2 \int_{\Om} \rho_{2} \Div u \rho \,dx \\
		&\le C \| \nabla u_1 \|_{L^\infty} \|\rho\|_{L^2}^2
		+ C \|u\|_{L^6} \|\rho\|_{L^2} \| \nabla \rho_2 \|_{L^3}
		+ C \| \rho_2\|_{L^\infty} \|\nabla u\|_{L^2} \|\rho\|_{L^2}\\
		&\le C \|\nabla^2 u_1 \|_{L^q} \|\rho\|_{L^2}^2
		+ C \|\nabla u\|_{L^2} \|\rho\|_{L^2}.
	\end{split}
\end{equation}

Testing $\eqref{NSAC_diff}_2$ with $u$, it follows from integration by parts that
\begin{equation*}
	\begin{aligned}
		&\frac{1}{2} \frac{d}{dt} \|\sqrt{\rho_1} u\|_{L^2}^2
		+ \nu \|\nabla u\|_{L^2}^2 + ( \lambda + \nu ) \|\Div u\|_{L^2}^2 \\
	 	=& \int_\Om \rho_1 ( u \cdot \nabla ) u \cdot u_1 \,dx
		- \int_\Om \rho_1 ( u_1 \cdot \nabla ) u \cdot u \,dx
		- \int_\Om \rho_1 ( u \cdot \nabla ) u_2 \cdot u \,dx
		- \int_\Om \rho ( u_2 \cdot \nabla ) u_2 \cdot u \,dx\\
		& - \int_\Om \rho \partial_t u_2 \cdot u \,dx
		+ \int_\Om [ P( \rho_1 ) - P( \rho_2 )] \Div u \,dx
		+ \int_\Om ( \nabla \chi_1 \otimes \nabla \chi
		+ \nabla \chi \otimes \nabla \chi_2 ) : \nabla u \,dx \\
		& + \frac{1}{2} \int_\Om \nabla \chi \cdot ( \nabla \chi_1 + \nabla \chi_2 ) \Div u \,dx \\
		=& \sum_{i=1}^8 M_i.
	\end{aligned}
\end{equation*}
Using the H\"{o}lder, Young, Sobolev and Gagliardo-Nirenberg inequalities, one obtains
\begin{align*}
	M_1+M_2 &\le 2 \|\rho_1\|_{L^\infty}^{\frac{1}{2}} \|\sqrt{\rho_1} u\|_{L^2}
	\|\nabla u\|_{L^2} \| u_1 \|_{L^\infty}\\
	&\le C \|\sqrt{\rho_1}u\|_{L^2}\|\nabla u\|_{L^2} \|\nabla u_1\|_{L^2}^{\frac{1}{2}}
	\|\nabla^2 u_1\|_{L^2}^{\frac{1}{2}}\\
	&\le \frac{\nu}{10}\|\nabla u\|_{L^2}^2
	+ C \| \nabla^2 u_1 \|_{L^2} \|\sqrt{\rho_1}u\|_{L^2}^2, \\
	M_3 &\le  \| \nabla u_2 \|_{L^\infty} \|\sqrt{\rho_1} u\|_{L^2}^2
	\le C \|\nabla^2 u_2\|_{L^q} \|\sqrt{\rho_1} u\|_{L^2}^2, \\
	M_4 &\le \|\rho\|_{L^2} \| u_2 \|_{L^6} \| \nabla u_2 \|_{L^6} \|  u \|_{L^6}\\
	&\le \|\rho\|_{L^2} \| \nabla u_2 \|_{L^2}
	\|\nabla^2 u_2\|_{L^2} \|\nabla u\|_{L^2}\\
	&\le \frac{\nu}{10}\|\nabla u\|_{L^2}^2
	+ C \|\nabla^2 u_2\|_{L^2}^2 \|\rho\|_{L^2}^2. \\
	M_5 &\le \|\rho\|_{L^2}\|\partial_t u_2\|_{L^3}\|u\|_{L^6}\\
	&\le \frac{\nu}{10}\|\nabla u\|_{L^2}^2
	+ C\| \nabla \partial_t u_2\|_{L^2}^2 \|\rho\|_{L^2}^2. \\
	M_6 &\le C (\|\rho_1\|_{L^\infty}^{\gamma-1} + \|\rho_2\|_{L^\infty}^{\gamma-1} )\int_\Om |\rho \Div u | \,dx \\
	&\le \frac{\nu}{10} \|\nabla u\|_{L^2}^2
	+ C \|\rho\|_{L^2}^2. \\
	M_7 + M_8 &\le C \|\nabla \chi\|_{L^2} \big( \| \nabla \chi_1 \|_{L^\infty}
	+ \| \nabla \chi_2 \|_{L^\infty} \big) \|\nabla u\|_{L^2}\\
	&\le C \|\nabla \chi\|_{L^2}
	\big( \| \nabla^2 \chi_1 \|_{L^2}^\frac{1}{2} \|\nabla^3 \chi_1 \|_{L^2}^\frac{1}{2}
	+ \|\nabla^2 \chi_2\|_{L^2}^\frac{1}{2} \|\nabla^3 \chi_2 \|_{L^2}^\frac{1}{2}
	\big) \|\nabla u\|_{L^2}\\
	&\le \frac{\nu}{10}\|\nabla u\|_{L^2}^2
	+ C (\|\nabla^3 \chi_1 \|_{L^2} + \|\nabla^3 \chi_2 \|_{L^2})
	\|\nabla\chi\|_{L^2}^2.
\end{align*}
Thus, we have
\begin{equation}\label{rho1u}
	\begin{split}
		&\frac{d}{dt} \|\sqrt{ \rho_1} u \|_{L^2}^2 + \nu \|\nabla u\|_{L^2}^2
		\le C \big(\| \nabla^2 u_1 \|_{L^2} + \|\nabla^2 u_2\|_{L^q} \big)
		\|\sqrt{\rho_1} u\|_{L^2}^2 \\
		&\quad+ C \big(\|\nabla^3 \chi_1 \|_{L^2} + \|\nabla^3 \chi_2 \|_{L^2}\big)
		\|\nabla\chi\|_{L^2}^2
		+ \big(\|\nabla \partial_t u_2\|_{L^2}^2 + 1 \big) \|\rho\|_{L^2}^2.
	\end{split}
\end{equation}

Multiplying $\eqref{NSAC_diff}_3$ by $\chi$ and integrating over $\Om$, 
it follows from integration by parts and the identity $\rho (\rho_1 + \rho_2) = - \rho^2 + 2 \rho \rho_1$ that
\begin{align*}
		&\frac{1}{2} \frac{d}{dt} \|\rho_{1} \chi\|_{L^2}^2 + \|\nabla \chi\|_{L^2}^2 \\
		=& -\frac{1}{2} \int_{\Om} \rho_1^{2} \Div u_1 \chi^2 \,dx
		- \int_{\Om} \rho (\rho_1 + \rho_2) \partial_t \chi_{2} \chi \,dx
		- \int_{\Om} \rho_1^2 ( u \cdot \nabla) \chi_2 \chi \,dx \\
		& + \int_{\Om} \rho^2 ( u_2 \cdot \nabla ) \chi_2 \chi \,dx
		- 2 \int_{\Om} \rho_{1} \rho ( u_2 \cdot \nabla ) \chi_2 \chi  \,dx
		- \int_{\Om} \rho_{1} \chi^2 (\chi _{1}^{2}+\chi _{1}\chi _{2}+\chi _{2}^{2}-1) \,dx \\
		& - \int_{\Om} \rho(\chi_{2}^3 - \chi_2 ) \chi \,dx \\
		=& \sum_{i=1}^{7} M_{i}.
\end{align*}
Using Lemma~\ref{lem1}, the H\"{o}lder, Young and Sobolev inequalities, one arrives at
\begin{align*}
	M_1 &\le \frac{1}{2} \| \nabla u_1 \|_{L^\infty} \|\rho_1 \chi\|_{L^2}^2  \le C \| \nabla^2 u_1 \|_{L^q} \|\rho_1 \chi\|_{L^2}^2, \\
	M_2 &\le ( \|\rho_1\|_{L^\infty} + \|\rho_2\|_{L^\infty} )
	\|\rho\|_{L^2} \|\chi\|_{L^3}  \| \partial_t \chi_2 \|_{L^6}\\
	&\le C \|\rho\|_{L^2} (\|\rho_1\chi\|_{L^2} + \|\nabla \chi\|_{L^2})
	(\|\rho_2 \partial_t \chi_2\|_{L^2} + \|\nabla \partial_t \chi_2\|_{L^2}) \\
	&\le \frac{1}{6} \|\nabla \chi\|_{L^2}^2
	+ C (1 + \|\nabla \partial_t \chi_2\|_{L^2}^2) \|\rho\|_{L^2}^2
	+ C \|\rho_1 \chi\|_{L^2}^2, \\
	M_3 &\le \|\rho_1\|_{L^\infty}^{\frac{1}{2}} \|\sqrt{\rho_1} u\|_{L^2}
	\|\rho_1 \chi\|_{L^2} \|\nabla \chi_2\|_{L^\infty} \\
	&\le C \|\sqrt{\rho_1} u\|_{L^2} \|\rho_1 \chi\|_{L^2}
	\|\nabla^2 \chi_2\|_{L^2}^\frac{1}{2} \|\nabla^3 \chi_2\|_{L^2}^\frac{1}{2} \\
	&\le \frac{\eta}{2} \|\nabla^3 \chi_2\|_{L^2} \|\rho_1 \chi\|_{L^2}^2
	+ \frac{C}{\eta} \|\sqrt{\rho_1} u\|_{L^2}^2, \\
	M_4 &\le \|\rho\|_{L^2}^2 \| u_2 \|_{L^\infty}
	\| \nabla \chi_2 \|_{L^\infty} \|\chi\|_{L^\infty}\\
	&\le C \| \nabla u_2 \|_{L^2}^{\frac{1}{2}} \|\nabla^2 u_2\|_{L^2}^{\frac{1}{2}}
	\| \nabla^2 \chi_2 \|_{L^2}^{\frac{1}{2}} \|\nabla^3 \chi_2\|_{L^2}^{\frac{1}{2}}
	\|\rho\|_{L^2}^2 \\
	&\le C (1 + \|\nabla^2 u_2\|_{L^2} \|\nabla^3 \chi_2\|_{L^2}) \|\rho\|_{L^2}^2, \\
	M_5 &\le 2 \|\rho\|_{L^2} \|\rho_1 \chi\|_{L^2}
	\| u_2 \|_{L^\infty} \| \nabla \chi_2 \|_{L^\infty} \\
	&\le C \|\nabla^2 u_2\|_{L^2}^{\frac{1}{2}} \|\nabla^3 \chi_2\|_{L^2}^{\frac{1}{2}}
	\|\rho\|_{L^2} \|\rho_1 \chi\|_{L^2} \\
	&\le C \|\rho_1 \chi\|_{L^2}^2 + C \|\nabla^2 u_2\|_{L^2} \|\nabla^3 \chi_2\|_{L^2} \|\rho\|_{L^2}^2, \\
	M_6 &\le \|\rho_1 \chi\|_{L^2}
	\|\chi\|_{L^6} \| \chi_1^2 +\chi_2^2 + \chi_1 \chi_2 -1 \|_{L^3} \\
	&\le C \|\rho_1 \chi\|_{L^2} (\|\rho_1 \chi\|_{L^2} + \|\nabla \chi\|_{L^2})
	( \|\chi_1\|_{L^6}^2 + \|\chi_1\|_{L^6} \|\chi_2\|_{L^6} + \|\chi_2\|_{L^6}^2 + 1) \\
	&\le \frac{1}{6} \|\nabla \chi\|_{L^2}^2 + C \|\rho_1 \chi\|_{L^2}^2, \\
	M_7 &\le \|\rho\|_{L^2} \|\chi\|_{L^6} \|\chi_2^3 - \chi_2\|_{L^3} \\
	&\le C \|\rho\|_{L^2} (\|\rho_1 \chi\|_{L^2} + \|\nabla \chi\|_{L^2})
	\|\chi_2\|_{L^3} (\|\chi_2\|_{L^\infty}^2 + 1) \\
	&\le \frac{1}{6} \|\nabla \chi\|_{L^2}^2 + C \|\rho\|_{L^2}^2 + C \|\rho_1 \chi\|_{L^2}^2
\end{align*}
for every $\eta>0$. Hence, we have
\begin{equation}\label{rho1chi}
	\begin{aligned}
	&\frac{d}{dt} \|\rho_1 \chi\|_{L^2}^2 + \|\nabla \chi\|_{L^2}^2
	\le \eta \|\nabla^3 \chi\|_{L^2} \|\rho_1 \chi\|_{L^2}^2
	+ C (1 + \|\nabla^2 u_1\|_{L^q}) \|\rho_1 \chi\|_{L^2}^2 \\
	&\quad+ C (1 + \|\nabla \partial_t \chi_2\|_{L^2}^2
	+ \|\nabla^2 u_2\|_{L^2} \|\nabla^3 \chi_2\|_{L^2}) \|\rho\|_{L^2}^2
	+ \frac{C}{\eta} \| \sqrt{\rho_1}u\|_{L^2}
	\end{aligned}
\end{equation}
for every $\eta>0$.

\textbf{Step 2. Growth estimates. }
It follows from \eqref{rho_diff} that
\begin{equation*}
	\frac{d}{dt}\|\rho\|_{L^2}(t) \le C \|\nabla^2 u_1 \|_{L^q} \|\rho\|_{L^2} + C.
\end{equation*}
By applying the Gr\"onwall inequality, \eqref{IU} yields that
\begin{equation}\label{GE1}
	\|\rho\|_{L^2}(t) \le C t, \quad \forall t \in (0, T_0).
\end{equation}

Using $\eqref{GE1}$, one deduces by \eqref{rho1chi} and the Poincar\'e inequality that
\begin{equation*}
	\begin{aligned}
	&\frac{d}{dt} \|\rho_1 \chi\|_{L^2}^2 + \|\nabla \chi\|_{L^2}^2
	\le C (1 + \|\nabla^3 \chi\|_{L^2} + \|\nabla^2 u_1\|_{L^q}) \|\rho_1 \chi\|_{L^2}^2 \\
	&\qquad+ C (1 + \|\sqrt{t}\nabla \partial_t \chi_2\|_{L^2}^2
	+ \|\sqrt{t}\nabla^2 u_2\|_{L^2} \|\sqrt{t}\nabla^3 \chi_2\|_{L^2}) t
	+ C \|\sqrt{\rho_1}\|_{L^\infty} \|\nabla u\|_{L^2} \\
	&\quad\le C (1 + \|\nabla^3 \chi\|_{L^2} + \|\nabla^2 u_1\|_{L^q}) \|\rho_1 \chi\|_{L^2}^2 + C.
	\end{aligned}
\end{equation*}
Hence, it follows from the Gr\"onwall inequality and \eqref{IU} that
\begin{equation}\label{GE2}
	\|\rho_1 \chi\|_{L^2}^2(t) + \int_0^t \|\nabla \chi\|_{L^2}^2 (s) \, ds
	\le C t, \quad \forall t \in (0, T_0).
\end{equation}

\textbf{Step 3. Singular $t$-weighted energy inequalities. }
Multiplying \eqref{rho_diff} by $\frac{1}{t}$, one deduces by the Young inequality that
\begin{align*}
	\frac{d}{dt} \frac{\|\rho\|_{L^2}^2}{t} + \frac{\|\rho\|_{L^2}^2}{t^2}
	\le C \| \nabla^2 u_1 \|_{L^q} \frac{\|\rho\|_{L^2}^2}{t}
	+ \frac{1}{2}\frac{\|\rho\|_{L^2}^2}{t^2}
	+ C_1 \|\nabla u\|_{L^2}^2 .
\end{align*}
Therefore, we have
\begin{equation}\label{wrho}
	\frac{d}{dt} \frac{\|\rho\|_{L^2}^2}{t} + \frac{\|\rho\|_{L^2}^2}{2 t^2}
	\le C \| \nabla^2 u_1 \|_{L^q} \frac{\|\rho\|_{L^2}^2}{t}
	+ C_1 \|\nabla u\|_{L^2}^2.
\end{equation}

Multiplying \eqref{rho1chi} by $\frac{1}{\sqrt{t}}$ yields
\begin{equation*}
	\begin{aligned}
	&\frac{d}{dt} \frac{\|\rho_{1} \chi\|_{L^2}^2}{\sqrt{t}}
	+ \frac{\|\rho_{1} \chi\|_{L^2}^2}{2 t^{\frac{3}{2}}}
	+ \frac{\|\nabla \chi\|_{L^2}^2}{ \sqrt{t}} \\
	\le& \frac{C}{\eta\sqrt{t}} \| \sqrt{\rho_1}u\|_{L^2} + \eta \sqrt{t} \|\sqrt{t} \nabla^3 \chi\|_{L^2} \frac{\|\rho_1 \chi\|_{L^2}^2}{t^{\frac{3}{2}}} + C (1 + \|\nabla^2 u_1\|_{L^q}) \frac{\|\rho_{1} \chi\|_{L^2}^2}{\sqrt{t}} \\
	&+ C (1 + \sqrt{t}\|\nabla \partial_t \chi_2\|_{L^2}^2
	+ \|\nabla^2 u_2\|_{L^2} \|\sqrt{t}\nabla^3 \chi_2\|_{L^2}) \frac{\|\rho\|_{L^2}^2}{t} \\
	\le& \frac{C}{\eta\sqrt{t}} \| \sqrt{\rho_1}u\|_{L^2}
	+ \hat{C} \eta \sqrt{t} \frac{\|\rho_1 \chi\|_{L^2}^2}{t^{\frac{3}{2}}}
	+ (1 + \|\nabla^2 u_1\|_{L^q}) \frac{\|\rho_{1} \chi\|_{L^2}^2}{\sqrt{t}} \\
	&\qquad+  C (1 + \|\nabla \partial_t \chi_2\|_{L^2}
	+ \|\nabla^2 u_2\|_{L^2}) \frac{\|\rho\|_{L^2}^2}{t} \\
\end{aligned}
\end{equation*}
for each $\eta>0$. Let $\eta = \frac{1}{4\sqrt{t}\hat{C}}$, one gets
\begin{equation}
	\begin{aligned}\label{wchi}
		&\frac{d}{dt} \frac{\|\rho_{1} \chi\|_{L^2}^2}{\sqrt{t}}
		+ \frac{\|\rho_{1} \chi\|_{L^2}^2}{4 t^{\frac{3}{2}}}
		+ \frac{\|\nabla \chi\|_{L^2}^2}{ \sqrt{t}}
		\le C \| \sqrt{\rho_1}u\|_{L^2}
		+ (1 + \|\nabla^2 u_1\|_{L^q}) \frac{\|\rho_{1} \chi\|_{L^2}^2}{\sqrt{t}} \\
		&\quad +  C (1 + \|\nabla \partial_t \chi_2\|_{L^2}
		+ \|\nabla^2 u_2\|_{L^2}) \frac{\|\rho\|_{L^2}^2}{t}. \\
	\end{aligned}
\end{equation}

Recalling \eqref{rho1u}, we have
\begin{equation}\label{wu}
	\begin{split}
		&\frac{d}{dt} \|\sqrt{ \rho_1} u \|_{L^2}^2 + \nu \|\nabla u\|_{L^2}^2
		\le C \big(\| \nabla^2 u_1 \|_{L^2} + \|\nabla^2 u_2\|_{L^q} \big)
		\|\sqrt{\rho_1} u\|_{L^2}^2 \\
		& + C \big(\|\sqrt{t}\nabla^3 \chi_1 \|_{L^2}
		+ \|\sqrt{t}\nabla^3 \chi_2 \|_{L^2}\big)
		\frac{\|\nabla\chi\|_{L^2}^2}{\sqrt{t}}
		+ C \big(\|\sqrt{t} \nabla \partial_t u_2\|_{L^2}^2 + 1 \big)
		\frac{\|\rho\|_{L^2}^2}{t} \\
		\le& C_2 \frac{\|\nabla\chi\|_{L^2}^2}{\sqrt{t}}
		+ C \big(\| \nabla^2 u_1 \|_{L^2} + \|\nabla^2 u_2\|_{L^q} \big)
		\|\sqrt{\rho_1} u\|_{L^2}^2
		+ C \big(\|\sqrt{t} \nabla \partial_t u_2\|_{L^2}^2 + 1 \big)
		\frac{\|\rho\|_{L^2}^2}{t}
	\end{split}
\end{equation}
Let $C_3$ and $C_4$ be positive constants such that $\nu C_3 \ge 2C_1$ and $C_4 \ge 2C_2C_3$.
Multiplying \eqref{wchi} by $C_4$ and \eqref{wu} by $C_3$,
and combining the resulting with \eqref{wrho}, one obtains by the H\"older inequality that
\begin{align*}
	&\frac{d}{dt} \left( \frac{\|\rho\|_{L^2}^2}{t}
	+ C_3 \|\sqrt{ \rho_1} u \|_{L^2}^2
	+ C_4 \frac{\|\rho_{1} \chi\|_{L^2}^2}{\sqrt{t}} \right)
	+ \frac{\|\rho\|_{L^2}^2}{2 t^2}
	+  C_1 \|\nabla u\|_{L^2}^2
	+ C_4 \frac{\|\rho_{1} \chi\|_{L^2}^2}{4 t^{\frac{3}{2}}}
	+ C_2 C_3 \frac{\|\nabla \chi\|_{L^2}^2}{\sqrt{t}} \\
	&\quad \le C (1 + \| \nabla^2 u_1 \|_{L^q} + \|\nabla^2 u_2\|_{L^q}
	+ \| \nabla \partial_t \chi_2 \|_{L^2}
	+\| \sqrt{t}\nabla \partial_t u_2 \|_{L^2}^2 + 1)
	\left( \frac{\|\rho\|_{L^2}^2}{t}
	+ \|\sqrt{ \rho_1} u \|_{L^2}^2\right) \\
	&\qquad+ C (1 + \|\nabla^2 u_1\|_{L^q}) \frac{\|\rho_{1} \chi\|_{L^2}^2}{\sqrt{t}}.
\end{align*}
	Therefore, noting \eqref{IU}, \eqref{GE1} and \eqref{GE2}, one concludes from the Gr\"onwall inequality that
\begin{equation*}
	\left(\frac{\|\rho\|_{L^2}^2}{t}+ \|\sqrt{\rho_1} u\|_{L^2}^2
	+ \frac{\|\rho_{1} \chi\|_{L^2}^2}{\sqrt{t}} \right)(t)
	+ \int_0^t \left(\|\nabla u\|_{L^2}^2
	+ \frac{\|\nabla \chi\|_{L^2}^2}{\sqrt{t}}\right) \,ds \equiv 0
\end{equation*}
for every $t \in (0, T_0)$. By Lemma~\ref{lem1} and the Poincar\'e inequality, we have $(\rho, u, \chi) \equiv 0$ in $(0, T_0)$, which completes the proof of uniqueness. \qed

\section{Proof of Theorem 1.2}\label{Proof of Theorem 1.2}
Finally, we establish a blow-up criterion (Theorem 1.2) which shows that the obtained local strong solution cannot break down unless certain norms blow up.
We argue by contradiction. Were \eqref{Blowup_Criteria} false, i.e.
\begin{equation} \label{assump_contra}
	C_0:= \int_0^{T^*} \left( \|\nabla u\|_{L^\infty}
	+ \|u\|_{L^\infty}^2
	+ \|\nabla \chi\|_{L^\infty}^2 \right) < \infty,
\end{equation}
we will show that there exists a generic constant $C$, depending only on that depend on
$C_0$, $A$, $\gamma$, $\nu$, $\lambda$, $q$, $\Om$ and $\Phi_0$, such that
\begin{equation*}
	\sup_{0  \le t < T^*} \| ( \nabla \rho, \rho \chi_t, \mu,
	\nabla u, \Div u) \|_{L^2}^2
	+ \int_0^{T^*} { \| ( \sqrt{\rho} u_t, \nabla\chi_t,
		\nabla^2 u ) \|_{L^2}^2 } \, dt \le C
\end{equation*}
and
\begin{equation*}
	\sup_{0  \le t < T^*} ( \|\nabla \rho\|_{L^q}^2
	+ \| ( \sqrt{t\rho}  u_t, \sqrt{t} \nabla \chi_t,
	\sqrt{t} \nabla^2u ) \|_{L^2}^2 )
	+ \int_0^{T^*} { \| ( \sqrt{t} \nabla u_t, \sqrt{t} \Div u_t,
		\sqrt{t} \rho \chi_{tt}) \|_{L^2}^2 } \, dt \le C, \\
\end{equation*}
which contradicts the definition of $T^*$.

\begin{lemma}\label{be1}
	Under the assumption \eqref{assump_contra}, we have
	\begin{align*}
		\sup_{0 \le t < T_*} \|\rho\|_{L^\infty}^2 + \int_{0}^{T_*} ( \| \rho  \chi_t \|_{L^2}^2
		+ \|\nabla^2 \chi\|_{L^2}^2 ) ds  \le  C. 
	\end{align*}
\end{lemma}

\begin{proof}
	According to  \eqref{sup_rho_infty} and \eqref{assump_contra}, we have
	\begin{align*}
		\sup_{0 \le t  \le T_*} \|\rho\|_{L^\infty}  \le \| \rho_0\|_{L^\infty} \Exp\left\{ \int_{0}^{T_*} \|\nabla u\|_{L^\infty} \,dt \right\}  \le C.
	\end{align*}
	It follows from $(\ref{NSAC})_3$ that
	\begin{align*}
		\| \rho\chi_t \|_{L^2}^2 & \le C\| \rho u \cdot \nabla \chi \|_{L^2}^2  + C\|\mu\|_{L^2}^2  \\
		& \le C( \|\rho\|_{L^\infty} \| \sqrt\rho u \|_{L^2}^2 \| \nabla\chi  \|_{L^\infty}^2 + \|\mu\|_{L^2}^2 ).
	\end{align*}
	Integrating the above estimates over $(0,T_*)$, combining \eqref{assump_contra} with \eqref{E0}, we have
	\begin{align*}
		\int_{0}^{T_*} \| \rho\chi_t \|_{L^2}^2  \,dt  \le C \int_{0}^{T_*} ( \|\rho\|_{L^\infty} \| \sqrt\rho u \|_{L^2}^2 \| \nabla\chi  \|_{L^\infty}^2 + \|\mu\|_{L^2}^2 ) \,dt  \le C.
	\end{align*}
	By using $\eqref{NSAC}_4$ and \eqref{E0},  we obtain
	\begin{align*}
		\int_{0}^{T_*} \left\| \Delta\chi \right\|_{L^2}^2 \, dt & \le
		C \int_{0}^{T_*} ( \| \rho \mu \|_{L^2}^2
		+ \| \rho f(\chi) \|_{L^2}^2  )  \, dt \\
		& \le C \int_{0}^{T_*} \|\rho\|_{L^\infty}^2 ( \|\mu\|_{L^2}^2
		+ \|f(\chi)\|_{L^2}^2) \, dt.
	\end{align*}
\end{proof}

\begin{lemma}\label{be2}
	Under the assumption \eqref{assump_contra}, we have
	\begin{equation*}
		\sup_{0  \le t < T^*} \| (  \nabla \rho, \rho \chi_t,
		\mu, \nabla u, \Div u) \|_{L^2}^2
		+  \int_0^{T^*} { \| ( \sqrt{\rho} u_t, \nabla\chi_t,
			\nabla^2 u ) \|_{L^2}^2 } \, dt \le C.
	\end{equation*}
\end{lemma}

\begin{proof}Denote
	\begin{align*}
		\mathcal{E}_1(t) &= \|(\nabla \rho, \rho \chi_t,
		\mu, \nabla u, \Div u)\|_{L^2}^2(t),\\
		\mathcal{D}_1( t ) &= \| ( \sqrt{\rho} u_t, \nabla\chi_t, \nabla^2 u ) \|_{L^2}^2 (t),
        \\
		\mathcal{F}_1(t) &= 1 + \|u\|_{L^\infty}^2(t) + \|\nabla \chi \|_{L^\infty}^2(t)
		+ \|\nabla u\|_{L^\infty}(t).
	\end{align*}
	It follows from $\eqref{NSAC}_1$ that
		\begin{align}\label{be2:1}
			\frac{d}{dt}\|\nabla\rho\|_{L^2}^2
			&\le C\|\nabla u\|_{L^\infty}\|\nabla\rho\|_{L^2}^2
			+ C \|\rho\|_{L^\infty} \|\nabla^2 u\|_{L^2} \|\nabla\rho\|_{L^2} \notag\\
			&\le \eta\|\nabla^2 u\|_{L^2}^2
			+ C_\eta(\|\nabla u\|_{L^\infty}+1) \|\nabla\rho\|_{L^2}^2
	\end{align}
	for every $\eta > 0$.

	Differentiating \eqref{NSAC_34} with respect to $t$, multiplying by $\chi_t$, and integrating over $\Omega$, we can obtain
	\begin{align*}
		&\frac{1}{2} \frac{d}{dt} \|\rho \chi_t\|_{L^2}^{2} + \| \nabla \chi_t \|_{L^2}^{2}    \\
		=& \frac{1}{2} \int_{\Om} \rho^2 \Div u \chi_t^{2} \,dx
		- 2 \int_{\Om} \rho^{2} (u \cdot \nabla\chi_t)\chi_t\,dx
		+ \int_{\Om} \rho^2 \Div u ( u\cdot \nabla \chi)  \chi_t \,dx     \\
		& -\int_{\Om} \rho^{2} ( u \cdot \nabla ) u \cdot \nabla \chi  \chi_t \,dx
		- \int_\Om \rho^2 u \otimes u : \nabla^2\chi \chi_t \,dx
		- \int_\Om \rho^2 (u \cdot \nabla\chi) (u \cdot \nabla\chi_t) \,dx   \\
		& - \int_{\Om} \rho (u_t \cdot\nabla \chi ) \rho \chi_t \,dx
		- \int_{\Om} \rho ( u \cdot \nabla \chi_t ) f(\chi) \,dx
		- \int_{\Om} \rho ( u\cdot \nabla \chi ) f'(\chi) \chi_t \,dx     \\
		& - \int_{\Om} \rho f'(x) \chi_t^{2} \,dx \\
		=& \sum_{i=1}^{10} O_i.
	\end{align*}
	Applying the H\"{o}lder, Sobolev and Young inequalities, and using $\eqref{NSAC}_3$, Lemma~\ref{lem1}--\ref{lem2}, Lemma~\ref{be1} and Corollary~\ref{cor1}, we obtain 
	\begin{align*}
		O_1 + O_3 &\le \|\nabla u\|_{L^\infty}
		(\|\rho \chi_t\|_{L^2}^2 + \|\mu\|_{L^2}^2 ), \\
		O_2 + O_6 + O_8 &\le C \|\rho\|_{L^\infty} \| u\|_{L^\infty} \|\nabla \chi_t\|_{L^2}
		(\|\rho \chi_t\|_{L^2} + \|\mu\|_{L^2} + \|f(\chi)\|_{L^2})\\
		&\le \frac{1}{4} \|\nabla \chi_t\|_{L^2}^2 + C \| u\|_{L^\infty}^2
		( \|\rho \chi_t\|_{L^2}^2 + \|\mu\|_{L^2}^2 + 1 ),\\
		O_4 &\le C \|\rho\|_{L^\infty}  \|u \|_{L^\infty} \|\nabla \chi\|_{L^\infty}
		\|\nabla u\|_{L^2} \|\rho \chi_t \|_{L^2} \\
		&  \le C( \|u\|_{L^\infty}^2 + \|\nabla\chi\|_{L^\infty}^2  ) ( \|\nabla u\|_{L^2}^2 + \|\rho \chi_t \|_{L^2}^2 ),   \\
		O_5 & \le C \|\rho\|_{L^\infty} \|u\|_{L^\infty}^2
		\|\nabla^2 \chi \|_{L^2} \|\rho\chi_t\|_{L^2} \\
		&\le C\|u\|_{L^\infty}^2  ( \| \Delta \chi \|_{L^2}^2
		+ \|\rho \chi_t \|_{L^2}^2  ) \\
		& \le C \|u\|_{L^\infty}^2  ( \|\mu\|_{L^2}^2
		+ \|f(\chi)\|_{L^2}^2
		+ \|\rho \chi_t \|_{L^2}^2 )    \\
		&\le C \| u\|_{L^\infty}^2
		( \|\rho \chi_t\|_{L^2}^2 + \|\mu\|_{L^2}^2 + 1 ),\\
		O_7 & \le \|\rho\|_{L^\infty}^{\frac{1}{2}}
		\| \sqrt \rho u_t \|_{L^2} \|\nabla \chi\|_{L^\infty} \|\rho\chi_t \|_{L^2}
		\le \frac{1}{16} \| \sqrt{\rho} u_t \|_{L^2}^2
		+ C \|\nabla \chi\|_{L^\infty}^2  \|\rho \chi_t\|_{L^2}^2 , \\
		O_9 + O_{10} &\le \| f'(\chi)\|_{L^3} \|\mu\|_{L^2}\|\chi_t\|_{L^6}
		\le C\|\mu\|_{L^2} ( \| \rho\chi_t \|_{L^2} + \|\nabla\chi_t\|_{L^2} ) \\
		&\le \frac{1}{4} \| \nabla \chi_t \|_{L^2}^2 + C \|\mu\|_{L^2}^2. 
	\end{align*}
	Thus, one obtains 
	\begin{align}\label{be2:2}
		\frac{d}{dt}\|\rho  \chi_t \|_{L^2}^{2} +\|\nabla \chi_t \|_{L^2}^{2}
		\le \frac{1}{8}\|\sqrt \rho u_t \|_{L^2}^{2}
		+ C \mathcal{F}_1(t) {(\mathcal{E}_1(t) + 1)}.
	\end{align}
	
	Testing $\eqref{NSAC}_2$ with $u_{t}$, and integrating over $\Omega$,
	one deduces by \eqref{lem:u1:2} and inetagration by parts that
	\begin{align*}
		&\frac{1}{2}\frac{d}{dt} \left(\nu \|\nabla u\|_{L^2}^2
		+ (\lambda + \nu) \|\Div u\|_{L^2}^2
		+ \|\mu\|_{L^2}^2
		+ \| \rho ( u \cdot \nabla \chi) \|_{L^2}^2\right)
		+ \| \sqrt \rho u_t \|_{L^2}^2 \\
		 =& - \int_\Om \rho (u \cdot \nabla )u \cdot u_t \,dx
		+ ( \gamma - 1) \int_\Om P(\rho) | \Div u|^2 \,dx
		- \int_\Om P(\rho) u \cdot \nabla \Div u \,dx  \\
		& - \int_\Om(\nabla \chi_t \cdot \nabla ) u \cdot \nabla \chi \,dx
		- \int_\Om(\nabla \chi \cdot \nabla ) u \cdot \nabla \chi_t \,dx
		+ \int_\Om \nabla \chi \cdot \nabla \chi_t \Div u \,dx\\
		& + \frac{d}{dt} \left(\int_\Om P(\rho) \Div u \,dx + \frac{1}{2} \|\rho \chi_t\|_{L^2}^2
		- \int_\Om \rho F'(\chi) (u \cdot \nabla \chi) \right) \\
		=& \sum_{i=1}^{6} P_i + \hat{G}'(t) + \frac{1}{2} \frac{d}{dt} \|\rho \chi_t\|_{L^2}^2. 	
	\end{align*}
	where $\hat{G}(t) :=  \int_{\Om} P(\rho) \Div  u \,dx
	- \int_\Om \rho F'(\chi) (u \cdot \nabla\chi )\,dx$.
	By virtue of the H\"{o}lder and Young inequalities, together with Lemma~\ref{lem2} and Lemma~\ref{be1}, we have
	\begin{align*}
		&P_1 \le C \|\rho\|_{L^\infty}^{\frac{1}{2}} \|u\|_{L^\infty}
		\|\nabla u\|_{L^2} \| \sqrt \rho u_t \|_{L^2}
		\le \frac{1}{8} \| \sqrt \rho u_t \|_{L^2}^2
		+ C \|u\|_{L^\infty}^2  \|\nabla u\|_{L^2}^2, \\
		&P_2 + P_3 \le C \|\rho\|_{L^\infty}^\gamma \|\nabla u\|_{L^2}^2
		+ C \|\rho\|_{L^\infty}^{\gamma - \frac{1}{2}} \| \sqrt \rho u \|_{L^2}^2
		\|\nabla^2 u\|_{L^2}^2 \le  \eta \|\nabla^2 u\|_{L^2}^2 + C, \\
		&P_4 + P_5 + P_6 \le 3 \| \nabla \chi_t \|_{L^2} \|\nabla u\|_{L^2}
		\|\nabla \chi\|_{L^\infty}
		\le \frac{1}{2} \| \nabla \chi_t \|_{L^2}^2
		+ C \|\nabla u\|_{L^2}^2 \|\nabla \chi\|_{L^\infty}^2
	\end{align*}
	for every $\eta > 0$. Thus, we have
	\begin{align}\label{be2:3}
		&\frac{1}{2}\frac{d}{dt} \left(\nu \|\nabla u\|_{L^2}^2
		+ (\lambda + \nu) \|\Div u\|_{L^2}^2
		+ \|\mu\|_{L^2}^2
		+ \| \rho ( u \cdot \nabla \chi) \|_{L^2}^2\right)
		+ \frac{7}{8} \|\sqrt \rho u_t\|_{L^2}^2 \notag\\
		&\quad \le \frac{1}{2} \| \nabla \chi_t \|_{L^2}^2 + \eta\|\nabla^2 u\|_{L^2}^2
		+ {C \mathcal{F}_1(t) (\mathcal{E}_1(t) + 1)
		+ \frac{1}{2} \frac{d}{dt} \|\rho \chi_t\|_{L^2}^2} + \hat{G}'(t).
	\end{align}
	We deduce  by \eqref{be2:1}, \eqref{be2:2} and \eqref{be2:3} that
	\begin{equation}
		\begin{aligned}\label{be2:4}
		&\frac{1}{2}\frac{d}{dt} \left( 2\|\nabla \rho\|_{L^2}^2 + \|\rho \chi_t\|_{L^2}^2
		+ \|\mu\|_{L^2}^2 + \|\rho (u \cdot \nabla \chi)\|_{L^2}^2
		+ \nu\|\nabla u\|_{L^2}^2 + (\lambda + \nu)\|\Div u\|_{L^2}^2 \right) \\
		&\quad + \frac{3}{4}\|\sqrt\rho u_t \|_{L^2}^2
		+ \|\nabla\chi_t\|_{L^2}^2 \le 2 \eta \|\nabla^2 u\|_{L^2}^2
		+ C \mathcal{F}_1(t) (\mathcal{E}_1(t) + 1) + \hat{G}'(t).
		\end{aligned}
	\end{equation}

	Applying the elliptic estimates to $\eqref{NSAC}_2$, one has
	\begin{align*}
		\| \nabla^2u \|_{L^2}^2 &\le C ( \| \rho u_t \|_{L^2}^2
		+ \| \rho ( u \cdot \nabla ) u \|_{L^2}^2
		+ \| \nabla P \|_{L^2}^2 + \| \Delta \chi \nabla \chi \|_{L^2}^2 ) \\
		&\le C ( \|\rho\|_{L^\infty}^{\frac{1}{2}} \| \sqrt \rho u_t \|_{L^2}^2
		+ \|\rho\|_{L^\infty}^2 \|u\|_{L^\infty}^2 \|\nabla u\|_{L^2}^2
		+ \|\rho\|_{L^\infty}^{2 ( \gamma - 1 ) } \|\nabla \rho\|_{L^2}^2
		+ \|\nabla \chi\|_{L^\infty}^2 \| \Delta \chi \|_{L^2}^2  )  \\
		& \le C ( \| \sqrt \rho u_t \|_{L^2}^2
		+ \|u\|_{L^\infty}^2 \|\nabla u\|_{L^2}^2 + \|\nabla \rho\|_{L^2}^2
		+ \|\nabla \chi\|_{L^\infty}^2 (\|\mu\|_{L^2}^2 + 1 )  ),
	\end{align*}
	that is,
	\begin{align}\label{be2:5}
		\|\nabla^2 u\|_{L^2}^2 \le C \|\sqrt \rho u_t\|_{L^2}^2
		+ C (\|u\|_{L^\infty}^2 + \|\nabla \chi\|_{L^\infty}^2 + 1) (\mathcal{E}_1(t) + 1).
	\end{align}
	Plugging \eqref{be2:5} into \eqref{be2:4} and choosing $\eta$ sufficiently small, one gets
	\begin{align*}
		&\frac{1}{2}\frac{d}{dt} \left( 2\|\nabla \rho\|_{L^2}^2 + \|\rho \chi_t\|_{L^2}^2
		+ \|\mu\|_{L^2}^2 + \|\rho (u \cdot \nabla \chi)\|_{L^2}^2
		+ \nu\|\nabla u\|_{L^2}^2 + (\lambda + \nu)\|\Div u\|_{L^2}^2 \right)\\
		&\quad+ \frac{1}{2}\|\sqrt\rho u_t\|_{L^2}^2
		+ \frac{1}{2}\|\nabla\chi_t\|_{L^2}^2
		+ \eta \|\nabla^2 u\|_{L^2}^2
		\le C \mathcal{F}_1(t) (\mathcal{E}_1(t) + 1) + \hat{G}'(t).
	\end{align*}
	Integrating the above estimate over $(0, t)$, we obtain
	\begin{align}\label{be2:6}
		\mathcal{E}_1(t) + \int_0^t \mathcal{D}_1(s) \, ds
		\le C \int_0^t \mathcal{F}_1(s) ( \mathcal{E}_1(s) + 1 )\, ds
		+ C_2 |\hat{G}(t)|,
	\end{align}
	where we have used the fact that
	\begin{align*}
		|\hat{G}(0)| &\le \left| \int_\Om P(\rho_0) \Div u_0 \,dx \right|
		+ \left| \int_\Om \rho_0 F'(\chi_0)(u_0 \cdot \nabla\chi_0) \,dx \right| \\
		& \le C \|P(\rho_0)\|_{L^2} \|\nabla u_0\|_{L^2}
		+ C \|\rho_0\|_{L^\infty} \| F'(\chi_0) \|_{L^2} \|u_0\|_{L^6}
			\|\nabla \chi_0\|_{L^2}^{1/2}\|\nabla \chi_0\|_{L^6}^{1/2} \\
		&\le C,
	\end{align*}
	and
	$
	\mathcal{E}_1(0) \le C.
	$
	Besides, observing that
	\begin{align*}
		C_2 |\hat{G}(t)|
		&\le C \|P(\rho)\|_{L^2} \|\nabla u\|_{L^2}(t)
		+  C \|F'(\chi)\|_{L^2} (\|\mu\|_{L^2} + \|\rho \chi_t\|_{L^2})(t)\\
		&\le C \mathcal{E}_1^{\frac{1}{2}}(t) \le \frac{1}{2} \mathcal{E}_1(t) + C.
	\end{align*}
	Plugging the above estimate into \eqref{be2:6}, then we have
	\begin{equation*}
		\mathcal{E}_1(t) + \int_0^t \mathcal{D}_1(s) \, ds
		\le C \int_0^t \mathcal{F}_1(s) ( \mathcal{E}_1(s) + 1 )\, ds.
	\end{equation*}
	It follows from $\mathcal{F}_1 \in L^1(0, T^*)$ and the Gr\"onwall inequality that
	\begin{equation*}
		\sup_{0 \le t < T^*} \mathcal{E}_1(t) + \int_0^{T^*} \mathcal{D}_1(s) \, ds \le C.
	\end{equation*}
	The proof is complete.
\end{proof}

The last estimates is the most difficult, as it requires combining the estimates for $\rho, u, \chi$ and carefully handling the time weights.
\begin{lemma}\label{be3}
	Under the assumption \eqref{assump_contra}, we have
	\begin{align*}
		&\sup_{0  \le t < T^*} \left(\|\nabla \rho\|_{L^q}^2
		+ \|(\sqrt{t} \sqrt\rho  u_t, \sqrt{t} \nabla \chi_t,
		\sqrt{t} \nabla^2u , \sqrt{t} \nabla^3\chi)\|_{L^2}^2 \right)  \\
		&+ \int_0^{T^*} { \big(\|(\sqrt{t} \nabla u_t, \sqrt{t} \Div u_t,
			\sqrt{t} \rho \chi_{tt}, \nabla^3 \chi) \|_{L^2}^2 + \|\nabla^2 u\|_{L^q}
			+ \|\sqrt{t} \nabla^2 u\|_{L^q}^2}\big) \,dt \le C. 
	\end{align*}
\end{lemma}

\begin{proof}Denote
	\begin{align*}
		\mathcal{E}_2(t) &= \|\nabla\rho\|_{L^q}^2(t)
		+ \|( \sqrt{t} \sqrt\rho  u_t, \sqrt{t} \nabla \chi_t ) \|_{L^2}^2(t), \\
		\mathcal{D}_2(t) &= \|( \sqrt{t} \nabla u_t, \sqrt{t} \Div u_t,
		\sqrt{t} \rho \chi_{tt}, \nabla^3 \chi ) \|_{L^2}^2(t), \\
		\mathcal{F}_2(t) &= 1 + t^{- \frac{3q-6}{2q}}
		+ \|\nabla u\|_{L^\infty}(t) + \|\nabla \chi\|_{L^\infty}^2(t)
		+ \|( \nabla \chi_t, \nabla^2 u, \sqrt{\rho}  u_t ) \|_{L^2}^2(t).
	\end{align*}

	{\textbf{Step 1. Estimates for the density.}}
	It follows from \eqref{nabla_rho_q} that
	\begin{align*}
		\frac{d}{dt} \|\nabla \rho\|_{L^q}^2
		\le C_1 ( \|\nabla u\|_{L^\infty} \|\nabla\rho\|_{L^q}^2
		+ \|\nabla^2 u\|_{L^q} \|\nabla\rho\|_{L^q} ).
	\end{align*}
	Recall the elliptic estimates
	\begin{align*}
		\|\nabla^2 u\|_{L^q} &\le C_2 ( \|\rho u_{t}\|_{L^q}
		+ \| \rho ( u \cdot \nabla ) u \|_{L^q} + \| \nabla (P(\rho)) \|_{L^q}
		+ \| \Delta \chi \nabla \chi \|_{L^q} ).
	\end{align*}
	Combining the H\"older, Young and Gagliardo-Nirenberg inequalities with Lemma~\ref{lem1} and  Lemma~\ref{be1}--\ref{be2}, we have
	\begin{align*}
		\|\rho u_{t}\|_{L^q}
		&\le C \|\rho\|_{L^\infty}^{\frac{5q - 6}{4q}}
		\|\sqrt \rho u_{t}\|_{L^2}^{\frac{6 - q}{2q}}
		\|\nabla u_{t}\|_{L^2}^{\frac{3q - 6}{2q}}\\
		&\le \frac{\eta}{C_1 C_2} \|\sqrt{t} \nabla u_{t}\|_{L^2}^2
		+ C (\|\sqrt \rho u_{t}\|_{L^2}^2
		+ \big( \frac{1}{\eta t} \big)^{\frac{3q - 6}{2q}}), \\
		\| \rho ( u \cdot \nabla ) u \|_{L^q}
		&\le \|\rho\|_{L^\infty} \|u\|_{L^\infty} \|\nabla u\|_{L^q}\\
		&\le C \|\nabla u\|_{L^2}^{\frac{1}{2}} \|\nabla^2 u\|_{L^2}^{\frac{1}{2}}
		\|\nabla u\|_{L^2}^{\frac{6 - q}{2q}} \|\nabla^2 u\|_{L^2}^{\frac{3q - 6}{2q}}
		\le C\|\nabla^2 u\|_{L^2}^{2 - \frac{3}{q}},\\
		\|\nabla P(\rho)\|_{L^q} &\le C \|\rho\|_{L^\infty}^{\gamma - 1} \|\nabla \rho\|_{L^q}
		\le C\|\nabla \rho\|_{L^q},\\
		\| \Delta \chi \nabla \chi \|_{L^q}
		&\le \|\nabla \chi\|_{L^\infty} \|\nabla^2 \chi\|_{L^q}\\
		&\le C\|\nabla^2 \chi\|_{L^2}^{\frac{1}{2}}
		\| \nabla^3 \chi \|_{L^2}^{\frac{1}{2}}
		\|\nabla^2 \chi\|_{L^2}^{\frac{6 - q}{2q}}
		\| \nabla^3 \chi \|_{L^2} ^{\frac{3q - 6}{2q}}
		\le C\| \nabla^3 \chi \|_{L^2} ^{2 - \frac{3}{q}}
	\end{align*}
	for every $\eta > 0$. By using the Young inequality, the above estimates yields
		\begin{align}\label{be3:rho1}
			\|\nabla^2 u\|_{L^q} &\le C ( \|\sqrt{t} \nabla u_{t}\|_{L^2}^2
			+  \mathcal{F}_2(t) +\mathcal{E}_2(t)
			+ \| \nabla^3 \chi \|_{L^2} ^{2 - \frac{3}{q}} ),
		\end{align}
		and
		\begin{align}\label{be3:rho2}
			\frac{d}{dt} \|\nabla \rho\|_{L^q}^2
			\le \eta \|\sqrt{t} \nabla u_{t}\|_{L^2}^2 \|\nabla\rho\|_{L^q}
			+ C \mathcal{F}_2(t) \mathcal{E}_2(t)
			+ C ( \big( \frac{1}{\eta t} \big)^{\frac{3q - 6}{2q}}
			+ \| \nabla^3 \chi \|_{L^2} ^{2 - \frac{3}{q}} ) \|\nabla\rho\|_{L^q}
		\end{align}
		for every $\eta > 0$. Taking $\eta = \frac{\nu}{8 \big( \nabla \|\rho\|_{L^q} + 1\big)}$ in \eqref{be3:rho2}, one has
		\begin{align}\label{be3:rho3}
			\frac{d}{dt} \|\nabla \rho\|_{L^q}^2
			\le \frac{\nu}{8} \|\sqrt{t} \nabla u_{t}\|_{L^2}^2
			+ C \mathcal{F}_2(t) \big( \mathcal{E}_2(t) + 1 \big)
			+ C \| \nabla^3 \chi \|_{L^2} ^{2 - \frac{3}{q}} \|\nabla\rho\|_{L^q}.
	\end{align}
	
	Applying the operator $\nabla$ to $\eqref{NSAC}_4$, then we have
	\begin{align*}
		\|\nabla \Delta \chi\|_{L^2} &\le \big( \| 2 \rho \nabla\rho \chi_t \|_{L^2}
		+ \| \rho^2 \nabla\chi_t \|_{L^2}
		+ \| 2 \rho \nabla\rho ( u \cdot \nabla\chi ) \|_{L^2}
		+ \|\rho^2 \nabla u \nabla\chi\|_{L^2} \\
		&\quad + \| \rho^2 \nabla^2\chi u \|_{L^2}
		+ \|\nabla\rho f(\chi)\|_{L^2}
		+ \|\rho f'(\chi) \nabla\chi\|_{L^2} \big).
	\end{align*}
	By Lemma~\ref{lem1}--\ref{lem2}, Lemma~\ref{be1}--\ref{be2}, the H\"older, Young, Sobolev and Gagliardo-Nirenberg inequalities,
	terms on the right-hand side are estimated as follows
	\begin{align*}
		\| 2 \rho \nabla\rho \chi_t \|_{L^2}
		+ \| \rho^2 \nabla\chi_t \|_{L^2}
		&\le 2 \|\rho\|_{L^\infty} \|\nabla\rho\|_{L^3} \|\chi_{t}\|_{L^6}
		+ \|\rho\|_{L^\infty}^2 \|\nabla\chi_{t}\|_{L^2} \\
		&\le C \|\nabla\rho\|_{L^2}^{\frac{2q-6}{3q-6}} \|\nabla\rho\|_{L^q}^{\frac{q}{3q-6}}
		( \| \rho \chi_{t} \|_{L^2} + \|\nabla\chi_{t}\|_{L^2} )
		+ C \|\nabla\chi_{t}\|_{L^2}\\
		&\le C (1 + \|\nabla\chi_{t}\|_{L^2} )
		( \|\nabla \rho\|_{L^q}^{\frac{q}{3q-6}} + 1 ), \\
		\|2 \rho \nabla\rho (u \cdot \nabla) \chi\|_{L^2}
		+ \|\rho^2 \nabla u \nabla\chi\|_{L^2}
		&\le ( 2 \|\rho\|_{L^\infty} \|\nabla\rho\|_{L^2} \|u\|_{L^\infty}
		+ \|\rho\|_{L^\infty}^2 \|\nabla u\|_{L^2} ) \|\nabla\chi\|_{L^\infty} \\
		&\le C ( \|\nabla u\|_{L^2}^{\frac{1}{2}} \|\nabla^2 u\|_{L^2}^{\frac{1}{2}}
		+ 1 ) \|\nabla^2 \chi\|_{L^2}^{\frac{1}{2}} \|\nabla ^3 \chi\|_{L^2} ^{\frac{1}{2}} \\
		&\le C ( \|\nabla^2 u\|_{L^2}^{\frac{1}{2}}
		+ 1 ) \| \nabla \Delta \chi\|_{L^2} ^{\frac{1}{2}} \\
		&\le \frac{1}{2} \|\nabla \Delta \chi\|_{L^2}
		+ C ( \|\nabla^2 u\|_{L^2} + 1 ),\\
		\| \rho^2 \nabla^2\chi u \|_{L^2}
		&\le  C \|\rho\|_{L^\infty}^2 \|\nabla^2 \chi\|_{L^2} \| u\|_{L^\infty}
		\le C \|\nabla^2 u\|_{L^2}^{\frac{1}{2}} \\
		&\le C ( \|\nabla^2 u\|_{L^2} + 1 ),\\
		\|\nabla\rho f(\chi)\|_{L^2}
		+ \|\rho f'(\chi) \nabla\chi\|_{L^2}
		&\le C ( \|\nabla \rho\|_{L^2} \|f(\chi) \|_{L^\infty}
		+ C \|\rho\|_{L^\infty} \|f'(\chi) \|_{L^3}\|\nabla \chi\|_{L^6} )
		\le C.
	\end{align*}
	Thus, we have
	\begin{align*}
		\|\nabla \Delta \chi\|_{L^2} &\le C (1 + \|\nabla\chi_{t}\|_{L^2} )
		(1 + \|\nabla \rho\|_{L^q}^{\frac{q}{3q-6}} ) + C \|\nabla^2 u\|_{L^2}.
	\end{align*}
	Then the $H^3-$ estimates of Neumann-Laplacian implies that
	\begin{align}\label{be3:chi1}
		\| \nabla^3 \chi \|_{L^2} &\le C (1 + \|\nabla\chi_{t}\|_{L^2} )
		(1 + \|\nabla \rho\|_{L^q}^{\frac{q}{3q-6}} ) + C \|\nabla^2 u\|_{L^2}.
	\end{align}
	Noticing that for $q\in(3,6)$, it holds that
	\begin{align*}
		\| \nabla^3 \chi \|_{L^2} ^{2 - \frac{3}{q}} \|\nabla\rho\|_{L^q}
		&\le C (1 + \|\nabla\chi_{t}\|_{L^2}^{2 - \frac{3}{q}} )
		(1 + \|\nabla \rho\|_{L^q}^{\frac{q}{3q-6}(2 - \frac{3}{q})} )
		\|\nabla\rho\|_{L^q} + C \|\nabla^2 u\|_{L^2}^{2 - \frac{3}{q}} \|\nabla\rho\|_{L^q}\\
		&\le C (1 + \|\nabla\chi_{t}\|_{L^2}^2 + \|\nabla^2 u\|_{L^2}^2 )
		(1 + \|\nabla \rho\|_{L^q}^2 ),
	\end{align*}
    Combining \eqref{be3:rho3} with \eqref{be3:chi1}, and using the Young inequality, one gets
	\begin{align}\label{be3:rho4}
		\frac{d}{dt} \|\nabla \rho\|_{L^q}^2
		\le \frac{\nu}{8} \|\sqrt{t} \nabla u_{t}\|_{L^2}^2
		+ C \mathcal{F}_2(t) \big( \mathcal{E}_2(t) + 1 \big).
	\end{align}

	{\textbf{Step 2. Estimates for the velocity.}}	
	Recalling \eqref{lem:u2:2}, we have
	\begin{equation}
		\begin{split}
			\label{be3:u1}
			&\quad \frac{1}{2} \frac{d}{dt} \| \sqrt{t} \sqrt\rho u_{t} \|_{L^2}^{2}
			+ \nu \|\sqrt{t} \nabla u_{t}\|_{L^2}^{2}
			+ ( \lambda + \nu ) \| \sqrt{t} \Div u_{t} \|_{L^2}^{2} \\
			&= \frac{1}{2}\|\sqrt\rho u_{t}\|_{L^2}^{2}
			- t \int_{\Om} \rho u \cdot \nabla|u_t|^2 \,dx
			- t \int_{\Om} \rho_t ( u\cdot\nabla ) u \cdot u_t \,dx\\
			&\quad - t \int_{\Om} \rho ( u_t \cdot \nabla ) u \cdot u_t \,dx
			+ A \gamma t \int_{\Om} \rho^{\gamma-1} \rho_t \Div u_t \,dx
			+ 2 t \int_{\Om} ( \nabla\chi_t \otimes \nabla\chi ) : Du_t \,dx \\
			&\quad - t \int_{\Om} \nabla\chi \cdot \nabla\chi_t \Div u_t\,dx
			= \sum_{i=1}^{7}{Q}_i.
		\end{split}
	\end{equation}
	Using Lemma~\ref{lem1}--\ref{lem2}, Lemma~\ref{be1}--\ref{be2}, and applying the H\"{o}lder, Sobolev, Young, Korn, Poincar\'e and Gagliardo-Nirenberg inequalities, one obtains
	\begin{align*}
		Q_2 &\le C\|\rho\|_{L^\infty}^{\frac{1}{2}} \|u\|_{L^\infty}
		\| \sqrt{t} \sqrt\rho u_t\|_{L^2} \|\sqrt{t} \nabla u_t\|_{L^2} \\
		&\le C \|\nabla u\|_{L^2}^{\frac{1}{2}} \|\nabla^2 u\|_{L^2}^{\frac{1}{2}}
		\| \sqrt{t} \sqrt\rho u_t \|_{L^2}\|\sqrt{t} \nabla u_t\|_{L^2} \\
		&\le \frac{\nu}{8} \| \sqrt{t} \nabla u_t\|_{L^2}^2
		+ C \|\sqrt{t} \sqrt\rho u_t \|_{L^2}^2 \|\nabla^2 u\|_{L^2}, \\
		Q_3 + Q_5& \le C \sqrt{t} \|\rho_t\|_{L^2} ( \|u\|_{L^\infty}
    	\|\nabla u\|_{L^3} \|\sqrt{t} u_t\|_{L^6} + \|\rho\|_{L^\infty}^{\gamma-1}
    	\|\sqrt{t} \Div u_t\|_{L^2}) \\
    	&\le C\sqrt{t} (1 + \|\nabla^2 u\|_{L^2}^{\frac{1}{2}})
    	(\|\nabla u\|_{L^2}^{\frac{1}{2}} \|\nabla^2 u\|_{L^2}^{\frac{1}{2}}
    	+ 1) \|\sqrt{t} \nabla u_t\|_{L^2}\\
    	&\le \frac{\nu}{8} \|\sqrt{t} \nabla u_t\|_{L^2}^2
		+ C(1 + \|\nabla^2 u\|_{L^2}^2 ),\\
		Q_4&\le \|\nabla u\|_{L^\infty} \| \sqrt{t} \sqrt\rho u_t\|_{L^2}^2 \\
		Q_6 + Q_7&\le C \sqrt{t} \|\sqrt{t} \nabla\chi_t\|_{L^2}  \| \nabla\chi\|_{L^\infty}
		( \| Du_t\|_{L^2} + \| \Div u_t\|_{L^2} ) \\
		&\le C \|\sqrt{t} \nabla\chi_t\|_{L^2}
		\| \nabla\chi\|_{L^\infty} \| \sqrt{t} \nabla u_t \|_{L^2} \\
		&\le \frac{\nu}{8} \|\sqrt{t} \nabla u_t\|_{L^2}^2
		+ C \| \nabla\chi\|_{L^\infty}^2 \|\sqrt{t} \nabla\chi_t\|_{L^2}^2,
	\end{align*}
	where we have used the fact that
	\begin{align}\label{be3:rho5}
		\|\rho_t\|_{L^2} &\le \|\nabla \rho\|_{L^2} \|u\|_{L^\infty}
		+ \|\rho\|_{L^\infty} \|\Div u\|_{L^2}\notag\\
		&\le C (1 + \|\nabla u\|_{L^2}^{\frac{1}{2}} \|\nabla^2 u\|_{L^2}^{\frac{1}{2}} )\notag\\
		&\le C (1 + \|\nabla^2 u\|_{L^2}^{\frac{1}{2}}).
	\end{align}
	Substituting the above estimates into $\eqref{be3:u1}$, we obtain
	\begin{align}\label{be3:u2}
		\frac{d}{dt}\| \sqrt{t} \sqrt\rho u_{t} \|_{L^2}^{2}
		+ \frac{5\nu}{4} \|\sqrt{t} \nabla u_{t}\|_{L^2}^{2}
		+ 2 ( \lambda + \nu) \|\sqrt{t} \Div u_{t} \|_{L^2}^{2}
		\le C \mathcal{F}_2(t) (1 + \mathcal{E}_2(t)).
	\end{align}

    Moreover, by using \eqref{be2:5}, \eqref{be3:chi1}, Lemma~\ref{be1}--\ref{be2},
	the Young inequality and Gagliardo-Nirenberg inequality, we show
	\begin{align*}
		\|\nabla^2 u\|_{L^2}^2 &\le C (1 + \|\sqrt \rho u_t\|_{L^2}^2 + \|u\|_{L^\infty}^2
		+ \|\nabla \chi\|_{L^\infty}^2 ) \\
		&\le  C (1 + \|\sqrt \rho u_t\|_{L^2}^2 + \|\nabla u\|_{L^2} \|\nabla^2 u\|_{L^2}
		+ \|\nabla^2 \chi\|_{L^2}  \| \nabla ^3 \chi \|_{L^2} )\\
		&\le C (1 + \|\sqrt \rho u_t\|_{L^2}^2 + \|\nabla^2 u\|_{L^2}
		+ (1 + \|\nabla\chi_{t}\|_{L^2}) (1 + \|\nabla \rho\|_{L^q})) \\
		&\le \frac{1}{2}  \|\nabla^2 u\|_{L^2}^2 +
		C (1 + \|\sqrt \rho u_t\|_{L^2}^2 + \|\nabla\chi_{t}\|_{L^2}^2
		+ \|\nabla \rho\|_{L^q}^2  ).
	\end{align*}
	That is,
	\begin{align}\label{be3:u3}
		\| \sqrt{t} \nabla^2 u \|_{L^2}^2
		\le C (1 + \mathcal{E}_2(t)).
	\end{align}
    This estimate will be used in the next step.

	{\textbf{Step 3. Estimates for the phase variable. }}
	Recalling \eqref{lem:chi3:2}, that is
	\begin{align*}
		&\quad\frac{d}{dt} \left( \frac{1}{2}\|\nabla\chi_t\|_{L^2}^2
		+ \int_{\Om} \rho_tf(\chi)\chi_t \,dx \right) +   \|\rho\chi_{tt}\|_{L^2}^2 \notag \\
		&= -2 \int_{\Om}  \rho \rho_t \chi_t \chi_{tt} \,dx
		- 2 \int_{\Om} \rho \rho_t ( u \cdot \nabla\chi ) \chi_{tt} \,dx
		- \int_{\Om}  \rho^2 ( u_t \cdot \nabla \chi ) \chi_{tt} \,dx \notag \\
		&\quad - \int_{\Om}  \rho^2 ( u \cdot \nabla\chi_t ) \chi_{tt} \,dx
		+ \int_{\Om} \rho_t f'(\chi) ( u \cdot \nabla\chi ) \chi_t \,dx
		+ \int_{\Om} \rho_t  f(\chi) (u \cdot \nabla\chi_t ) \,dx \notag \\
		&\quad +\int_{\Om} \rho f'(\chi) ( u_t \cdot \nabla\chi ) \chi_t \,dx
		+ \int_{\Om} \rho f(\chi) ( u_t \cdot \nabla\chi_t) \,dx
		+ \int_{\Om} \rho_t f'(\chi) \chi_t^2  \,dx \notag \\
		&\quad - \int_{\Om} \rho f'(\chi) \chi_t \chi_{tt} \,dx.
	\end{align*}
	Multiplying both sides of the above equation by $t$ and applying the identity
	\begin{align*}
		- 2 \int_{\Om} \rho \rho_{t} \chi_t \chi_{tt} \,dx
		&= 2 \int_{\Om} \rho \Div ( \rho u ) \chi_t \chi_{tt} \,dx\\
		&= \int_{\Om} ( \nabla ( \rho^2 ) \cdot u + 2 \rho^2 \Div u ) \chi_t \chi_{tt} \,dx\\
		&= \int_{\Om} \rho^2 \Div u \chi_t \chi_{tt} \,dx
		- \int_{\Om} \rho^2 ( u \cdot \nabla \chi_t ) \chi_{tt} \,dx
		- \int_{\Om} \rho^2 ( u \cdot \nabla \chi_{tt} ) \chi_t \,dx\\
		&= \int_{\Om} \rho^2 \Div u \chi_t \chi_{tt} \,dx
		- \frac{d}{dt} \int_{\Om} \rho^2 ( u \cdot \nabla \chi_t ) \chi_t \,dx \\
		&\quad+ 2 \int_{\Om} \rho \rho_t ( u \cdot \nabla \chi_t ) \chi_t \,dx
		+ \int_{\Om} \rho^2 ( u_t \cdot \nabla \chi_t ) \chi_t \,dx,
	\end{align*}
	we have
	\begin{align}\label{be3:chi2}
				&\frac{1}{2} \frac{d}{dt} \|\sqrt{t} \nabla\chi_t\|_{L^2}^2 +
			\| \sqrt{t} \rho\chi_{tt}\|_{L^2}^2   \nonumber\\
			=& \frac{d}{dt} [ t \hat{H}(t) ] - \hat{H}(t)
			+ \frac{1}{2}\|\nabla\chi_t\|_{L^2}
			+ t \int_{\Om} \rho^2 \Div u \chi_t \chi_{tt} \,dx \nonumber\\
			& + 2 t \int_{\Om} \rho \rho_t ( u \cdot \nabla \chi_t ) \chi_t \,dx
			+ t\int_{\Om} \rho^2 ( u_t \cdot \nabla\chi_t ) \chi_t \,dx
			- 2 t \int_{\Om} \rho \rho_t ( u \cdot \nabla\chi ) \chi_{tt} \,dx   \nonumber\\
			& - t \int_{\Om}  \rho^2 ( u_t \cdot \nabla \chi ) \chi_{tt} \,dx
			- t \int_{\Om}  \rho^2 ( u \cdot \nabla\chi_t ) \chi_{tt} \,dx
			+ t \int_{\Om} \rho_t f'(\chi) ( u \cdot \nabla\chi ) \chi_t \,dx \\
			& + t \int_{\Om} \rho_t  f(\chi) (u \cdot \nabla\chi_t ) \,dx
			+ t \int_{\Om} \rho f'(\chi) ( u_t \cdot \nabla\chi ) \chi_t \,dx
			+ t \int_{\Om} \rho f(\chi) ( u_t \cdot \nabla\chi_t) \,dx  \nonumber \\
			& + t \int_{\Om} \rho_t f'(\chi) \chi_t^2  \,dx
			- t \int_{\Om} \rho f'(\chi) \chi_t^2  \,dx   \nonumber\\
			=& \frac{d}{dt} [ t \hat{H}(t) ] - \hat{H}(t)
			+ \frac{1}{2}\|\nabla\chi_t\|_{L^2} + \sum_{i=1}^{12}{R}_i,  \nonumber
	\end{align}
	where
	\begin{equation*}
		\hat{H}(t) = - \int_{\Om} \rho^2 ( u \cdot \nabla \chi_t ) \chi_t \,dx
		- \int_{\Om} \rho_t f(\chi) \chi_t \,dx.
	\end{equation*}
	Applying Lemma~\ref{lem1}--\ref{lem2}, Lemma~\ref{be1}--\ref{be2}, it follows from \eqref{be2:5}, \eqref{be3:chi1}, \eqref{be3:rho5}, \eqref{be3:u3},
	the H\"older, Young, Poincar\'e and Gagliardo-Nirenberg inequalities that
	\begin{align*}
		|\hat{H}(t)|&\le \|\rho\|_{L^\infty} \|\rho \chi_t\|_{L^2}
		\|u\|_{L^\infty} \|\nabla\chi_t\|_{L^2}
		+ \|\rho_t\|_{L^2} \|f(\chi)\|_{L^3} \|\chi_t\|_{6}\\
		&\le C  \|\nabla u\|_{L^2}^{\frac{1}{2}} \|\nabla^2 u\|_{L^2}^{\frac{1}{2}}
		\|\nabla\chi_t\|_{L^2} + C ( \|\nabla^2 u\|_{L^2}^{\frac{1}{2}} + 1 )
		(\|\rho \chi_t\|_{L^2} + \|\nabla\chi_t\|_{L^2})\\
		&\le C (1 + \|\nabla^2 u\|_{L^2}
		+ \|\nabla\chi_t\|_{L^2}^2 ) \le C \mathcal{F}_2(t), \\
		R_1 + R_6&\le \|\sqrt{t} \rho\chi_{tt}\|_{L^2} \|\rho\|_{L^\infty}
		(\|\nabla u\|_{L^3} \|\sqrt{t} \chi_t\|_{L^6}
		+ \|u\|_{L^\infty} \|\sqrt{t} \nabla\chi_t\|_{L^2} ), \\
		&\le C \|\sqrt{t} \rho\chi_{tt}\|_{L^2}
		\|\nabla u\|_{L^2}^{\frac{1}{2}} \|\nabla^2 u\|_{L^2}^{\frac{1}{2}}
		(\|\rho \chi_t\|_{L^2} + \| \sqrt{t} \nabla\chi_t \|_{L^2}) \\
		&\le \frac{1}{8} \|\sqrt{t} \rho\chi_{tt}\|_{L^2}^2
		+ C \|\nabla^2 u\|_{L^2} (1 + \|\sqrt{t} \nabla\chi_t\|_{L^2}^2) \\
		&\le \frac{1}{8} \|\sqrt{t} \rho\chi_{tt}\|_{L^2}^2
		+ C \mathcal{F}_2(t) (1 + \mathcal{E}_2(t)),   \\
		R_2&\le 2 \sqrt{t} \|\rho\|_{L^\infty} \|\rho_t\|_{L^3}
		\|u\|_{L^\infty} \| \sqrt{t}\nabla \chi_t \|_{L^2} \|\chi_t\|_{L^6}\\
		&\le C  \|\nabla u\|_{L^2}^{\frac{1}{2}} \|\nabla^2 u\|_{L^2}
		(\|\nabla \rho\|_{L^q} + 1) \|\sqrt{t} \nabla \chi_t\|_{L^2}
		(\|\rho \chi_t\|_{L^2} + \|\nabla\chi_t\|_{L^2}) \\
		&\le C  (1 + \|\nabla^2 u\|_{L^2}^2 + \| \nabla \chi_t \|_{L^2}^2)
		(1 + \|\nabla \rho\|_{L^q}^2 + \|\sqrt{t} \nabla \chi_t\|^2_{L^2})\\
		&\le C \mathcal{F}_2(t) ( \mathcal{E}_2(t) + 1  ),   \\
		R_3 + R_5 &\le \|\rho\|_{L^\infty}^{\frac{1}{2}} \|\sqrt{t} \sqrt\rho u_{t}\|_{L^3}
		(  \|\rho\|_{L^\infty} \|\sqrt{t} \nabla\chi_t\|_{L^2} \|\chi_t\|_{L^6}
		+ \| \sqrt{t} \rho\chi_{tt}\|_{L^2} \|\nabla\chi\|_{L^6}) \\
		&\le C \|\sqrt{t} \sqrt\rho u_{t}\|_{L^2}^{\frac{1}{2}}
		\|\sqrt{t} \sqrt\rho u_{t}\|_{L^6}^{\frac{1}{2}}
		\big(\|\sqrt{t} \nabla\chi_t\|_{L^2} (\|\rho \chi_t\|_{L^2} + \|\nabla\chi_t\|_{L^2})
		+ \| \sqrt{t} \rho\chi_{tt}\|_{L^2}\big) \\
		&\le C \|\rho\|_{L^\infty}^{\frac{1}{4}}
		\|\sqrt{t} \sqrt\rho u_{t}\|_{L^2}^{\frac{1}{2}}
		\|\sqrt{t} \nabla u_{t}\|_{L^2}^{\frac{1}{2}}
		\big(\|\sqrt{t} \nabla\chi_t\|_{L^2} (1 + \|\nabla\chi_t\|_{L^2})
		+ \| \sqrt{t} \rho\chi_{tt}\|_{L^2} \big) \\
		&\le \frac{1}{8} \| \sqrt{t} \rho\chi_{tt}\|_{L^2}^2
		+ \frac{\nu}{16} \| \sqrt{t} \nabla u_t \|_{L^2}^2
		+ C \| \sqrt{t} \sqrt\rho u_{t} \|_{L^2}^2
		+ C (1 + \|\nabla\chi_t\|_{L^2}^2) \|\sqrt{t} \nabla\chi_t\|_{L^2}^2 \\
		& \le \frac{1}{8} \| \sqrt{t} \rho\chi_{tt}\|_{L^2}^2
		+ \frac{\nu}{16} \| \sqrt{t} \nabla u_t \|_{L^2}^2
		+ C \mathcal{F}_2(t)  \mathcal{E}_2(t),   \\
		R_4& \le \sqrt{t} \|\sqrt{t} \rho\chi_{tt}\|_{L^2} \|\rho_t\|_{L^2}
		\|u\|_{L^\infty} \|\nabla\chi\|_{L^\infty} \\
		&\le C \sqrt{t} \|\sqrt{t} \rho\chi_{tt}\|_{L^2}
		(1 + \|\nabla^2 u\|_{L^2}^{\frac{1}{2}} )
		\|\nabla u\|_{L^2}^{\frac{1}{2}} \|\nabla^2 u\|_{L^2}^{\frac{1}{2}}
		\|\nabla^2\chi\|_{L^2}^{\frac{1}{2}} \|\nabla^3\chi\|_{L^2}^{\frac{1}{2}}\\
		&\le C \|\sqrt{t} \rho\chi_{tt}\|_{L^2} (1 + \|\nabla^2 u\|_{L^2} )
		\left((1 + \|\sqrt{t} \nabla\chi_{t} \|_{L^2})
		(1 + \|\nabla\rho\|_{L^q}^{\frac{q}{3q-6}})
		+ \|\sqrt{t} \nabla^2 u\|_{L^2}\right)^{\frac{1}{2}} \\
		&\le \frac{1}{8} \|\sqrt{t} \rho\chi_{tt}\|_{L^2}^2
		+ C \mathcal{F}_2(t) (\mathcal{E}_2(t) + 1) ,   \\
		R_7 + R_8&\le C \sqrt{t} \|\rho_{t}\|_{L^2} \left(\|u\|_{L^6} \|f'(\chi)\|_{L^\infty}
		\|\nabla\chi\|_{L^6}\| \sqrt{t} \chi_t\|_{L^6} + \|u\|_{L^\infty} \|f(\chi)\|_{L^\infty}\|\sqrt{t} \nabla\chi_t\|_{L^2}\right) \\
		&\le C (1 + \|  \nabla^2 u\|_{L^2}^{\frac{1}{2}}) (\|\nabla^2\chi\|_{L^2}
		( \|\rho\chi_t\|_{L^2} + \|\sqrt{t} \nabla\chi_t\|_{L^2})
		+ \|\sqrt{t} \nabla\chi_t\|_{L^2}) \\
		&\le C \mathcal{F}_2(t) (1 + \mathcal{E}_2(t)), \\
		R_9 + R_{10}&\le C \|\rho\|_{L^\infty}^{\frac{1}{2}} \|\sqrt{t} \sqrt\rho u_{t}\|_{L^2}
		( \|f'(\chi)\|_{L^6} \|\nabla\chi\|_{L^6} \|\sqrt{t} \chi_{t}\|_{L^6}
		+ \|f(\chi)\|_{L^\infty} \|\sqrt{t} \nabla\chi_t\|_{L^2}) \\
		& \le C \|\sqrt{t} \sqrt\rho u_{t}\|_{L^2} ( \|\nabla^2\chi\|_{L^2}
		(\|\rho\chi_t\|_{L^2} + \|\sqrt{t} \nabla\chi_t\|_{L^2})
		+ \|\sqrt{t} \nabla\chi_t\|_{L^2} )  \\
		&\le C (1 + \mathcal{E}_2(t)),   \\
		R_{11} + R_{12}&\le C ( \|f'(\chi)\|_{L^6} \|\rho_t\|_{L^2}
		\| \sqrt{t} \chi_t\|_{L^6}^2 + \|f'(\chi)\|_{L^3}
		\|\sqrt{t} \rho\chi_{tt}\|_{L^2} \|\sqrt{t} \chi_{t}\|_{L^6}) \\
		&\le \frac{1}{8} \|\sqrt{t} \rho\chi_{tt}\|_{L^2}^2
		+ C (1 + \|\nabla^2 u\|_{L^2}^{\frac{1}{2}})
		( \|\rho\chi_t\|_{L^2}^2 + \|\sqrt{t} \nabla\chi_t\|_{L^2}^2) \\
		& \le \frac{1}{8} \|\sqrt{t}\rho\chi_{tt}\|_{L^2}^2
		+ C \mathcal{F}_2(t) (1 + \mathcal{E}_2(t)),
	\end{align*}
	where we have used
	\begin{align*}
		\|\rho_t\|_{L^3} &\le \|\rho\|_{L^\infty} \|\nabla u\|_{L^3}
		+ \|u\|_{L^\infty} \|\nabla\rho\|_{L^3}\\
		&\le C  \|\nabla u\|_{L^2}^{\frac{1}{2}}  \|\nabla^2 u\|_{L^2}^{\frac{1}{2}}
		(1 + \|\nabla\rho\|_{L^2}^{\frac{2q - 6}{3q - 6}}
		\|\nabla\rho\|_{L^q}^{\frac{q}{3q - 6}})\\
		&\le C \|\nabla^2 u\|_{L^2}^{\frac{1}{2}}
		(1 + \|\nabla\rho\|_{L^q} ).
	\end{align*}
	Plugging the above estimates into \eqref{be3:chi2}, one has
	\begin{equation}\label{be3:chi3}
		\frac{d}{dt} \|\sqrt{t} \nabla\chi_t\|_{L^2}^2
		+  \| \sqrt{t} \rho\chi_{tt}\|_{L^2}^2  \le 2\frac{d}{dt} [ t \hat{H}(t) ]
		+ \frac{\nu}{8} \| \sqrt{t} \nabla u_t \|_{L^2}^2
		+ C \mathcal{F}_2(t) ( \mathcal{E}_2(t) + 1).
	\end{equation}

	{\textbf{Step 4. Closure of the estimates. }}
	Combining \eqref{be3:chi1}, \eqref{be3:rho4}, \eqref{be3:u2} with \eqref{be3:chi3} gives
	\begin{align*}
		\mathcal{E}_2'(t) +  \mathcal{D}_2(t) \le 2\frac{d}{dt} [ t \hat{H}(t) ]
		+ C \mathcal{F}_2(t) ( \mathcal{E}_2(t) + 1 ).
	\end{align*}
	Integrating over $(0, t)$, one gets
	\begin{align*}
		\mathcal{E}_2(t) + \int_0^t \mathcal{D}_2(s) \, ds
		\le 2 t \hat{H}(t) + C \int_0^t \mathcal{F}_2(s) ( \mathcal{E}_2(s) + 1 )\, ds.
	\end{align*}
	By using the estimate of $\| \sqrt{t} \nabla^2 u \|_{L^2}$ in \eqref{be3:u3},
	we have
	\begin{align*}
		2 t \hat{H}(t) &\le 2 t \|\rho\|_{L^\infty} \|u\|_{L^\infty}
		\| \nabla \chi_t  \|_{L^2} \|\rho \chi_t\|_{L^2}
		+ 2 t \|\rho_t\|_{L^2} \|f(\chi)\|_{L^3} \|\chi_t\|_{L^6}\\
		&\le C t^{\frac{1}{4}} \|\nabla u\|_{L^2}^{\frac{1}{2}}
		\| \sqrt{t} \nabla^2 u \|_{L^2}^{\frac{1}{2}}
		\| \sqrt{t} \nabla \chi_t  \|_{L^2}
		+ C t^{\frac{1}{4}} (1 + \| \sqrt{t} \nabla^2 u \|_{L^2}^{\frac{1}{2}} )
		(\|\rho\chi_t\|_{L^2} + \|\sqrt{t} \nabla\chi_t\|_{L^2}) \\
		&\le C (1 + \mathcal{E}_2^{\frac{1}{4}}(t)) \mathcal{E}_2^{\frac{1}{2}}(t)
		+ C (1 + \mathcal{E}_2^{\frac{1}{4}}(t))
		(1 + \mathcal{E}_2^{\frac{1}{2}}(t))\\
		&\le \frac{1}{2} \mathcal{E}_2(t) + C.
	\end{align*}
	Hence, one gets
	\begin{align*}
		\frac{1}{2}\mathcal{E}_2(t) + \int_0^t \mathcal{D}_2(s) \, ds
		\le C \int_0^t \mathcal{F}_2(s) ( \mathcal{E}_2(s) + 1 )\, ds.
	\end{align*}
	It follows from $\mathcal{F}_2 \in L^1(0, T^*)$ and the Gr\"onwall inequality that
	\begin{equation}\label{be3:1}
		\sup_{0 \le t < T^*} \mathcal{E}_2(t) + \int_0^{T^*} \mathcal{D}_2(t) \, dt \le C.
	\end{equation}
	By using \eqref{be3:rho1}, \eqref{be3:chi1},
	\eqref{be3:u3} and \eqref{be3:1}, there holds
	\begin{equation}\label{be3:2}
		\sup_{0 \le t < T^*} \left(\mathcal{E}_2(t) + \| \sqrt{t} \nabla^2u \|_{L^2}^2(t)
		+ \|\sqrt{t} \nabla^3\chi\|_{L^2}^2(t)\right)
		+ \int_0^{T^*} \left( \mathcal{D}_2(t) + \|\nabla^2 u\|_{L^q}(t) \right) \,dt \le C.
	\end{equation}
	Recalling \eqref{nabla2uq}, we have
		\begin{equation*}
			\begin{aligned}		
				\|\nabla^{2}u\|_{L^q}&\le C ( \|\rho u_{t}\|_{L^q}
				+ \|\rho ( u \cdot \nabla ) u\|_{L^q}
				+ \|\nabla (P(\rho))\|_{L^q} + \|\Delta \chi \nabla \chi\|_{L^q}) \\
				&\le C (\|u_t\|_{L^6} + \|u\|_{L^\infty} \|\nabla u\|_{L^6} 			
				+ \|\nabla \rho\|_{L^q} \|\rho\|_{L^\infty}^{\gamma-1}
				+ \|\nabla \chi\|_{L^\infty} \|\nabla^2\chi\|_{L^q}) \\
				&\le C (1 + \|\nabla u_t\|_{L^2} + \|\nabla^2 u\|_{L^2}^{\frac{3}{2}}
				+ \|\nabla^3 \chi\|_{L^2}^{\frac{3}{2}}).
			\end{aligned}
		\end{equation*}
		Combining the above inequalities and \eqref{be3:2} yields
		\begin{equation*}
			\begin{aligned}
				\int_0^{T^*} \|\sqrt{t} \nabla^2 u\|_{L^q}^2 \,dt
				&\le C \sup_{0 \le t < T^*} \left(\mathcal{E}_2^{\frac{1}{2}}(t)
				+ \|\sqrt{t} \nabla^2u\|_{L^2}(t) + \|\sqrt{t} \nabla^3\chi\|_{L^2}(t)
				\right) \int_0^{T^*} \mathcal{D}_2(t)\,dt \\
				&\le C.
			\end{aligned}
		\end{equation*}
	This completes the proof.
\end{proof}

\vskip2mm
{\bf Acknowledgment.}
Li's work is supported by the National Natural Science Foundation of China (No.12371205)
and the Natural Science Foundation of Guangdong Province (No.2025A151\\5012026, 2025A1515040001).

\vskip6mm
\section{Appendix: Proof of Proposition~\ref{appro_sol}}

{Before proceeding with the proof, we first recall the following fixed point theorem. }
\begin{lemma}[Schaefer's fixed point theorem]\label{SFT}
	Let $X$ be a Banach space and let $\Lambda : X \to X$ be a continuous and compact operator.
	Assume further that the set
	\begin{align*}
		\left\{ x \in X: x = \kappa \Lambda x \ \text{for some } 0 \le \kappa \le 1 \right\}
	\end{align*}
	is bounded in $X$. Then $\Lambda$ admits at least one fixed point in $X$.
\end{lemma}
Now we are ready to give the proof of proposition~\ref{appro_sol}.
Given $\tilde{u} \in C([0, T]; X_N)$. It follows from the classical theory (cf. \cite[Lemma 1.3]{LS75, LS78}) that, there exists a unique $\tilde{\rho} \in C^1(\overline{\Om}\times [0, T])$, such that
\begin{equation*}
	\begin{aligned}
		\partial_t \tilde{\rho} + \Div(\tilde{\rho} \tilde{u} ) &= 0 && \text{\rm in } \Om\times(0,T),\\
		\tilde{\rho}|_{t = 0} &= \rho_{0N} && \text{\rm in } \Om.
	\end{aligned}
\end{equation*}
In fact, $\tilde{\rho}$ is given by
\begin{equation}\label{App:rho1}
	\tilde{\rho}(x, t) = \rho_{0N}\left(U(x, 0; t)\right) \Exp \left\{
	-\int_0^t \Div \tilde{u}\left( U(x, \tau; t), \tau\right) \, d\tau \right\},
	\quad (x, t) \in \overline{\Om} \times [0, T],
\end{equation}
where $U \in C(\overline{\Om} \times [0, T] \times [0, T])$ is the solution to
\begin{equation*}
	\begin{cases}
		\frac{\partial}{\partial t}U(x, t; s) = \tilde{u}(U(x, t; s), t), &\quad 0 \le t \le T,\\
		U(x, s; s) = x, &\quad 0 \le s \le T, x \in \overline{\Om}.
	\end{cases}
\end{equation*}
It follows from \eqref{App:rho1} and the regularity of the eigenfunctions that
\begin{equation}\label{App:rho2}
	\begin{aligned}
		\tilde{\rho}(x, t) \ge \left(\inf_{\overline{\Om}} \rho_{0N} \right)
		\Exp \left\{ - \int_0^t \| \nabla \tilde{u} \|_{L^\infty} \, d\tau \right\}
		\ge \frac{C}{N} \Exp \left\{ -
		\int_0^t \| \nabla \tilde{u} \|_{L^2} \, d\tau \right\}.
	\end{aligned}
\end{equation}
Define the map
\begin{equation*}
	\begin{aligned}
		D_N: C([0, T]; X_N) &\rightarrow C^1(\overline{\Om} \times [0, T]), \\
		\tilde{u} &\rightarrow \tilde{\rho}.
	\end{aligned}
\end{equation*}
Owing to the uniform positive lower bound of $\tilde{\rho}$ and $\chi_{0N}\in H^2$, the well-posedness of the
semi-linear parabolic equation implies that there exists a unique solution $\tilde{\chi} \in L^2(0, T; H^2) \cap H^1(0, T; L^2)$ satisfies
\begin{equation}\label{App:chi1}
	\begin{cases}
		\partial_t \tilde{\chi}  + \tilde{u} \cdot \nabla \tilde{\chi}
		- \frac{1}{\tilde{\rho}^2} \Delta \tilde{\chi}
		+ \frac{1}{\tilde{\rho}} F' (\tilde{\chi})
		= 0 &\text{\rm in } \Om \times (0, T),\\
		\partial_{\boldsymbol{n}} \tilde{\chi} = 0
		&\text{\rm on } \partial \Om \times (0, T), \\
		\tilde{\chi}|_{t = 0} = \chi_{0N} &\text{\rm in } \Om.
	\end{cases}
\end{equation}
Define the map
\begin{equation*}
	\begin{aligned}
		C_N: C([0, T]; X_N) &\rightarrow C([0, T]; H^1), \\
		\tilde{u} &\rightarrow \tilde{\chi},
	\end{aligned}
\end{equation*}
where we have used the continuous embedding $L^2(0, T; H^2) \cap H^1(0, T; L^2)
\hookrightarrow  C([0, T]; H^1)$.
Within $D_N(\tilde{u})$ and $C_N(\tilde{u})$, we seek the solution to the following problem
\begin{equation}\label{App:u1}
	\begin{cases}
		D_N(\tilde{u}) \partial_t u_N + P_N[D_N(\tilde{u}) (u_N \cdot \nabla) u_N]
		- \mathcal{L}u_N
		= G_N(\tilde{u}), \\
		u_N |_{t=0} = u_{0N},
	\end{cases}
\end{equation}
where $G_N(\tilde{u}) = - P_N[\nabla (P(D_N(\tilde{u}))) + \Div(\nabla C_N(\tilde{u}) \otimes
\nabla C_N(\tilde{u}) - \frac{|\nabla C_N(\tilde{u})|^2}{2}\mathbb{I})]$.
To this end, we consider the linearized problem
\begin{equation}\label{App:u2}
	\begin{cases}
		D_N(\tilde{u}) \partial_t u_N + P_N[D_N(\tilde{u}) (\tilde{u} \cdot \nabla) u_N]
		- \mathcal{L}u_N
		= G_N(\tilde{u}), \\
		u_N |_{t=0} = u_{0N}.
	\end{cases}
\end{equation}

Once the problem \eqref{App:u2} is uniquely solvable for any given $\tilde{u}$, we can define a solution map
\begin{equation*}
	\begin{aligned}
		\Lambda_N: C([0, T]; X_N) &\rightarrow C([0, T]; X_N), \\
		\tilde{u} &\rightarrow u_N,
	\end{aligned}
\end{equation*}
where $u_N$ is the unique solution to the system \eqref{App:u2}.
{We first claim that $\Lambda_N$ is well-defined. Furthermore, we claim that
the mapping $\Lambda_N$ satisfies the following properties:

1. $\Lambda_N$ is continuous.

2. $\Lambda_N$ is compact.

3. the set $\{\, u \in C([0,T];X_N) : u = \kappa \Lambda_N(u),\ \kappa \in [0,1] \,\}$ is bounded in $C([0,T]; X_N)$.
\\
Then one can apply Schaefer's fixed point theorem (see Lemma \ref{SFT}) to obtain a fixed point
$u$. More precisely, we have found
$u = \Lambda_N (u)$, $\rho = D_N(u)$ and $\chi = C_N(u)$, which together solve \eqref{appro_pro}.
So it remains for us to prove the unique solvability of \eqref{App:u2} and Claim 1--3.

\vskip2mm
\textbf{Unique solvability of \eqref{App:u2}}.
We reformulate problem \eqref{App:u2} in the following equivalent form
\begin{equation}\label{App:u3}
	\begin{cases}
		\langle D_N(\tilde{u}) \partial_t u_N, w_k \rangle
		+ \langle D_N(\tilde{u}) (\tilde{u} \cdot \nabla) u_N, w_k \rangle
		+ \nu \langle \nabla u, \nabla w_k \rangle
		+ (\nu + \lambda) \langle \Div u, \Div w_k \rangle\\
		\quad = \langle \nabla P(D_N(\tilde{u}))
		- \Delta C_N(\tilde{u}) \nabla C_N(\tilde{u}) , w_k \rangle,
		\quad 0 \le t \le T, \\
		\langle u_N(0), w_k \rangle = \langle u_0, w_k \rangle, \quad k = 1, 2, \cdots, N.
	\end{cases}
\end{equation}
Suppose that $u_N$ has the form $u_N(x, t) = \sum\limits_{j=1}^N {\alpha_j}(t) w_j(x)$. It follows from integration by parts and the properties of the eigenfunction that problem \eqref{App:u3} is equivalent to the following problem
\begin{equation}\label{App:u4}
	\begin{cases}
		\sum_{j=1}^N \langle D_N(\tilde{u}) w_j, w_k \rangle \alpha_j'(t)
		+ \sum_{j=1}^N \langle D_N(\tilde{u}) (\tilde{u} \cdot \nabla) w_j, w_k \rangle \alpha_j(t) + \lambda_k \alpha_k(t) \\
		\quad = \langle \nabla P(D_N(\tilde{u}))
		- \Delta C_N(\tilde{u}) \nabla C_N(\tilde{u}), w_k \rangle,
		\quad 0 \le t \le T, \\
		\alpha_k(0) = \langle u_0, w_k \rangle,  \quad k = 1, 2, \cdots, N.
	\end{cases}
\end{equation}
Denote
\begin{align*}
	&a(t) = (\alpha_1(t), \cdots, \alpha_N(t))^T, \quad
	\mathcal{A}(t) = (\mathcal{A}_{ij}(t)), \quad
	\mathcal{B}(t) = (\mathcal{B}_{ij}(t)), \quad
	\mathcal{F}(t) = (\mathcal{F}_1(t), \cdots, \mathcal{F}_N(t))^T,\\
	&\mathcal{A}_{ij}(t) = \langle D_N(\tilde{u}) w_j, w_i \rangle, \quad
	\mathcal{B}_{ij}(t) = \langle D_N(\tilde{u}) (\tilde{u} \cdot \nabla) w_j, w_i \rangle
	+ \delta_{ij} \lambda_j , \\
	&\mathcal{F}_{j}(t) = \langle \nabla P(D_N(\tilde{u}))
	- \Delta C_N(\tilde{u}) \nabla C_N(\tilde{u}), w_j \rangle,
	\quad i,j = 1, \cdots, N.
\end{align*}
Then, problem \eqref{App:u4} can be rewritten  in the compact form
\begin{equation}
	\begin{cases}\label{Odes1}
		\mathcal{A}(t) a'(t) + \mathcal{B} (t) a(t) = \mathcal{F}(t), \quad 0 \le t \le T, \\
		a(0) = (\langle u_0, w_1 \rangle, \cdots, \langle u_0, w_N \rangle)^{T}.
	\end{cases}
\end{equation}
It follows from \eqref{App:rho2} that the matrix $\mathcal{A}(t)$ is positive definite for any $t\in [0,T]$.
Indeed, for any $\beta=(\beta_1,\cdots,\beta_N)^T\in\mathbb{R}^N$, we have
\begin{align*}
		\beta^T \mathcal{A}(t) \beta
		&= \sum_{i,j=1}^N \langle D_N(\tilde{u}) w_j, w_i \rangle \beta_j \beta_i
		= \Big\langle D_N(\tilde{u}) \sum_{j=1}^N \beta_j w_j,
		\sum_{i=1}^N \beta_i w_i \Big\rangle \\
		&= \int_{\Omega} D_N(\tilde{u})
		\Big| \sum_{j=1}^N \beta_j w_j \Big|^2 dx
		\ge \inf_{\Omega} D_N(\tilde{u})
		{\Big| \sum_{j=1}^N \beta_j^2 \Big|}.
\end{align*}
Since $D_N(\tilde{u})=\tilde{\rho}$ satisfies \eqref{App:rho2}, it follows that
\[
\beta^T \mathcal{A}(t) \beta >0 \qquad \text{for all } \beta\neq 0.
\]
Hence, $\mathcal{A}(t)$ is positive definite and therefore invertible for all $t\in[0,T]$.

Consequently, problem \eqref{Odes1} can be rewritten as
\begin{equation}
	\begin{cases}\label{Odes2}
		a'(t) + \mathcal{A}(t)^{-1} \mathcal{B} (t) a(t)
		= \mathcal{A}(t)^{-1} \mathcal{F}(t), \quad 0 \le t \le T, \\
		a(0) = (\langle u_0, w_1 \rangle, \cdots, \langle u_0, w_N \rangle)^{T}.
	\end{cases}
\end{equation}
Note that
\begin{equation*}
	\begin{aligned}
		\mathcal{F}_{j}(t) &= \langle \nabla P(D_N(\tilde{u}))
		- \Delta C_N(\tilde{u}) \nabla C_N(\tilde{u}), w_j \rangle \\
		&= \langle \nabla P(D_N(\tilde{u}))
		+ \nabla C_N(\tilde{u}) \otimes \nabla C_N(\tilde{u}), \nabla w_j \rangle
		- \frac{1}{2} \langle |\nabla C_N(\tilde{u})|^2, \Div w_j \rangle.
	\end{aligned}
\end{equation*}
Moreover,
\[
\mathcal{A},\mathcal{B}\in C([0,T];\mathbb{R}^{N\times N}),
\qquad
\mathcal{F}\in C([0,T];\mathbb{R}^{N}),
\]
since
\[
D_N(\tilde{u})\in C^1(\overline{\Omega}\times[0,T]),
\quad
C_N(\tilde{u})\in C([0,T];H^1),
\quad
w_k\in C^2(\overline{\Omega}).
\]
Together with the fact that $\mathcal{A}(t)$ is invertible for all
$t\in[0,T]$, we deduce that
\[
\mathcal{A}^{-1}\in C([0,T];\mathbb{R}^{N\times N}).
\]
Therefore, by the classical theory of linear ordinary differential
equations with continuous coefficients, problem \eqref{Odes2}
admits a unique global solution
\[
a\in C^1([0,T];\mathbb{R}^N).
\]
Consequently,
\[
u_N(x,t)=\sum_{j=1}^N \alpha_j(t)w_j(x)
\in C^1([0,T];X_N).
\]

\textbf{The proof of Claim 2. } Let $K > 0$ be given and define
\begin{align*}
	B_K=\{\,u\in C([0,T];X_N): \|u\|_{C([0,T];X_N)}\le K \,\}.
	\end{align*}
For any $\tilde{u} \in B_K$, it follows from \eqref{App:rho2} that there exists a positive constant $\underline{\rho}=\underline{\rho}(K,N)$ such that
\begin{align*}
	\underline{\rho}=\inf_{\overline{\Omega}\times[0,T]} D_N(\tilde u)>0.
\end{align*}

Let $t\in[0,T]$ be arbitrary. By the non-singularity and positive definiteness of $\mathcal{A}(t)$, we have
\begin{align*}
	\|\mathcal{A}^{-1}(t)\|
	&:= \max_{y\neq 0} \frac{|\mathcal{A}^{-1}(t)y|}{|y|}
	= \max_{z\neq 0} \frac{|z|}{|\mathcal{A}(t)z|}
	= \left(\min_{z\neq 0} \frac{|\mathcal{A}(t)z|}{|z|}\right)^{-1} \\
	&\le \left(\min_{z\neq 0} \frac{(z, \mathcal{A}(t) z)}{|z|^2}\right)^{-1}
	\le \frac{1}{\underline{\rho}},
\end{align*}
where $|\cdot|$ and $(\cdot,\cdot)$ denote the Euclidean norm and inner product in $\mathbb{R}^N$, respectively, and $\|\cdot\|$ denote the Euclidean norm in $\mathbb{R}^{N \times N}$.
Here we used the fact that for any $z\in\mathbb{R}^N$,
\begin{align*}
	(z, \mathcal{A}(t) z)
	&= \sum_{i,j=1}^{3} z_i z_j \int_\Omega D_N(\tilde u) w_i w_j \, dx
	= \int_\Omega D_N(\tilde u) \Big|\sum_{i=1}^{3} z_i w_i\Big|^2 \, dx \\
	&\ge \underline{\rho} \int_\Omega \Big|\sum_{i=1}^{3} z_i w_i\Big|^2 \, dx
	= \underline{\rho} \sum_{i,j=1}^{3} z_i z_j \int_\Omega w_i w_j \, dx
	= \underline{\rho} |z|^2.
\end{align*}
Consequently, we obtain
\begin{equation}\label{App:A}
	\max_{0\le t \le T} \|\mathcal{A}^{-1}(t)\| \le \frac{1}{\underline{\rho}}.
\end{equation}
In what follows, $C$ denotes a generic positive constant depending only on
$A$, $\gamma$, $K$, $N$, $\underline{\rho}$, $\|\rho_0\|_{L^\infty}$,
$\|\nabla \chi_0\|_{L^2}$, $T$, and $|\Omega|$.

Since $\dim X_N < \infty$, all norms on $X_N$ are equivalent.
Let
\[
\tilde{\mu} := {- \tilde{\rho}}
 \Big(\tilde{\chi}_t + (\tilde{u}\cdot\nabla)\tilde{\chi}\Big).
\]
Multiplying \eqref{App:chi1} by ${\tilde{\rho}}\tilde{\mu}$ and integrating by parts, we obtain
\begin{align}\label{App:chi2}
	\frac{1}{2} \frac{d}{dt} \|\nabla \tilde{\chi}\|_{L^2}^2
	+ \frac{d}{dt} \int_\Omega \tilde{\rho} F(\tilde{\chi})\,dx + \|\tilde{\mu}\|_{L^2}^2
	&= - \int_\Omega \Delta \tilde{\chi} (\tilde{u}\cdot\nabla)\tilde{\chi}\,dx \notag\\
	&= \int_\Omega (\nabla \tilde{\chi} \otimes \nabla \tilde{\chi}) : \nabla \tilde{u}\,dx
	- \frac{1}{2} \int_\Omega |\nabla \tilde{\chi}|^2 \Div \tilde{u}\,dx \notag\\
	&\le C \|\nabla \tilde{u}\|_{L^\infty} \|\nabla \tilde{\chi}\|_{L^2}^2
	\le C \|\nabla \tilde{\chi}\|_{L^2}^2,
\end{align}
where we used the identity
\[
\int_\Omega \tilde{\rho} f(\tilde{\chi}) \big(\tilde{\chi}_t + (\tilde{u}\cdot\nabla)\tilde{\chi}\big)\,dx
= \frac{d}{dt} \int_\Omega \tilde{\rho} F(\tilde{\chi})\,dx.
\]
Integrating \eqref{App:chi2} over $(0,t)$ yields the uniform bound
\[
\|\nabla \tilde{\chi}\|_{L^2}^2(t) + \int_\Omega \tilde{\rho}(x,t) F(\tilde{\chi}(x,t))\,dx
+ \int_0^t \|\tilde{\mu}\|_{L^2}^2 \,ds \le C, \quad 0\le t\le T.
\]
Consequently, one gets
\begin{equation}\label{App:chi3}
	\max_{0\le t\le T} \Big( \|\nabla C_N(\tilde{u})\|_{L^2}^2
	+ \int_\Omega D_N(\tilde{u}) F(C_N(\tilde{u}))\,dx \Big)
	+ \int_0^T \|\tilde{\mu}\|_{L^2}^2\,dt \le C.
\end{equation}

Combining \eqref{App:chi3} with the H\"older inequality gives
\begin{equation*}
	\begin{aligned}
		|\mathcal{B}_{ij}(t)| &\le \| D_N(\tilde{u}) \|_{L^\infty} \| \tilde{u} \|_{L^\infty}
		\| \nabla w_j \|_{L^2} \| w_i \|_{L^2} + \lambda_N\le C,\\
		|\mathcal{F}_j(t)| &\le C (\| D_N(\tilde{u}) \|_{L^\infty}^{\gamma}
		\| \nabla w_j \|_{L^1} + \| \nabla C_N(\tilde{u})
		\|_{L^2}^2 \| \nabla \tilde{u} \|_{L^\infty})
		\le C,
	\end{aligned}
\end{equation*}
for any $i,j = 1, \cdots, N$ and $t \in [0, T]$.
Hence, we have
\begin{equation}\label{App:BF}
	\max_{0\le t\le T} \big(\|\mathcal{B}(t)\| + |\mathcal{F}(t)|\big) \le C.
\end{equation}

Multiplying \eqref{Odes2} by $a(t)$ and using \eqref{App:A}, \eqref{App:BF}, we obtain
\[
\frac{1}{2} \frac{d}{dt} |a(t)|^2 \le C \big(|a(t)|^2 + 1\big), \quad 0\le t\le T,
\]
which, by the Gr\"onwall inequality, gives
\begin{equation}\label{App:u5}
	\sup_{0\le t\le T} |a(t)|^2 = \sup_{0\le t\le T} \|\Lambda_N(\tilde{u})\|_{L^2}^2 \le C.
\end{equation}
Similarly, multiplying \eqref{Odes2} by $a'(t)$ and using \eqref{App:A}, \eqref{App:BF}, \eqref{App:u5}, we obtain
\[
\sup_{0\le t\le T} |a'(t)|^2 = \sup_{0\le t\le T} \|\partial_t \Lambda_N(\tilde{u})\|_{L^2}^2 \le C.
\]
Combining the above, we conclude
\begin{equation}\label{App:u6}
	\sup_{0\le t\le T} \big( \|\Lambda_N(\tilde{u})\|_{L^2}^2 + \|\partial_t \Lambda_N(\tilde{u})\|_{L^2}^2 \big) \le C.
\end{equation}
Finally, since $X_N$ is finite dimensional, the Arzel\`{a}-Ascoli theorem implies that $\overline{\Lambda_N(B_K)}$ is compact, and thus $\Lambda_N$ is a compact operator.

\vskip2mm
\textbf{The proof of Claim 3. }
Suppose that $u = \kappa \Lambda_N(u)$ for some $\kappa \in (0, 1]$, and define
\begin{align*}
	\rho := D_N(u), \ \chi := C_N(u), \
	\mu := - \rho \chi_t - \rho (u \cdot \nabla) \chi, \
	u = \sum_{i = 1}^N c_i(t) w_i(x),
\end{align*}
where $c_i = \langle u, w_i \rangle$. It follows from \eqref{App:u1} and the definitions of
$D_N$, $C_N$ that, $(\rho, u, \chi)$ satisfies
\begin{equation}\label{claim3:1}
	\begin{cases}
		\rho_t + \Div (\rho u)=0,\\
		\langle   \rho \partial_t u + \rho (u \cdot \nabla) u, w_k \rangle
		+ \nu \langle \nabla u, \nabla w_k \rangle
		+ (\nu + \lambda) \langle \Div u, \Div w_k \rangle \\
		\quad= \kappa \langle \nabla P(\rho) - \Delta\chi \nabla\chi, w_k \rangle,  \\
		\rho^2 \chi_t + \rho^2 ( u \cdot \nabla ) \chi - \Delta\chi + \rho F' (\chi) = 0, \\
		(u, \partial_{\boldsymbol n}\chi)|_{\partial \Om \times (0, T)} = 0, \\
		(\rho, u, \chi)|_{t=0} = (\rho_{0N}, \kappa u_{0N}, \chi_{0N}), \quad k = 1, \cdots, N,
	\end{cases}
\end{equation}
where we have used $\Lambda_N u = \frac{1}{\kappa} u$. Multiplying $\eqref{claim3:1}_2$ with
$c_k$ and summing over $k = 1, \cdots, N$, one deduces by integration by parts and $\eqref{claim3:1}_1$ that
\begin{equation}\label{claim3:2}
	\begin{aligned}
		&\frac{d}{dt} \left( \frac{1}{2} \| \sqrt{\rho} u \|_{L^2}^2
		+ \frac{\kappa A} {\gamma-1}  \int_{\Om} \rho^{\gamma} \,dx \right)
		+ \nu \|\nabla u\|_{L^2}^2 + ( \lambda + \nu) \|\Div u\|_{L^2}^2 \\
		&\quad + \kappa \int_{\Om} \Delta\chi ( u \cdot \nabla ) \chi \,dx = 0.
	\end{aligned}
\end{equation}
where we have used the following identity
\begin{equation*}
	\begin{split}
		\kappa \langle \nabla P(\rho), u \rangle
		= - \kappa A \frac{d}{dt} \int_{\Om} \rho^{\gamma} \,dx
		+ \kappa \gamma \int_{\Om} \nabla P(\rho) \cdot u \,dx.
	\end{split}
\end{equation*}
Multiplying $\eqref{claim3:1}_3$ with
$\kappa (\chi_t + (u \cdot \nabla) \chi)$ and integrating over $\Om$, it follows from integration by parts and $\eqref{claim3:1}_1$ that
\begin{equation}\label{claim3:3}
	\frac{\kappa}{2} \frac{d}{dt} \|\nabla \chi\|_{L^2}^2
	+ \kappa \frac{d}{dt} \int_{\Om} \rho F(\chi) dx + \kappa \|\mu\|_{L^2}^2
	- \kappa \int_{\Om} \Delta\chi ( u \cdot \nabla ) \chi \,dx = 0,
\end{equation}
where we have used
\begin{align*}
	\kappa \int_{\Om} \rho f(\chi) \big( \chi_t + ( u \cdot \nabla ) \chi \big) \,dx
	= \kappa \frac{d}{dt} \int_{\Om} \rho F(\chi) dx.
\end{align*}
Combining \eqref{claim3:2} with \eqref{claim3:3}, one has
\begin{align*}
	& \frac{d}{dt} \left( \frac{1}{2} \| \sqrt{\rho} u \|_{L^2}^2
	+ \frac{\kappa}{2} \|\nabla \chi\|_{L^2}^2
	+ \frac{\kappa A}{\gamma-1} \int_{\Om} \rho^{\gamma} \,dx
	+ \kappa \int_{\Om}\rho F(\chi)dx \right)\\
	&\quad + \kappa \|\mu\|_{L^2}^2 + \nu \|\nabla u\|_{L^2}^2
	+ ( \lambda + \nu) \|\Div u\|_{L^2}^2  = 0.
\end{align*}
Integrating it over $(0, t)$, it follows that
\begin{equation*}
	\frac{1}{2} \| \sqrt{\rho} u \|_{L^2}^2(t)
	+ \nu \int_0^t \|\nabla u\|_{L^2}^2(s) \, ds \le E_{0,\kappa},
\end{equation*}
for any $t \in [0, T]$, where
\begin{equation*}
	E_{0,\kappa} = \frac{\kappa^2}{2} \| \sqrt{\rho_{0N}} u_{0N} \|_{L^2}^2
	+ \frac{\kappa A}{\gamma-1} \int_{\Om} \rho_{0N}^{\gamma}(x) \,dx
	+ \frac{\kappa}{2} \| \nabla\chi_{0N} \|_{L^2}^2
	+ \kappa \int_{\Om} \rho_{0N} F(\chi_{0N}) (x, t) \,dx.
\end{equation*}
Combining the above estimate with \eqref{App:rho2}, it follows
from the H\"older inequality that
\begin{align*}
	\|u\|_{L^2}^2(t) &\le \sup_{(x, s) \in \overline{\Omega} \times [0, t]}
	\frac{1}{\rho(x, s)} \| \sqrt{\rho} u \|_{L^2}^2(t)\\
	&\le C N E_{0, \kappa}\Exp \left\{ \int_0^t \|\nabla u\|_{L^2}(s) \,ds \right\}\\
	&\le C N E_{0, \kappa} \Exp(T^{\frac{1}{2}} E_{0, \kappa}),
\end{align*}
for any $t \in [0, T]$. Since $E_{0, \kappa}$ is non-decreasing with respect to $\kappa$, we conclude that
\begin{equation*}
	\|u\|_{C([0, T]; X_N)} \le  C N E_{0, 1} \Exp(T^{\frac{1}{2}} E_{0, 1}).
\end{equation*}
This completes the proof of Claim 3.

\vskip2mm
\textbf{The proof of Claim 1. } Suppose that
\begin{align*}
	\tilde{u}_m &\rightarrow \tilde{u} \quad \text{in } C([0, T]; X_N) \quad
	\text{with } \| \tilde{u}_m \|_{C([0, T]; X_N) } \le 2 \| \tilde{u} \|_{C([0, T]; X_N)} .
\end{align*}
Let $K = 2 \| \tilde{u} \|_{C([0, T]; X_N)}$ and denote
\begin{align}\label{claim1:rho}
	\tilde{\rho}_m = D_N(\tilde{u}_m),~ \tilde{\chi}_m = C_N(\tilde{u}_m), ~
	\tilde{\mu}_m = - \tilde{\rho}_m \partial_t\tilde{\chi}_m
	- \tilde{\rho}_m (\tilde{u}_m \cdot \nabla) \tilde{\chi}_m, \quad m = 1, 2, \cdots
\end{align}
Then we infer from \eqref{App:rho1} that
\begin{equation*}
	\tilde{\rho}_m \rightarrow \tilde{\rho} = D_N(\tilde{u})
	\quad \text{in } C(\overline{\Om} \times [0, T]).
\end{equation*}
It follows from \eqref{App:chi3} and \eqref{App:u6} that
\begin{equation}
	\begin{aligned}\label{claim1:E}
		&\sup_{0 \le t \le T} \left( \|\Lambda_N(\tilde{u}_m)\|_{L^2}^2
		+ \|\partial_t\Lambda_N(\tilde{u}_m)\|_{L^2}^2 + \|\nabla \tilde{\chi}_m\|_{L^2}^2
		+ \int_{\Om} \tilde{\rho}_m F(\tilde{\chi}_m) \,dx (t) \right) \\
		&\quad+ \int_0^T \|\tilde{\mu}_m\|_{L^2}^2 \,dt \le C.
	\end{aligned}
\end{equation}
Combining the above estimates with Lemma~\ref{lem1} and using the same argument as in Lemma~\ref{cor1}, we deduce
\begin{equation*}
	\sup_{0 \le t \le T}  \| \tilde{\chi}_m \|_{H^1}^2 \le C.
\end{equation*}
Note that
\begin{align*}
	\sup_{0 \le t \le T} \| \tilde{\rho}_m (\tilde{u}_m \cdot \nabla) \tilde{\chi}_m \|_{L^2}
	\le \sup_{0 \le t \le T} \left( \| \tilde{\rho}_m \|_{L^\infty}
	\| \tilde{u}_m \|_{L^\infty} \| \nabla\tilde{\chi}_m \|_{L^2} \right) \le C,
\end{align*}
and
\begin{align*}
	\inf_{\overline{\Om} \times [0, T]} \tilde{\rho}_m
	\ge \frac{1}{N}  \Exp \left\{ -\int_0^t \| \nabla \tilde{u}_m \|_{L^\infty} \, d\tau \right\}
	\ge \frac{C}{N}.
\end{align*}
Plugging the above estimates into \eqref{claim1:E}, one gets
\begin{equation*}
	\sup_{0 \le t \le T} \left( \|\Lambda_N(\tilde{u}_m)\|_{L^2}^2
	+ \|\partial_t\Lambda_N(\tilde{u}_m)\|_{L^2}^2 + \|\tilde{\chi}_m\|_{H^1}^2\right)
	+ \int_0^T ( \|\partial_t\tilde{\chi}_m\|_{L^2}^2
	+ \|\nabla^2\tilde{\chi}_m\|_{L^2}^2) \,dt \le C.
\end{equation*}
By the Banach-Alaoglu theorem, the Arzel\`{a}-Ascoli theorem, and the Aubin-Lions lemma, we can extract a subsequence of $\{(\tilde{\chi}_m, \tilde{u}_m)\}$, still denoted by $\{(\tilde{\chi}_m, \tilde{u}_m)\}$, such that
\begin{align}
	\Lambda_N(\tilde{u}_m) \rightarrow \hat{u} &\s C([0, T]; X_N), \label{claim1:u1}\\
	\partial_t \Lambda_N(\tilde{u}_m) \wsc \partial_t\hat{u}
	&\ws L^{\infty}(0, T; L^2), \label{claim1:u2}\\
	\tilde{\chi}_m \rightarrow \hat{\chi} &\s C([0, T]; L^2)
	\cap L^2(0, T; H^1),  \label{claim1:chi1}\\
	\tilde{\chi}_m \wc \hat{\chi} &\w L^2((0, T); H^2)
	\cap H^1(0, T; L^2),  \label{claim1:chi2}
\end{align}
where $(\hat{u}, \hat{\chi})$ satisfies
$$\hat{u} \in C([0, T]; X_N) \cap W^{1,\infty}(0, T; X_N), \quad
\hat{\chi} \in L^2(0,T; H^2)\cap H^1(0, T; L^2). $$
Substituting $\tilde{\chi}_m$ and $\tilde{\rho}_m$ into \eqref{App:chi1}, we obtain
\begin{equation*}
	\begin{cases}
		\partial_t \tilde{\chi}_m  + \tilde{u}_m \cdot \nabla \tilde{\chi}_m
		- \frac{1}{\tilde{\rho}_m^2} \Delta \tilde{\chi}_m
		+ \frac{1}{\tilde{\rho}_m} F' (\tilde{\chi}_m)
		= 0 &\text{\rm in } \Om \times (0, T),\\
		\partial_{\boldsymbol{n}} \tilde{\chi}_m = 0
		&\text{\rm on } \partial \Om \times (0, T), \\
		\tilde{\chi}_m|_{t = 0} = \chi_{0N} &\text{\rm in } \Om.
	\end{cases}
\end{equation*}
Passing to the limit $m\rightarrow\infty$, it follows from \eqref{claim1:rho}, \eqref{claim1:chi1} and \eqref{claim1:chi2} that $\hat{\chi}$ satisfies
\begin{equation*}
	\begin{cases}
		\partial_t \hat{\chi}  + \tilde{u} \cdot \nabla \hat{\chi}
		- \frac{1}{\tilde{\rho}^2} \Delta \hat{\chi}
		+ \frac{1}{\tilde{\rho}} F' (\hat{\chi})
		= 0 &\text{\rm in } \Om \times (0, T),\\
		\partial_{\boldsymbol{n}} \hat{\chi} = 0
		&\text{\rm on } \partial \Om \times (0, T), \\
		\hat{\chi}|_{t = 0} = \chi_{0N} &\text{\rm in } \Om.
	\end{cases}
\end{equation*}
Then, the unique solvability of system \eqref{App:chi1} implies that $\hat{\chi} = \tilde{\chi}$.
Similarly, substituting $\tilde{u}_m$, $\tilde{\chi}_m$ and $\tilde{\rho}_m$ into \eqref{App:u3}, it follows from integration by parts that
\begin{equation*}
	\begin{cases}
		\langle \tilde{\rho}_m \partial_t \Lambda_N(\tilde{u}_m)
		+ \tilde{\rho}_m (\tilde{u}_m \cdot \nabla) \Lambda_N(\tilde{u}_m), w_k \rangle
		+ \langle P(\tilde{\rho}_m), \Div w_k \rangle\\
		\quad = - \nu \langle \nabla \Lambda_N(\tilde{u}_m), \nabla w_k \rangle
		- (\nu + \lambda) \langle \Div \Lambda_N(\tilde{u}_m), \Div w_k \rangle \\
		\qquad+ \langle \nabla\tilde{\chi}_m \otimes \nabla\tilde{\chi}_m, \nabla w_k \rangle
		- \frac{1}{2} \langle |\nabla\tilde{\chi}_m|^2 , \Div w_k \rangle,
		\quad 0 \le t \le T, \\
		\langle \Lambda_N(\tilde{u}_m)(\cdot, 0), w_k \rangle
		= \langle u_0, w_k \rangle, \quad k = 1, 2, \cdots, N.
	\end{cases}
\end{equation*}
Passing to the limit $m\rightarrow\infty$ and combining \eqref{claim1:u1}--\eqref{claim1:chi2} with the above convergence results, we have
\begin{equation*}
	\begin{cases}
		\langle \tilde{\rho} \partial_t \hat{u}
		+ \tilde{\rho} (\tilde{u} \cdot \nabla) \hat{u} , w_k \rangle
		+ \nu \langle \nabla \hat{u} , \nabla w_k \rangle
		+ (\nu + \lambda) \langle \Div \hat{u} , \Div w_k \rangle\\
		\quad = - \langle P(\tilde{\rho}), \Div w_k \rangle
		+ \langle \nabla\tilde{\chi} \otimes \nabla\tilde{\chi}, \nabla w_k \rangle
		- \frac{1}{2} \langle |\nabla\tilde{\chi}|^2 , \Div w_k \rangle,
		\quad \text{a.e. } 0 \le t \le T, \\
		\langle \hat{u} (\cdot, 0), w_k \rangle
		= \langle u_0, w_k \rangle, \quad k = 1, 2, \cdots, N.
	\end{cases}
\end{equation*}
Denote $\hat{u} = \sum\limits_{j=1}^N {\hat{\alpha}_j}(t) w_j(x)$ and
$\hat{a}(t) = (\hat{\alpha}_1(t), \cdots, \hat{\alpha}_N(t))^T$, where
$\hat{\alpha}_j(t) = \langle \hat{u}(\cdot, t), w_j \rangle$.
Using integration by parts, it follows that $\hat{a}$ satisfies
\begin{equation*}
	\begin{cases}
		\hat{a}'(t) + \mathcal{A}(t)^{-1} \mathcal{B} (t) \hat{a}(t)
		= \mathcal{A}(t)^{-1} \mathcal{F}(t), \quad 0 \le t \le T, \\
		\hat{a}(0) = (\langle u_0, w_1 \rangle, \cdots, \langle u_0, w_N \rangle)^{T}.
	\end{cases}
\end{equation*}
By the unique solvability of the ODE system \eqref{Odes2}, we conclude that
\begin{equation*}
	a(t) = \hat{a}(t) \quad \text{a.e. in }(0, T),
\end{equation*}
which implies
\begin{equation*}
	\Lambda_N(\tilde{u}) = \hat{u} \quad \text{a.e. in }(0, T).
\end{equation*}
Thus, the proof of the Claim 1 is complete.

\bigskip



\begin{thebibliography}{12}
	\bibitem{ADN64}
	S.~Agmon, A.~Douglis, and L.~Nirenberg,
	Estimates near the boundary for solutions of elliptic partial differential equations satisfying general boundary conditions. II,
	Comm. Pure Appl. Math., \textbf{17} (1964), 35--92.
	
	
	\bibitem{BH74}
	J.~P.~Bourguignon and H.~Brezis,
	Remarks on the Euler equation,
	\textit{J. Funct. Anal.}, \textbf{15} (1974), no.~4, 341--363.
	
	\bibitem{BLesgen99}
	T. Blesgen, A generalisation of Navier-Stokes equations to two-phase-flows, J. Phys. D: Appl. Phys., \textbf{32} (1999), 1119.
	
	
	\bibitem{CK03}
	H. Choe, and H. Kim, Strong solutions of the Navier-Stokes equations for isentropic compressible fluids, J. Differential Equations, \textbf{190}(2) (2003), 504--523.
	
	\bibitem{CCK04}
	Y.~Cho, H.~J.~Choe, and H.~Kim,
	Unique solvability of the initial boundary value problems for compressible viscous fluids,
	J. Math. Pures Appl, \textbf{83}(2) (2004), 243--275.
	
	\bibitem{CHS24}
	Y. Chen, H. Hong, X. Shi, Stability of the phase separation state for compressible Navier-Stokes/Allen-Cahn system, Acta Math. Appl. Sin. Engl. Ser, \textbf{40}(1) (2024) 45--74.
	
	\bibitem{CHHS21}
	Y. Chen, Q. He, B. Huang, X. Shi, Global strong solution to a thermodynamic compressible diffuse interface model with temperature dependent heat-conductivity in 1-D,
	Math. Methods Appl. Sci., \textbf{17} (2021), 12945--12962.
	
	\bibitem{CHHS25}
	Y. Chen, Q. He, B. Huang, X. Shi, The Cauchy Problem for Non-Isentropic Compressible Navier-Stokes/Allen-Cahn system with Degenerate Heat-Conductivity, Acta Math. Appl. Sin. Engl. Ser., \textbf{41} (2025), 1088--1105.
	

    \bibitem{CK06}
	Y.~Cho and H.~Kim,
	Existence results for viscous polytropic fluids with vacuum, J. Differential Equations, \textbf{228} (2006), no.~2, 377--411.

	
	\bibitem{CG17}
	M. Chen, X. Guo, Global large solutions for a coupled compressible Navier-Stokes/Allen Cahn system with initial vacuum, Nonlinear Anal. Real World Appl., \textbf{37} (2017), 350--373.
	
	\bibitem{CWZ19}
	S. Chen, H. Wen, C. Zhu, Global existence of weak solution to compressible Navier-Stokes/Allen-Cahn system in three dimensions, J. Math. Anal. Appl., \textbf{477} (2019), 1265--1295.
	
	\bibitem{CZ21}
	S. Chen, C. Zhu, Blow-up criterion and the global existence of strong/classical solutions to Navier-Stokes/Allen-Cahn system, Z. Angew. Math. Phys., \textbf{72} (1)(2021), No.~14, 24.
	
	\bibitem{DLL13}
	S. Ding, Y. Li, W. Luo, Global solutions for a coupled compressible Navier-Stokes/Allen Cahn system in 1D, J. Math. Fluid Mech., \textbf{15} (2013), no.~2, 335--360.
	
	\bibitem{DLT19}
	S. Ding, Y. Li, Y. Tang, Strong solutions to 1D compressible Navier-Stokes/Allen-Cahn system with free boundary, Math. Methods Appl. Sci., \textbf{42} (14)(2019), 4780--4794.

   	\bibitem{DLW24}	
   S. Ding, Y. Li and Y. Wang, Global solutions to 1D compressible Navier-Stokes/Allen-Cahn system with density-dependent viscosity and free-boundary, Acta Math. Sci. Ser. B (Engl. Ed.), \textbf{44} (2024), no.~1, 195--214.

	\bibitem{Evans10}
	L.~C.~Evans,
	\textit{Partial Differential Equations}, Graduate Studies in Mathematics, Vol.~19,
	American Mathematical Society, Providence, RI, 2010.

	\bibitem{FL19}
    J. Fan and F. Li, Regularity criteria for Navier-Stokes-Allen-Cahn and related systems,
    Front. Math. China, {\bf 14} (2019), no.~2, 301--314.

	\bibitem{FNG25}
    L. Fang, R. Nei and Z. Guo, Global well-posedness of the nonhomogeneous incompressible Navier-Stokes-Cahn-Hilliard system with Landau potential, J. Differential Equations, {\bf 445} (2025), Paper No.~113585, 33 pp.


    \bibitem{FDG25}
    L. Fang, X. Duan and Z. Guo, Well-posedness for a diffuse interface model of non-Newtonian two-phase flows,
    preprint, arXiv:2511.08876 (2025).

	\bibitem{F04}
	E.~Feireisl,
	\textit{Dynamics of Viscous Compressible Fluids},
	Oxford Lecture Series in Mathematics and Its Applications, Vol.~26,
	Oxford University Press, Oxford, 2004.
	
	\bibitem{FPRS10}
	E. Feireisl, H. Petzeltová, E. Rocca and G. Schimperna, Analysis of a phase-field model for two-phase compressible fluids, Math. Models Methods Appl. Sci., \textbf{20} (2010), no.~7, 1129--1160.
	
	
	
	\bibitem{GT20}
    A. Giorgini and R. Temam, Weak and strong solutions to the nonhomogeneous incompressible Navier-Stokes-Cahn-Hilliard system,
    J. Math. Pures Appl., (9) {\bf 144} (2020), 194--249.

	\bibitem{HLX12}
	X. Huang, J. Li, and Z. Xin, Global well-posedness of classical solutions with large oscillations
		and vacuum to the three-dimensional isentropic compressible Navier-Stokes equations, Comm. Pure
	Appl. Math., \textbf{65} (2012), 549--585.
	
	
	\bibitem{HWW12}
	T. Huang, C. Wang and H. Wen, Strong solutions of the compressible nematic liquid crystal flow, J. Differ. Equ., \textbf{252} (2012), 2222--2256.
	
	
	
	
	\bibitem{JWX25}
    L. Jiang, J. Wu, F. Xu, On the uniqueness of strong solution to the nonhomogeneous incompressible Navier-Stokes-Cahn-Hilliard system,
    preprint (2025), arXiv:2508.09761.


	\bibitem{JLY14}
	Q. Jiu, M. Li, and Y. Ye, Global classical solution of the Cauchy problem to 1D compressible
	Navier-Stokes equations with large initial data, J. Differential Equations, \textbf{257} (2014), 311--350.
	
	\bibitem{Kotschote12}
	M. Kotschote, Strong solutions of the Navier-Stokes equations for a compressible fluid of Allen-Cahn type, Arch. Ration. Mech. Anal., \textbf{206} (2012), no.~2, 489--514.
	
	\bibitem{LS75}
	O. A. Ladyzhenskaya, V. A. Solonnikov, On unique solvability of an initial boundary value problem for viscous incompressible nonhomogeneous liquids, Zap. Nau\v cn. Sem. Leningrad. Otdel. Mat. Inst. Steklov. (LOMI) {\bf 52} (1975), 52--109, 218--219.

	\bibitem{LS78}
	O. A. Ladyzhenskaya, V. A. Solonnikov, Unique solvability of an initial and boundary-value problem
	for viscous incompressible nonhomogeneous fluids, J. Math. Sci, \textbf{9} (1978),  697--749.

	\bibitem{Li17}
	J. Li,  Local existence and uniqueness of strong solutions to the Navier-Stokes equations with nonnegative density, J. Differential Equations, \textbf{263} (2017), 6512--6536.
	
	\bibitem{LZ23}
	J.~Li and Y.~Zheng,
	{\it Local existence and uniqueness of heat-conductive compressible Navier-Stokes equations in the presence of vacuum without initial compatibility conditions},
	J. Math. Fluid Mech., \textbf{25} (2023), no.~1,~14.
	
	\bibitem{LM68}
	J.~L.~Lions and E.~Magenes,
	\textit{Problèmes aux limites non homogènes et applications},
	Vol.~1, Grundlehren der mathematischen Wissenschaften, Springer, Berlin, 1968, 165--167.
	
	%

	\bibitem{LDH16}
    Y. Li, S. Ding and M. Huang, Blow-up criterion for an incompressible Navier-Stokes/Allen-Cahn system with different densities,
    Discrete Contin. Dyn. Syst. Ser. B, {\bf 21} (2016), no.~5, 1507--1523.

    \bibitem{LH18}
    Y. Li and M. Huang, Strong solutions for an incompressible Navier-Stokes/Allen-Cahn system with different densities,
    Z. Angew. Math. Phys., {\bf 69} (2018), no.~3, Paper No.~68, 18 pp.
	
	
    \bibitem{LX25}
    Y. Li, M. Xie, Incompressible limit of strong solutions to the diffuse interface model for two-phase flows,
    preprint, arXiv: 2503.00857v1 (2025).

	

	\bibitem{LY25}
    Y. Li, W. Ye, Well-posedness of the nonhomogeneous incompressible Navier-Stokes/Allen-Cahn system,
    preprint, arXiv: 2503.03279v1 (2025).
	


	\bibitem{LYZ18}
    T. Luo, H. Yin and C. Zhu, Stability of the rarefaction wave for a coupled compressible Navier-Stokes/Allen-Cahn system,
    Math. Methods Appl. Sci., {\bf 41} (2018), no.~12, 4724--4736.
	
	\bibitem{LYZ20}
    T. Luo, H. Yin and C. Zhu, Stability of the composite wave for compressible Navier-Stokes/Allen-Cahn system,
    Math. Models Methods Appl. Sci., {\bf 30} (2020), no.~2, 343--385.

	\bibitem{SS93}
	R. Salvi, I. Stra{\v s}kraba, Global existence for viscous compressible fluids and their behavior as $t\to \infty$, J. Fac. Sci. Univ. Tokyo Sect. IA, Math., \textbf{40} (1993), 17--51.

	\bibitem{SZW20}
    C. Song, J. Zhang and Y. Wang, Time-periodic solution to the compressible Navier-Stokes/Allen-Cahn system, Acta Math. Sin. (Engl. Ser.),
    \textbf{36} (2020), no.~4, 419--442.

	\bibitem{S21}
    M. Su, On global classical solutions to one-dimensional compressible Navier-Stokes/Allen-Cahn system with density-dependent viscosity and vacuum, Bound. Value Probl., \textbf{2021} (2021), Paper No.~92, 21 pp.
	
	\bibitem{YDL22}
	Y. Yan, S. Ding, Y. Li, Strong solutions for 1D compressible Navier-Stokes/Allen-Cahn system with phase variable dependent viscosity, J. Differential Equations, \textbf{326} (2022), 1--48.

	\bibitem{Z16}
    J. Zhang, A regularity criterion for the 3D incompressible density-dependent Navier-Stokes-Allen-Cahn equations,
    J. Partial Differ. Equ., {\bf 29} (2016), no.~2, 116--123.

    \bibitem{Z21}
    J. Zhang, Regularity of solutions to 1D compressible Navier-Stokes-Allen-Cahn system, Appl. Anal., \textbf{100} (2021), no.~9, 1827-1842.

	\bibitem{Zhao22}
	X. Zhao, Global well-posedness and decay estimates for three-dimensional compressible Navier-Stokes-Allen-Cahn systems, Proc. Roy. Soc. Edinburgh Sect. A,
	\textbf{152} (2022), no.~5, 1291--1322.
	
\end{thebibliography}
\end{document}